\documentclass[reqno,11pt]{article}
\usepackage{a4wide}
\usepackage[pdfborder={0 0 0}]{hyperref}  
\usepackage{color,eucal,enumerate,mathrsfs}
\usepackage[normalem]{ulem}
\usepackage{amsmath,amssymb,epsfig,bbm}
\numberwithin{equation}{section}

\usepackage[latin1]{inputenc}

\newcommand{\N}{\mathbb{N}}

\newcommand{\R}{\mathbb{R}}
\newcommand{\Z}{\mathbb{Z}}

\newcommand{\mm}{{\mbox{\boldmath$m$}}}

\newcommand{\smm}{{\mbox{\scriptsize\boldmath$m$}}}

\newcommand{\ggamma}{{\mbox{\boldmath$\gamma$}}}

\newcommand{\ppi}{{\mbox{\boldmath$\pi$}}}

\newcommand{\sggamma}{{\mbox{\scriptsize\boldmath$\gamma$}}}

\newcommand{\sppi}{{\mbox{\scriptsize\boldmath$\pi$}}}

\newcommand{\sfd}{{\sf d}}

\newcommand{\Kliminf}{K\kern-3pt-\kern-2pt\mathop{\rm lim\,inf}\limits}  
\newcommand{\supp}{\mathop{\rm supp}\nolimits}   
\newcommand{\Lip}{\mathop{\rm Lip}\nolimits}          
\renewcommand{\d}{{\mathrm d}}
\newcommand{\dt}{{\d t}}

\newcommand{\restr}[1]{\lower3pt\hbox{$|_{#1}$}}
\newcommand{\la}{{\langle}}                  
\newcommand{\ra}{{\rangle}}
\newcommand{\eps}{\varepsilon}  
\newcommand{\nchi}{{\raise.3ex\hbox{$\chi$}}}
\newcommand{\weakto}{\rightharpoonup}

\newcommand{\limi}{\varliminf}
\newcommand{\lims}{\varlimsup}

\newcommand{\fr}{\penalty-20\null\hfill$\blacksquare$}                      

\newcommand{\gopt}{{\rm{OptGeo}}}                   
\newcommand{\prob}[1]{\mathscr P(#1)}                   
\newcommand{\probp}[2]{\mathscr P_{#2}(#1)}                   
\newcommand{\probt}[1]{\mathscr P_2(#1)}                   
\newcommand{\e}{{\rm{e}}}                           
\newcommand{\geo}{{\rm Geo}}                      

\renewcommand{\mm}{\mathfrak m}                                
\renewcommand{\smm}{{\mbox{\scriptsize$\mm$}}}

\newcommand{\weakgrad}[1]{|\nabla #1|_w} 

\newenvironment{proof}{\removelastskip\par\medskip   
\noindent{\em proof} \rm}{\penalty-20\null\hfill$\square$\par\medbreak}

\newtheorem{theorem}{Theorem}[section]

\newtheorem{corollary}[theorem]{Corollary}
\newtheorem{lemma}[theorem]{Lemma}
\newtheorem{proposition}[theorem]{Proposition}

\newtheorem{definition}[theorem]{Definition}

\newtheorem{remark}[theorem]{Remark}

\newcommand{\bd}{{\mathbf\Delta}}
\newcommand{\rp}{{\rm p}}
 
\newcommand{\s}{{\rm S}} 
\renewcommand{\c}{{\rm Ch}} 
\newcommand{\mmz}{\tilde\mm} 
 
\newcommand{\ctang}[1]{{\rm CoTan}^{#1}}
\newcommand{\tang}[1]{{\rm Tan}^{#1}}

\newcommand{\test}[1]{{\rm Test}(#1)}
\newcommand{\Int}[1]{{\rm Int}(#1)}

\newcommand{\heatl}{{\sf H}}
\newcommand{\heatw}{{\mathscr H}}
\newcommand{\CD}{{\sf CD}}
\newcommand{\RCD}{{\sf RCD}}
\newcommand{\lip}{{\rm lip}}

\renewcommand{\u}{\mathcal U}
\renewcommand{\weakgrad}[1]{|D #1|_w}

\title{On the differential structure of metric measure spaces and applications}
\begin{document}
\author{
   Nicola Gigli\
   \thanks{Universit\'e de Nice. email: \textsf{nicola.gigli@unice.fr}}
   }

\maketitle

\begin{abstract}
The main goals of this paper are:
\begin{itemize}
\item[i)] To develop an abstract differential calculus on metric measure spaces by investigating the duality relations between differentials and gradients of  Sobolev functions. This will be achieved without calling into play any sort of analysis in charts, our assumptions being: the metric space is complete and separable and the measure is Radon and non-negative.
\item[ii)] To  employ these notions of calculus to provide, via integration by parts, a general definition of distributional Laplacian, thus giving a meaning to an expression like $\Delta g=\mu$, where $g$ is a function and $\mu$ is a measure.
\item[iii)] To  show that on spaces with Ricci curvature bounded from below and dimension bounded from above, the Laplacian of the distance function is always a measure and that this measure has the standard sharp comparison properties. This result requires an additional assumption on the space, which reduces to strict convexity of the norm in the case of smooth Finsler structures and is always satisfied on spaces with linear Laplacian, a situation which is analyzed in detail.

\end{itemize}

\end{abstract}

\tableofcontents

\section{Introduction}
The original motivation behind the project which lead to this paper, was to understand if and up to what extent the standard Laplacian comparison estimates for the distance function - valid on smooth Riemannian manifolds with Ricci curvature bounded from below - hold on a non-smooth setting. Recall that if $M$ is a Riemannian manifold with Ricci curvature not less than, say, 0, then for every $x_0\in M$ the function $g(x):=\sfd(x,x_0)$ satisfies
\begin{equation}
\label{eq:lapbase}
\Delta g\leq \frac{N-1}{g},\qquad\textrm{ on }M\setminus\{x_0\},
\end{equation}
where $N$ is (a bound from above on) the dimension of $M$. Similar bounds from above hold if the Ricci is bounded from below by $K\in\R$. Even in a smooth context, the distance function is not smooth, in particular is not $C^2$, therefore the inequality \eqref{eq:lapbase} must be properly interpreted. There are various ways of doing this, a possibility is to think \eqref{eq:lapbase} as an inequality between distributions, so that  its meaning is: the distributional Laplacian of $g$ is a measure whose  singular part w.r.t. the volume measure is non-positive on $M\setminus\{x_0\}$ and such that the density of the absolutely continuous part is bounded from above by $\frac{N-1}{g}$.

In the seminal papers \cite{Lott-Villani09} and \cite{Sturm06I}, \cite{Sturm06II} Lott-Villani on one side and Sturm on the other proposed a general definition of metric measure space with Ricci curvature bounded from below by $K\in \R$ and dimension bounded from above by $N\in(1,\infty]$  (for finite $N$'s, in \cite{Lott-Villani09} only the case $K=0$ was considered), in short, these are called $\CD(K,N)$ spaces. The study of the properties of these spaces is still a very active research area. Broadly speaking, a central question is `which of the properties valid on a manifold with Ricci curvature bounded below and dimension bounded above are true in the non-smooth context?'. 

As said, this paper addresses the problem of the Laplacian comparison for the distance function. Given that we expect such Laplacian to be a measure, part of the work that will be carried out here is to give a meaning, in the setting of metric measure spaces, to expressions like $\Delta g=\mu$, where $g$ is a function and $\mu$ a measure. This is a non-trivial task because the standard analytic tools available in this setting are only sufficient to speak about `the norm of the distributional gradient' of a Sobolev function (sometimes called minimal generalized upper gradient \cite{Cheeger00}, or minimal weak upper gradient \cite{KoskelaMacManus98}, \cite{Shanmugalingam00}), which is not enough to give a meaning to an integration by parts formula. We refer to \cite{Heinonen07}  for an overview on the subject and detailed bibliography (see also \cite{Ambrosio-Gigli-Savare-compact} for recent progresses). In terms of defining the differential of a function, in the seminal paper \cite{Cheeger00}, Cheeger, assuming a doubling condition on the measure and the validity of a local weak Poincar\'e inequality, proved the existence of an abstract notion of differential for which the Leibniz rule holds (for a comparison of his approach with ours, see Remark \ref{re:cheeger}). A different approach has been proposed by Weaver \cite{Weaver01} using the notion of $L^\infty$-module. Yet, these approaches are not sufficient to define integration by parts because the duality relations between differentials and gradients has been not investigated. This is the focus of the current  work.

The distinction between differentials and gradients is particularly important in this framework, because $\CD(K,N)$ spaces include Finsler structures (as proved by Ohta in \cite{OhtaF} by generalizing the analogous result valid for flat normed spaces proved by Cordero-Erasquin, Sturm and Villani - see the last theorem in \cite{Villani09}), so we must be ready to deal with situations where tangent and cotangent vectors - whatever they are - can't be identified. Recall also that on a Finsler setting  the Laplacian is \emph{not} a linear operator (see e.g. \cite{Shen}). In particular, since the construction we propose reduces to the standard one on Finsler spaces, the Laplacian we define is different from the linear, chart dependent Laplacian introduced in \cite{Cheeger00}.

\bigskip

Shortly said, what we will do is to define the $\mm$-a.e. value of $Df(\nabla g)$, i.e. the differential of $f$ applied to the gradient of $g$, for Sobolev functions $f,g$. Then the definition of distributional Laplacian $\Delta g=\mu$ will be given by
\[
\int f\,\d \mu=-\int Df(\nabla g)\,\d\mm,
\]
for any $f$ in an appropriate class of Lipschitz test functions. In order to illustrate the construction we will use, we open a parenthesis and show how it works in a flat normed situation. This will serve as guideline - intended for non-specialists - for the rest of the introduction and provide a quick and simple reference for the procedures that will be used in the paper. Also, it will hopefully make clear why the distributional Laplacian, in general, is not only not linear but also multivalued.

\subsection{The simple case of normed spaces}\label{se:normato}
Consider the space  $(\R^d,\|\cdot\|,\mathcal L^d)$, where $\|\cdot\|$ is any norm, possibly not coming from a scalar product. Assume for the moment that $\|\cdot\|$ is smooth and strictly convex. Given a smooth function $f:\R^d\to\R$, its differential $Df$ is intrinsically well defined as a cotangent vector field, so that we may think at it as a map from $\R^d$ to its dual $(\R^d)^*$.  To define the gradient $\nabla f$, which is a tangent vector field, we need some duality map from the cotangent to the tangent space. This is produced by using the norm $\|\cdot\|$: the differential of the map $\frac{\|\cdot\|^2}{2}:\R^d\to\R$ at some point $v\in\R^d$ is a linear map from $T_v\R^d\sim\R^d$ to $\R$, i.e. a cotangent vector. In other words, the differential of $\frac{\|\cdot\|^2}{2}$ is a map ${\rm Dual}:\R^d\to(\R^d)^*$. If $\|\cdot\|$ is strictly convex, as we are assuming for the moment, ${\rm Dual}$ is injective and can therefore be inverted to get a map ${\rm Dual}^{-1}:(\R^d)^*\to\R^d$. Then the gradient $\nabla g:\R^d\to\R^d$ of the smooth function $g$ is defined by
\[
\nabla g:={\rm Dual}^{-1}(Dg).
\]
The divergence operator $\nabla\cdot$, which takes a smooth tangent vector field $\xi$ and returns a function, is defined as (the opposite of) the adjoint of the differential operator, i.e.:
\[
\int f\,\nabla\cdot\xi\,\d\mathcal L^d:=-\int Df(\xi)\,\d\mathcal L^d,\qquad\forall f\in C^\infty_c(\R^d).
\]
In particular, in order to define the divergence, we need to use the reference measure.

Finally, the Laplacian of the smooth function $g$ is defined as
\[
\Delta g:=\nabla\cdot(\nabla g)=\nabla\cdot({\rm Dual}^{-1}(Dg)),
\]
(in particular,  in order to define $\Delta$ we had to use  both the norm and the reference measure. This is an heuristic explanation of why the correct abstract setting where a Laplacian can be defined is that of metric measure spaces). Notice that the differential and the divergence operators are always linear functionals, but the duality map ${\rm Dual}^{-1}$ is linear if and only if the norm comes from a scalar product. Hence, in 
the general normed situation, the Laplacian is non-linear.

Now suppose that we want to  give a meaning to a formula like $\Delta g=\mu$, where $\mu$ is some measure on $\R^d$, i.e., suppose that we want to provide a notion of distributional Laplacian.  An effect of the non-linearity of $\Delta$, is that there is no hope to `drop all the derivatives on the test function', that is, we can't hope that a formula like
\[
\int f\,\d\Delta g:=\int \Delta^{\rm tr} f\, g\,\d\mathcal L^d,\qquad\forall f\in C^\infty_c(\R^d),
\]
holds, whatever the operator $\Delta^{\rm tr}$ is. Simply because the right hand side would be linear in $g$, while the left one is not. 

Hence, the best we can do is to drop only one derivative on the test function and define
\begin{equation}
\label{eq:defsmooth}
\int f\,\d\Delta g:=-\int Df(\nabla g)\,\d\mathcal L^d,\qquad\forall f\in C^\infty_c(\R^d).
\end{equation}
Observe that  an effect of this approach, is that in order to define the distributional Laplacian, we need to first impose some Sobolev regularity on $g$ in order to be sure that $\nabla g$ is well defined. This is of course different from the standard approach in the Euclidean setting, where distributional derivatives make sense as soon as the function is in $L^1_{\rm loc}$. Yet, a posteriori this is not a big issue, see Remark \ref{re:apriori} for comments in this direction.

\bigskip

In proposing definition \eqref{eq:defsmooth}, which is the standard one in Finsler geometries, we worked under the assumption that $\|\cdot\|$ was smooth and strictly convex, so that the gradient of a Sobolev function was a well defined vector field. Actually, we didn't really use the smoothness assumption: if this fails but the norm is still strictly convex, then the duality map ${\rm Dual}$ is possibly multivalued, but still injective, so that its inverse ${\rm Dual}^{-1}$ is still uniquely defined. A new behavior comes if we drop the assumption of strict convexity:  in this case the map ${\rm Dual}$ is not injective anymore, which means that its inverse ${\rm Dual}^{-1}$ is multivalued. Hence the gradient $\nabla g$ of a function is typically a multivalued vector field. A simple example which shows this behavior is given by the space $(\R^2,\|\cdot\|_{\infty})$ and the linear map $g(x_1,x_2):=x_1$: all the vectors of the kind $(1,x_2)$, with $x_2\in[-1,1]$, have the right to be called gradient of $g$. 

Thus in general, the object $Df(\nabla g)$ is not single valued, not even for smooth $f,g$, so that the best we can do is  to define the two single valued maps $D^+f(\nabla g)$, $D^-f(\nabla g)$ by
\begin{equation}
\label{eq:def1}
D^+f(\nabla g)(x):=\max_{v\in \nabla g(x)}Df(v),\qquad\qquad D^-f(\nabla g)(x):=\min_{v\in \nabla g(x)}Df(v),
\end{equation}
and then the distributional Laplacian of $g$ by the chain of inequalities
\begin{equation}
\label{eq:lapfin}
-\int D^+f(\nabla g)\,\d\mathcal L^d\leq \int f\,\d\Delta g\leq -\int D^-f(\nabla g)\,\d\mathcal L^d,\qquad\forall f\in C^\infty_c(\R^d).
\end{equation}
A clear drawback of this approach is that in general the Laplacian can be multivalued (see Proposition \ref{prop:comp} for explicit examples).  Yet, even for this potentially multivalued notion, a reasonable calculus can be developed, see Section \ref{se:lapcalc}.

\bigskip

Both for theoretical and practical reasons, it is good to have an alternative description of the objects $D^\pm f(\nabla g)$.  Here it comes a key point: it holds
\begin{align}
\label{eq:def2}
D^+f(\nabla g)&=\inf_{\eps >0}\frac{\| D(g+\eps f)\|^2_*-\|D g\|_*^2}{2\eps},\\
\label{eq:def3}
D^-f(\nabla g)&=\sup_{\eps <0}\frac{\|D(g+\eps f)\|^2_*-\|D g\|_*^2}{2\eps}.
\end{align}
To see why, observe that for any $g\in C^\infty_c(\R^d)$, any $x\in\R^d$ and any tangent vector $v\in T_x\R^d\sim\R^d$ it holds
\[
Dg(v)\leq\frac{\|Dg(x)\|_*^2}{2}+\frac{\|v\|^2}2,
\]
and that $v\in\nabla g(x)$ if and only if
\[
Dg(v)\geq\frac{\|Dg(x)\|_*^2}{2}+\frac{\|v\|^2}2.
\]
Write the first inequality with $g$ replaced by $g+\eps f$, subtract the second and divide by $\eps>0$ (resp. $\eps<0$) to get that $\leq$ holds in \eqref{eq:def2} (resp. $\geq$ in \eqref{eq:def3}). The fact that equality holds can be deduced with a compactness argument from the convexity of $\eps\mapsto \|D(g+\eps f)\|_*^2(x)$,  by considering $v_\eps\in \nabla (g+\eps f)$ and letting $\eps\downarrow 0$ (resp. $\eps\uparrow 0$).
\subsection{The general situation}
Let us now analyze how to deal with the case of metric measure spaces $(X,\sfd,\mm)$. Recall that our assumptions are: $(X,\sfd)$ is complete and separable and $\mm$ is a Radon non-negative measure. There is a standard notion concerning the meaning of `$f:X\to\R$ has a $p$-integrable $p$-weak upper gradient', $p\in(1,\infty)$, one of the various equivalent ways (see \cite{Ambrosio-Gigli-Savare-pq} and Appendix \ref{app:sob})  of formulating it is: there exists $G\in L^p(X,\mm)$ such that
\begin{equation}
\label{eq:introsob}
\int|f(\gamma_1)-f(\gamma_0)|\,\d\ppi(\gamma)\leq\iint_0^1G(\gamma_t)|\dot\gamma_t|\,\d t\,\d\ppi(\gamma),
\end{equation}
for any $\ppi\in\prob{C([0,1],X)}$ such that:
\begin{itemize}
\item $\displaystyle{\iint_0^1|\dot\gamma_t|^q\,\d t\,\d\ppi(\gamma)<\infty}$, where $q$ is the conjugate exponent of $p$,
\item  $(\e_t)_\sharp\ppi\leq C\mm$ for any $t\in[0,1]$ and some $C\in\R$, where $\e_t:C([0,1],X)\to X$ is the evaluation map defined by $\e_t(\gamma):=\gamma_t$ for any $\gamma\in C([0,1],X)$.
\end{itemize} 
The class of such $f$'s will be called $p$-Sobolev class and denoted by $\s^p(X,\sfd,\mm)$. Notice that we are not asking for any integrability of the functions themselves.  For $f\in\s^p(X,\sfd,\mm)$ there exists a minimal $G\geq 0$ satisfying \eqref{eq:introsob}, which we will denote by $\weakgrad f$ ($\weakgrad f$ is typically called minimal $p$-weak/generalized upper gradient, but from the perspective we are trying to propose, this object is closer to a differential than to a gradient, given that it acts in duality with derivative of curves, whence the notation used). For the potential dependence of $\weakgrad f$ on $p$, see  Remark \ref{re:notazione}.

\bigskip

\noindent{\bf Chapter \ref{se:diff}} The real valued map $\weakgrad f$ is what plays the role of the norm of the distributional differential for Sobolev functions $f\in\s^p(X,\sfd,\mm)$. Hence by analogy with \eqref{eq:def2}, \eqref{eq:def3}, for $f,g\in\s^p(X,\sfd,\mm)$ we define $\mm$-a.e. the maps $D^\pm f(\nabla g):X\to\R$ by:
\begin{equation}
\label{eq:def5}
\begin{split}
D^+f(\nabla g)&:=\inf_{\eps>0}\frac{\weakgrad{(g+\eps f)}^2-\weakgrad g^2}{2\eps}=\inf_{\eps>0}\frac{\weakgrad{(g+\eps f)}^p-\weakgrad g^p}{p\eps\weakgrad g^{p-2}},\\
D^-f(\nabla g)&:=\sup_{\eps<0}\frac{\weakgrad{(g+\eps f)}^2-\weakgrad g^2}{2\eps}=\sup_{\eps<0}\frac{\weakgrad{(g+\eps f)}^p-\weakgrad g^p}{p\eps\weakgrad g^{p-2}}.
\end{split}
\end{equation}
Since in general $D^+f(\nabla g)\neq D^-f(\nabla g)$, for given $g$ the maps  $f\mapsto D^\pm f(\nabla g)$ are typically not linear, yet some structural properties are retained even in this generality, because $f\mapsto D^+ f(\nabla g)$ is pointwise convex, and  $f\mapsto D^- f(\nabla g)$ pointwise concave.  For given $f$, the maps $g\mapsto D^\pm f(\nabla g)$ have some very general semicontinuity properties, see Proposition \ref{prop:convconc}. We will call spaces such that $D^+f(\nabla g)=D^-f(\nabla g)$ $\mm$-a.e. for any $f,g\in\s^p(X,\sfd,\mm)$, $q$-infinitesimally strictly convex spaces: in normed spaces, this notion reduces to the strict convexity of the norm. On such spaces, the common value of $D^+f(\nabla g)$ and $D^- f(\nabla g)$ will be denoted by $Df(\nabla g)$.

From the definition \eqref{eq:def5} it is not hard to prove that $D^\pm f(\nabla g)$ satisfies natural chain rules. The Leibniz rule could then be proved as a consequence of them, see Remark \ref{re:horver}. Yet we will proceed in a different way, as this will give us the opportunity to introduce important concepts which play a key role in the proof of the Laplacian comparison: we will define the notion of gradient vector field.  The idea is the following. On general metric measure spaces we certainly do not have tangent vectors. But we have curves. And we can think at a tangent vector as the derivation at time 0 along a sufficiently smooth curve. Even on $\R^d$, in order to get this interpretation it is not sufficient to deal with Lipschitz curves, because nothing ensures that their derivatives at 0 exist. In order to somehow enforce this derivative to exist in a non-smooth setting, we will take advantage of ideas coming from the gradient flow theory: notice that from \eqref{eq:introsob} it is not difficult to see that for any $g\in\s^p(X,\sfd,\mm)$ and $\ppi$ as in \eqref{eq:introsob} it holds
\begin{equation}
\label{eq:def6}
\lims_{t\downarrow0}\int\frac{g(\gamma_t)-g(\gamma_0)}{t}\,\d\ppi(\gamma)\leq \frac1p\int\weakgrad g^p(\gamma_0)\,\d\ppi(\gamma)+\frac1q\lims_{t\downarrow0}\frac1t\iint_0^t|\dot\gamma_t|^q\,\d t\d\ppi(\gamma).
\end{equation}
Therefore if for some $g,\ppi$ it happens that 
\begin{equation}
\label{eq:def7}
\limi_{t\downarrow0}\int\frac{g(\gamma_t)-g(\gamma_0)}{t}\,\d\ppi(\gamma)\geq \frac1p\int\weakgrad g^p(\gamma_0)\,\d\ppi(\gamma)+\frac1q\lims_{t\downarrow0}\frac1t\iint_0^t|\dot\gamma_t|^q\,\d t\d\ppi(\gamma),
\end{equation}
it means that the curves where $\ppi$ is concentrated `indicate the direction of maximal increase of $g$ at time $t=0$'. In a smooth setting, this would mean $\gamma'_0=\nabla g(\gamma_0)$ for $\ppi$-a.e. $\gamma$. We then use \eqref{eq:def7} as definition and say that if it holds, then $\ppi$ is a plan which $q$-represents $\nabla g$. Arguing exactly as at the end of  the Section \eqref{se:normato}, we get the following crucial result: for $f,g\in \s^p(X,\sfd,\mm)$ and $\ppi$ which $q$-represents $\nabla g$ it holds
\begin{equation}
\label{eq:introhv}
\begin{split}
\int D^+f(\nabla g)\weakgrad g^{p-2}\,\d\ppi(\gamma_0)&\geq \lims_{t\downarrow0}\int\frac{f(\gamma_t)-f(\gamma_0)}t\,\d\ppi(\gamma),\\
\int D^-f(\nabla g)\weakgrad g^{p-2}\,\d\ppi(\gamma_0)&\leq \limi_{t\downarrow0}\int\frac{f(\gamma_t)-f(\gamma_0)}t\,\d\ppi(\gamma).
\end{split}
\end{equation}
Notice that if the space is $q$-infinitesimally strictly convex, then \eqref{eq:introhv} tells that the limit $\lim_{t\downarrow0}\int\frac{f(\gamma_t)-f(\gamma_0)}t\,\d\ppi(\gamma)$ exists and in particular the map
\[
\qquad\s^p(X,\sfd,\mm)\ni f\qquad\mapsto\qquad  \lim_{t\downarrow0}\int\frac{f(\gamma_t)-f(\gamma_0)}t\,\d\ppi(\gamma),
\]
is well defined and linear, so that we can really think at plans $\ppi$ representing gradients as differentiation operators (see also Appendix \ref{app:cottan}).

The non-trivial task here is to prove that plans $q$-representing gradients exist. This is achieved by calling into play optimal transport as key tool: we will combine  gradient flow theory, a powerful duality principle due to Kuwada \cite{Kuwada10} and a general superposition principle due to Lisini \cite{Lisini07}, see the proof of Lemma \ref{le:planf}. Such existence result as well as part of the computations done in this chapter were already present, although in an embryonal form and for some special class of spaces, in \cite{Ambrosio-Gigli-Savare11bis}.

Coming back to the Leibniz rule, it will be easy to show - essentially as a consequence of its validity on the real line - that for $\ppi$ which $q$-represents $\nabla g$, $g\in \s^p(X,\sfd,\mm)$ and non-negative $f_1,f_2\in \s^p(X,\sfd,\mm)\cap L^\infty(X,\mm)$ it holds
\[
\begin{split}
\lims_{t\downarrow 0}\int\frac{f_1(\gamma_t)f_2(\gamma_t)-f_1(\gamma_0)f_2(\gamma_0)}{t}\,\d\ppi(\gamma)\leq&\lims_{t\downarrow 0}\int f_1(\gamma_0)\frac{f_2(\gamma_t)-f_2(\gamma_0)}{t}\,\d\ppi(\gamma)\\
&+\lims_{t\downarrow 0}\int f_2(\gamma_0)\frac{f_1(\gamma_t)-f_1(\gamma_0)}{t}\,\d\ppi(\gamma).
\end{split}
\]
From this inequality and a little bit of work,  taking advantage of \eqref{eq:introhv} it will then be possible to show that the Leibniz rule holds in the following general form:
\[
\begin{split}
 D^+(f_1f_2)(\nabla g)&\leq f_1D^{s_1}f_2(\nabla g)+f_2D^{s_2}f_1(\nabla g),\\
 D^-(f_1f_2)(\nabla g)&\geq f_1D^{-s_1}f_2(\nabla g)+f_2D^{-s_2}f_1(\nabla g),
\end{split}
\]
$\mm$-a.e., where $s_i:={\rm sign}f_i$, $i=1,2$.

\bigskip

\noindent{\bf Chapter \ref{se:lap}} Once a rigorous meaning to $D^\pm f(\nabla g)$ is given,  the definition of Laplacian can be proposed along the same lines of \eqref{eq:lapfin}. We say that $g:X\to\R$ is in the domain of the Laplacian, and write $g \in D(\bd )$, provided it belongs to $\s^p(X,\sfd,\mm)$ for some $p\in(1,\infty)$ and there exists a Radon  measure $\mu$ such that
\begin{equation}
\label{eq:introlap}
-\int D^+f(\nabla g)\,\d\mm\leq \int f\,\d\mu\leq-\int D^-f(\nabla g)\,\d\mm,
\end{equation}
holds for any $f:X\to\R$ Lipschitz, in $L^1(X,|\mu|)$ and with support bounded and of finite $\mm$-measure - see Definition \ref{def:lap}.  In this case we write $\mu\in \bd g$ (we use the bold character to emphasize  that we will always think of the Laplacian as a measure). The calculus tools developed in Chapter \ref{se:diff} allow to prove that also for this general potentially non-linear and multivalued notion of Laplacian, a reasonable calculus can be developed. For instance, the chain rule 
\[
\bd(\varphi\circ g)\supset \varphi'\circ g\,\bd g+\varphi''\circ g\weakgrad g^2\,\mm,
\]
holds under very general assumptions on $g,\varphi$  (see Propositions \ref{prop:chainlapl} and \ref{prop:chainlinear}).

A particularly important class of spaces is that of `infinitesimally Hilbertian spaces', i.e. those such that $W^{1,2}(X,\sfd,\mm)$ is an Hilbert space. In these spaces, the Laplacian is unique and linearly depends on $g$, at least for $g\in \s^2(X,\sfd,\mm)$. The key consequence of this assumption is that, as already noticed in \cite{Ambrosio-Gigli-Savare11bis}, not only the space is 2 infinitesimally strictly convex, so that $Df(\nabla g)$ is well defined for $f,g\in\s^2(X,\sfd,\mm)$, but also
\[
D f(\nabla g)=Dg(\nabla f),\qquad\mm-a.e.\qquad\qquad\forall f,g\in\s^2(X,\sfd,\mm).
\]
This identity is the non-smooth analogous of the fact that on Riemannian framework we can identify differentials and gradients via Riesz theorem.
This common value will be denoted by $\nabla f\cdot\nabla g$ to emphasize its symmetry. An important consequence of such symmetry  is that the Leibniz rule for the Laplacian
\[
\bd (g_1g_2)=g_1\bd g_2+g_2\bd g_1+2\nabla g_1\cdot\nabla g_2\mm,
\]
holds, see Theorem \ref{thm:leiblap}.

\bigskip

On a smooth setting, a completely different approach to a definition of distributional Laplacian is to look at the short time behavior of the heat flow. It is therefore natural to question whether such approach  is doable on abstract spaces and how it compares with the one we just proposed. Certainly, in order to do this, one has to deal with spaces where  a `well-behaved' heat flow exists. We will analyze this approach, and show that under pretty general assumptions it produces the same notion as the one coming from integration by parts, on spaces with Riemannian Ricci curvature bounded from below. This class of spaces, introduced in \cite{Ambrosio-Gigli-Savare11bis}, is the class of infinitesimally Hilbertian $\CD(K,\infty)$ spaces, see Definition \ref{def:rcd}.

\bigskip

\noindent{\bf Chapter \ref{se:comp}} Our task is now to prove that the sharp Laplacian comparison estimates for the distance function are true in abstract $\CD(K,N)$ spaces. The problem here is that the typical proof of these comparison results relies either on Jacobi fields calculus, or on the Bochner inequality. None of the two tools are available, at least as of today, in the non-smooth setting. 

The only information that we have is the one coming from the $\CD(K,N)$ condition. In order to illustrate how to get the Laplacian comparison out of this, let us assume for a moment that we are dealing with a smooth Finsler manifold $(F,\sfd,\mm)$ that satisfies the $\CD(0,N)$ condition and such that the norms in the tangent spaces are all strictly convex. The $\CD(0,N)$ condition tells, in this case, that for any $\mu_0,\mu_1\in\prob F$ with bounded support the unique   (constant speed and minimizing) geodesic $(\mu_t)$ on the Wasserstein space $(\probt F,W_2)$ connecting $\mu_0$ to $\mu_1$ satisfies
\begin{equation}
\label{eq:cd0n}
-\int\rho_t^{1-\frac1N}\,\d\mm\leq-(1-t)\int\rho_0^{1-\frac1N}\,\d\mm-t\int\rho_1^{1-\frac1N}\,\d\mm,\qquad\forall t\in[0,1],
\end{equation}
where $\rho_t$ is the density of the absolutely continuous part of $\mu_t$ w.r.t. $\mm$. Fix $x_0\in F$, let $\varphi(x):=\frac{\sfd^2(x,x_0)}{2}$ and let $\rho_0$ be an arbitrary smooth probability density with compact support. Put $\mu:=\rho_0\mm$ and notice that the only geodesic connecting $\mu_0:=\mu$ to $\mu_1:=\delta_{x_0}$ is given by $\mu_t:=(\exp(-t\nabla\varphi))_\sharp\mu$. Hence, letting $\rho_t$ be the density of $\mu_t$, by explicit computation we get
\begin{equation}
\label{eq:poi}
\lim_{t\downarrow0}\frac{\displaystyle{-\int\rho_t^{1-\frac1N}\,\d\mm +\int\rho_0^{1-\frac1N}\,\d\mm }}t=-\frac1N\int D\big(\rho^{1-\frac1N}\big)(\nabla\varphi)\,\d\mm.
\end{equation}
On the other hand, \eqref{eq:cd0n} gives
\begin{equation}
\label{eq:poi2}
\lims_{t\downarrow0}\frac{\displaystyle{-\int\rho_t^{1-\frac1N}\,\d\mm +\int\rho_0^{1-\frac1N}\,\d\mm }}t\leq \int \rho_0^{1-\frac1N}\,\d\mm,
\end{equation}
so that it holds
\[
\int \rho_0^{1-\frac1N}\,\d\mm\geq -\frac1N\int  D\big(\rho_0^{1-\frac1N}\big)(\nabla\varphi)\,\d\mm.
\]
Given  that $\rho_0$  is arbitrary, non-negative and was chosen independently on $\varphi$, we get $\Delta\varphi\leq N$, which - up to a chain rule - is \eqref{eq:lapbase}. For bounds on the Ricci curvature different than 0, \eqref{eq:cd0n} is modified into a distorted geodesic convexity of $\rho\mapsto -\int\rho^{1-\frac1N}\,\d\mm$, and the very same computations give the general sharp  Laplacian comparison estimate (we also remark that on a smooth setting, taking two derivatives of the internal energy instead than one leads, via the $\CD(K,N)$ condition, to the Bochner inequality - see \cite{Bochner-CD}).

The difficulty in proceeding this way  in a non-smooth setting lies in proving   that \eqref{eq:poi} holds, at least with $\geq $ replacing $=$. Notice indeed that we don't have any sort of change of variable formula. The idea to overcome this issue, which was also used in \cite{Ambrosio-Gigli-Savare11bis}, is to first notice - as in the proof of Proposition 3.36 of \cite{Lott-Villani09} - that the convexity of $z\mapsto u_N(z):=-z^{1-\frac1N}$ yields
\[
\begin{split}
\frac{\displaystyle{-\int\rho_t^{1-\frac1N}\,\d\mm +\int\rho_0^{1-\frac1N}\,\d\mm }}t&\geq\frac1t\int u_N'(\rho_0)(\rho_t-\rho_0)\,\d\mm=\frac1t\int u_N'(\rho_0)\circ \e_t-u_N'(\rho_0) \circ \e_0\,\d\ppi,
\end{split}
\]
where $\ppi\in\prob{C([0,1],X}$ is such that $(\e_t)_\sharp\ppi=\mu_t$ for any $t\in[0,1]$ and $\iint_0^1|\dot\gamma_t|^2\,\d t\,\d\ppi(\gamma)=W_2^2(\mu_0,\mu_1)$. Then, taking advantage of Cheeger's results in \cite{Cheeger00}, one can see that $\ppi$ 2-represents $\nabla(-\varphi)$ (see Corollary \ref{cor:bm} for the rigorous statement), hence the limit as $t\downarrow0$ of the last term in the previous expression can be estimated using \eqref{eq:introhv} to get
\[
\begin{split}
\limi_{t\downarrow0}\frac1t\int u_N'(\rho_0)\circ \e_t-u_N'(\rho_0) \circ \e_0\,\d\ppi&\geq \int D^-(u_N'(\rho))(\nabla(-\varphi))\rho_0\,\d\mm\\
&=-\frac1N\int D^+\big(\rho_0^{1-\frac1N}\big)(\nabla\varphi)\,\d\mm.
\end{split}
\]
Using \eqref{eq:poi2}, which remains valid in the abstract setting, we then get
\begin{equation}
\label{eq:poi3}
-\int D^+\big(\rho_0^{1-\frac1N}\big)(\nabla\varphi)\,\d\mm\leq N\int \rho_0^{1-\frac1N}\,\d\mm,
\end{equation}
this inequality being true for any Lipschitz probability density $\rho_0$ with compact support (see the proof of Theorem \ref{thm:controllo} for the rigorous justification). Notice that from \eqref{eq:poi3} and the definition \eqref{eq:introlap} we could deduce that $\bd \varphi\leq N\mm$ if there was $D^-\big(\rho_0^{1-\frac1N}\big)$ in place of $D^+\big(\rho_0^{1-\frac1N}\big)$ in the left hand side. This means that at least for spaces which are infinitesimally strictly convex for some $q\in(1,\infty)$, where  $ D^+\big(\rho_0^{1-\frac1N}\big)(\nabla\varphi)= D^-\big(\rho_0^{1-\frac1N}\big)(\nabla\varphi)$,  we can conclude that the Laplacian comparison holds, in line with what is known in the smooth Finsler setting (see \cite{Sturm-Ohta-CPAM}).

\bigskip

\emph{I wish to warmly thank L. Ambrosio, C. De Lellis, G. Savar\'e and K.-T. Sturm for valuable conversations I had with them while writing this work.}

\emph{In preparing this paper, I was partially supported by ERC AdG GeMeThNES.}

\section{Preliminaries}
\subsection{Metric spaces and Wasserstein distance}
All the metric spaces $(X,\sfd)$ considered will be complete and separable.  The open ball of center $x\in X$ and radius $r>0$ is denoted by $B_r(x)$. $C([0,1],X)$ is the space of continuous curves from $[0,1]$ to $X$ equipped with the $\sup$ norm, which is complete and separable. For $t\in [0,1]$, the evaluation map $\e_t:C([0,1],X)\to X$ is given by
\[
\e_t(\gamma):=\gamma_t,\qquad\forall\gamma\in C([0,1],X).
\]
For $t,s\in[0,1]$, the restriction operator ${\rm restr}_t^s:C([0,1],X)\to C([0,1],X)$ is defined by
\begin{equation}
\label{eq:restr}
({\rm restr}_t^s(\gamma))_r:=\gamma_{t+r(s-t)},\qquad\forall\gamma\in C([0,1],X).
\end{equation}
So that for $t<s$, ${\rm restr}_t^s$ restricts the curve $\gamma$ to the interval $[t,s]$ and then `stretches' it on the whole $[0,1]$, for $t>s$ also a change in the orientation occurs.

A curve $\gamma\in C([0,1],X)$ is said absolutely continuous provided there exists $f\in L^1([0,1])$ such that
\begin{equation}
\label{eq:acc}
\sfd(\gamma_t,\gamma_s)\leq\int_t^s  f(r)\,\d r,\qquad\forall t,s\in [0,1],\ t<s. 
\end{equation}
The set of absolutely continuous curves from $[0,1]$ to $X$ will be denoted by $AC([0,1],X)$. More generally, if the function $f$ in \eqref{eq:acc} belongs to $L^q([0,1])$, $q\in[1,\infty]$, $\gamma$ is said $q$-absolutely continuous, and $AC^q([0,1],X)$ is  the corresponding set of $q$-absolutely continuous curves. It turns out (see for instance Theorem 1.1.2 of \cite{Ambrosio-Gigli-Savare05}) that for $\gamma\in AC^q([0,1],X)$ the limit
\[
\lim_{h\to 0}\frac{\sfd(\gamma_{t+h},\gamma_t)}{|h|},
\]
exists for a.e. $t\in [0,1]$, and defines an $L^q$ function. Such function is called metric speed or metric derivative, is denoted by $|\dot\gamma_t|$ and is the minimal (in the a.e. sense) $L^q$ function which can be chosen as $f$ in the right hand side of \eqref{eq:acc}.

$(X,\sfd)$ is said to be a length space provided for any $x,y\in X$ it holds
\[
\sfd(x,y)=\inf\int_0^1 |\dot\gamma_t|\,\d s,
\]
the infimum being taken among all $\gamma\in AC([0,1],X)$ such that $\gamma_0=x$, $\gamma_1=y$.  If the infimum is always a minimum, then the space is said geodesic, and a geodesic from $x$ to $y$ is any minimizer which is parametrized by constant speed (they are sometime referred as constant speed and minimizing geodesics). Equivalently, $\gamma$ is a geodesic from $x$ to $y$ provided
\[
\sfd(\gamma_s,\gamma_t)=|s-t|\sfd(\gamma_0,\gamma_1),\qquad\forall t,s\in[0,1],\qquad\qquad\gamma_0=x,\qquad\gamma_1=y.
\]
The space of all geodesics on $X$ will be denoted by $\geo (X)$. It  is a closed subset of $C([0,1],X)$.

\bigskip

Given Borel functions $f:X\to\R$, $G:X\to[0,\infty]$ we say that $G$ is an upper gradient of $f$ provided
\[
|f(\gamma_1)-f(\gamma_0)|\leq \int_0^1 G(\gamma_t)|\dot\gamma_t|\,\d t,\qquad\forall \gamma\in AC([0,1],X).
\]
For $f:X\to\R$  the local Lipschitz constant $\lip (f):X\to[0,\infty]$  is defined by
\begin{equation}
\label{eq:loclip}
\lip(f)(x):=\lims_{y\to x}\frac{|f(y)-f(x)|}{\sfd(y,x)},\quad\textrm{ if }x\textrm{ is not isolated, }0\textrm{ otherwise}.
\end{equation}
The one sided analogous $\lip^+(f),\lip^-(f)$, defined by
\begin{equation}
\label{eq:slopes}
\lip^+(f)(x):=\lims_{y\to x}\frac{(f(y)-f(x))^+}{\sfd(y,x)},\qquad\qquad\lip^-(f)(x):=\lims_{y\to x}\frac{(f(y)-f(x))^-}{\sfd(y,x)},
\end{equation}
if $x$ is not isolated and 0 otherwise, are called ascending and descending slope respectively, where $z^+:=\max\{z,0\}$, $z^-:=\max\{-z,0\}$. It is easy to check that if $f$ is locally Lipschitz, then $\lip^\pm(f)$, $\lip(f)$ are all upper gradients of $f$.

The space of all Lipschitz real valued functions on $X$ will be denoted by ${\rm LIP}(X)$.

\bigskip

Given a Borel measure $\sigma$ on $X$, by $\supp(\sigma)$ we intend the smallest closed set where $\sigma$ is concentrated.

The set of Borel probability measures on $X$ will be denoted by $\prob X$. For $q\in[1,\infty)$, $\probp Xq\subset\prob X$ is the set of measures with finite $q$-moment, i.e. $\mu\in\probp Xq$ if $\mu\in\prob X$ and $\int\sfd^q(x,x_0)\,\d\mu(x)<\infty$ for some (and thus every) $x_0\in X$. $\probp Xq$ is endowed with the $q$-Wasserstein distance $W_q$, defined by
\[
W_q^q(\mu,\nu):=\inf_{\sggamma}\int \sfd^q(x,y)\,\d\ggamma(x,y),
\]
where the $\inf$ is taken among all transport plans $\ggamma$, i.e. among all $\ggamma\in\prob{X\times X}$ such that $\pi^1_\sharp\ggamma=\mu$, $\pi^2_\sharp\ggamma=\nu$.

We recall the following superposition result, proved by Lisini in \cite{Lisini07} as a generalization of an analogous result proved in \cite{Ambrosio-Gigli-Savare05} in the Euclidean context.
\begin{theorem}[Lisini]\label{thm:lisini} Let $(X,\sfd)$ be a complete and separable metric space, $q\in(1,\infty)$, and  $[0,1]\ni t\mapsto \mu_t\in \probp Xq$ a $q$-absolutely continuous curve w.r.t. $W_q$.  Then there exists a measure $\ppi\in\prob{C([0,1],X)}$ concentrated on $AC^q([0,1],X)$ such that 
\begin{equation}
\label{eq:lisini}
\begin{split}
(\e_t)_\sharp\ppi&=\mu_t,\qquad\forall t\in[0,1],\\
\int|\dot\gamma_t|^q\,\d\ppi(\gamma)&=|\dot\mu_t|^q,\qquad a.e.\ t.
\end{split}
\end{equation}
\end{theorem}
We remark that for a plan $\ppi\in\prob{C([0,1],X)}$ such that $\iint_0^1|\dot\gamma_t|^q\,\d t\,\d\ppi(\gamma)<\infty$ and satisfying $(\e_t)_\sharp\ppi=\mu_t$ for any $t\in[0,1]$, it always holds $\int|\dot\gamma_t|^q\,\d\ppi(\gamma)\geq |\dot\mu_t|^q$ for a.e. $t\in[0,1]$, so that Lisini's result is of variational nature, as it selects plans with minimal energy.

\bigskip

Although we will have to work with general $q$-Wasserstein distances, we will reserve to the special case $q=2$ the wording of `$c$-concavity', `Kantorovich potential' and `optimal geodesic plan'. Thus the $c$-transform $\varphi^c:X\to\R\cup\{-\infty\}$ of a map $\varphi:X\to\R\cup\{-\infty\}$ is defined by
\[
\varphi^c(y)=\inf_{x\in X}\frac{\sfd^2(x,y)}{2}-\varphi(x).
\] 
A function $\varphi:X\to\R\cup\{-\infty\}$ is called $c$-concave provided it is not identically $-\infty$ and it holds $\varphi=\psi^c$ for some $\psi$. Notice that for a generic function $\varphi:X\to\R\cup\{-\infty\}$ it always holds $\varphi^{ccc}=\varphi^c$ and
\begin{equation}
\label{eq:cc}
\varphi^{cc}\geq\varphi,
\end{equation}
thus $\varphi$ is $c$-concave if and only if $\varphi^{cc}\leq \varphi$. The $c$-superdifferential $\partial^c\varphi$ of the $c$-concave map $\varphi$ is the subset of $X^2$ defined by
\[
(x,y)\in\partial^c\varphi\qquad\Leftrightarrow\qquad\varphi(x)+\varphi^c(y)=\frac{\sfd^2(x,y)}{2}.
\]
The set $\partial^c\varphi(x)\subset X$ is the set of those $y$'s such that $(x,y)\in\partial^c\varphi$.

Given $\mu,\nu\in\probt X$, a Kantorovich potential from $\mu$ to $\nu$ is any $c$-concave map $\varphi$ such that
\[
\frac12W_2^2(\mu,\nu)=\int\varphi\,\d\mu+\int\varphi^c\,\d\nu.
\]
The set $\gopt(\mu,\nu)\subset\prob{\geo(X)}$ is the set of those $\ppi$'s such that
\[
W_2\big((\e_s)_\sharp\ppi,(\e_t)_\sharp\ppi\big)=|s-t|W_2\big((\e_1)_\sharp\ppi,(\e_0)_\sharp\ppi\big),\qquad\forall t,s\in[0,1],\qquad(\e_0)_\sharp\ppi=\mu,\quad(\e_1)_\sharp\ppi=\nu.
\]
Recall that $(\mu_t)\subset\probt X$ is a constant speed geodesic if and only if for some $\ppi\in\gopt(\mu_0,\mu_1)$ it holds
\[
(e_t)_\sharp\ppi=\mu_t,\qquad\forall t\in[0,1].
\]
We conclude with the following simple result, whose proof is extracted from \cite{Figalli-Gigli11}. We remind that $(X,\sfd)$ is said proper, provided bounded closed sets are compact.
\begin{proposition}\label{prop:fg}
Let $(X,\sfd)$ be a proper geodesic space and $\varphi:X\to\R$ a locally Lipschitz $c$-concave function. Then for every $x\in X$ the set $\partial^c\varphi(x)$ is non-empty, and $\cup_{x\in K}\partial^c\varphi(x)$ is compact for any compact set $K\subset X$.
\end{proposition}
\begin{proof}
Fix $x\in X$ and notice that since $\varphi$ is $c$-concave, there exists a sequence $(y^n)\subset X$ such that
\[
\begin{split}
\varphi(x)&=\lim_{n\to\infty}\frac{\sfd^2(x,y^n)}{2}-\varphi^c(y^n),\\
\varphi(z)&\leq \frac{\sfd^2(z,y^n)}2-\varphi^c(y^n),\qquad\forall n\in\N,\ z\in X.
\end{split}
\]
Now let $\gamma^n:[0,\sfd(x,y^n)]\to X$ be a unit speed geodesic connecting $x$ to $y^n$, choose $z=\gamma^n_1$ in the inequality above and subtract $\varphi(x)$ to get
\[
\limi_{n\to\infty}\varphi(\gamma^n_1)-\varphi(x)\leq\limi_{n\to\infty} \frac{\sfd^2(\gamma^n_1,y^n)}{2}-\frac{\sfd^2(x,y^n)}{2}=\limi_{n\to\infty}-\sfd(x,y^n)+\frac12.
\]
The set $\{\gamma^n_1\}_n$ is bounded by construction, hence, since $\varphi$ is Lipschitz on bounded sets, the left hand side of the inequality is finite, which forces $\lims_{n\to\infty}\sfd(x,y_n)<\infty$. Thus, since 
the space is proper, the sequence $(y^n)$ converges, up to pass to subsequences, to some $y\in X$, which can be easily seen to belong to $\partial^c\varphi(x)$. 

This argument also gives the quantitative estimate $\sfd(x,y)\leq \Lip(\varphi\restr{B_1(x)})+\frac12$ for $y\in\partial^c\varphi(x)$. Therefore for a compact $K\subset X$, the set $\cup_{x\in K}\partial^c\varphi(x)$ is bounded. Since it can be easily checked to be closed, the proof is finished.
\end{proof}
\subsection{Metric measure spaces}
The main object of investigation in this paper are metric measure spaces $(X,\sfd,\mm)$, which will always be of the following kind 
\begin{equation}
\label{eq:mms}
\begin{split}
(X,\sfd)&\textrm{ is a complete separable metric space},\\
\mm\ \ &\textrm{ is  a Radon  non-negative measure on }X,
\end{split}
\end{equation}
where by local finiteness we mean that for every $x\in X$ there exists a neighborhood $U_x$ of $x$ such that $\mm(U_x)<\infty$.
\bigskip

From the fact that $\mm$ is locally finite and $(X,\sfd)$ separable, using the Lindel\"of property it follows that there exists a Borel probability measure $\mmz\in\prob X$ such that 
\begin{equation}
\label{eq:mmz}
\begin{split}
\mm\ll&\mmz\leq C\mm,\qquad\textrm{for some constant }C,\\
\textrm{with }& \frac{\d\mmz}{\d\mm}\textrm{ locally bounded from below by a positive constant},
\end{split}
\end{equation}
where again by locally bounded from below we intend that for any $x\in X$ there exists a neighborhood $U_x$ and a constant $c_x>0$ such that $\mm$-a.e. on $U_x$ it holds $\frac{\d\mmz}{\d\mm}\geq c_x$. It is also not difficult to see that $\mmz$ can be chosen so that
\begin{equation}
\label{eq:momp}
\int \sfd^q(x,x_0)\,\d\mmz<\infty,\qquad \forall q\geq 1,\ \textrm{ for some, and thus any, }x_0\in X.
\end{equation}
We fix once and for all such measure $\mmz$: its use will be useful  to get some compactness in the proof of Theorem \ref{thm:extest} (in particular, in Lemma \ref{le:planf}). Yet, we remark that $\mmz$ is not really part of our data, in the sense that none of our results depend on the particular measure $\mmz$ satisfying \eqref{eq:mmz} and \eqref{eq:momp}.

\subsection{Sobolev classes}\label{se:sobcla}
In this section we recall the definition of  Sobolev classes $\s^p(X,\sfd,\mm)$, which are the metric-measure analogous of the space of functions having distributional gradient in $L^p$ when our space is the Euclidean one, regardless of any integrability assumption on the functions themselves. Definitions are borrowed  from  \cite{Ambrosio-Gigli-Savare11}, \cite{Ambrosio-Gigli-Savare-pq}, but the presentation proposed here is slightly different: see Appendix \ref{app:sob} for the simple proof that the two approaches are actually the same. The basic properties of the Sobolev classes are recalled mainly without proofs, we refer to \cite{Ambrosio-Gigli-Savare-pq} for a detailed discussion (see also \cite{Ambrosio-DiMarino} for the case $p=1$).

\begin{definition}[$q$-test plan]\label{def:testplan}
Let $(X,\sfd,\mm)$ be as in \eqref{eq:mms} and $\ppi\in \prob{C([0,1],X)}$. We say that $\ppi$ has bounded compression provided there exists $C>0$ such that
\[
(\e_t)_\sharp\ppi\leq C\mm,\qquad\forall t\in[0,1].
\]
For  $q\in(1,\infty)$ we say that $\ppi$ is a $q$-test plan if it has bounded compression, is concentrated on $AC^q([0,1],X)$ and 
\[
\iint_0^1|\dot\gamma_t|^q\,\d t\,\d\ppi(\gamma)<\infty.
\]
\end{definition}
\begin{definition}[Sobolev classes]\label{def:sobcl}
Let $(X,\sfd,\mm)$ be as in \eqref{eq:mms}, $p\in(1,\infty)$ and $q$ the conjugate exponent. A Borel function $f:X\to\R$ belongs to the Sobolev class $\s^p(X,\sfd,\mm)$ (resp. $\s^p_{\rm loc}(X,\sfd,\mm)$) provided there exists a function $G\in L^p(X,\mm)$ (resp. in $L^p_{\rm loc}(X,\sfd,\mm)$) such that 
\begin{equation}
\label{eq:defsob}
\int\big| f(\gamma_1)-f(\gamma_0)\big|\,\d\ppi(\gamma)\leq \iint_0^1G(\gamma_s)|\dot\gamma_s|\,\d s\,\d\ppi(\gamma),\qquad\forall q\textrm{-test plan }\ppi.
\end{equation}
In this case, $G$ is called a $p$-weak upper gradient of $f$. 
\end{definition}
Since the class of $q$-test plans contains the one of $q'$-test plans for $q\leq q'$, we have that  $\s^p_{\rm loc}(X,\sfd,\mm)\subset \s^{p'}_{\rm loc}(X,\sfd,\mm)$ for $p\geq p'$, and that if $f\in \s^p_{\rm loc}(X,\sfd,\mm) $ and $G$ is a $p$-weak upper gradient, then $G$ is also a $p'$-weal upper gradient.

Arguing as in Section 4.5 of \cite{Ambrosio-Gigli-Savare-pq}, we get that for $f\in \s^p(X,\sfd,\mm)$ (resp. $\s^p_{\rm loc}(X,\sfd,\mm)$) there exists a minimal function $G\geq 0$, in the $\mm$-a.e. sense, in $L^p(X,\mm)$ (resp. $L^p_{\rm loc}(X,\mm)$) such that \eqref{eq:defsob} holds. We will denote such minimal function by $\weakgrad f$. In line with the terminology used in this context, we will refer to $\weakgrad f$ as the $p$-minimal upper gradient of $f$. Yet, given that such quantity is defined in duality with speed of curves, it is closer in spirit to the dual norm of the differential rather than to the norm of the gradient, whence the notation with a `$D$' in place of a `$\nabla$'.

As for the local Lipschitz constant and the slopes, we switched from the notation $|\nabla f|_w$ used in   \cite{Ambrosio-Gigli-Savare11} and \cite{Ambrosio-Gigli-Savare-pq} to $\weakgrad f$ to underline that this is a cotangent notion, a fact which plays a key role in our analysis.

\begin{remark}\label{re:notazione}{\rm Observe that in denoting the minimal $p$-weak upper `gradient' by $\weakgrad f$, we are losing the reference to the exponent $p$, which plays a role in the definition as it affects the class of test plans. A more appropriate notation would be $|D f|_{w,p}$. With this notation we have that for $f\in \s^p(X,\sfd,\mm)$ and $p'\leq p$ it holds $|Df|_{w,p'}\leq |Df|_{w,p}$ $\mm$-a.e.. It is a longstanding open problem to understand whether the other inequality holds. 

To  drop  the dependence on $p$ on the notation is certainly risky, but in all our statements we always recall at first to which Sobolev class our functions are belonging, thus indicating which `$p$' is chosen in the definition of $p$-minimal upper gradient and hopefully minimizing the risk of confusion. 

We also recall that as a consequence of the results of Cheeger \cite{Cheeger00}, we have that if the measure is doubling and the space supports a weak-local 1-$p'$ Poincar\'e inequality, then for $f\in \s^p_{\rm loc}(X,\sfd,\mm)$, $p\geq p'$, we have $|Df|_{w,p'}= |Df|_{w,p}$ $\mm$-a.e., thus at least in this case there is no dependence on the Sobolev exponent.
}\fr\end{remark}
Notice that if $\ppi$ is a $q$-test plan and $\Gamma\subset C([0,1],X)$ is a Borel set such that $\ppi(\Gamma)>0$, then also the plan $\ppi(\Gamma)^{-1}\ppi\restr\Gamma$ is a $q$-test plan. Also, for any $t,  s\in[0,1]$ the plan $({\rm restr}_t^s)_\sharp\ppi$ is a $q$-test plan as well (recall definition \eqref{eq:restr}). Then a simple localization argument yields that from \eqref{eq:defsob} it follows that
\begin{equation}
\label{eq:defsobpunt}
\forall t< s\in[0,1]\quad\textrm{it holds}\qquad |f(\gamma_s)-f(\gamma_t)|\leq \int_t^s\weakgrad f(\gamma_r)|\dot\gamma_r|\,\d r,\qquad\ppi-a.e.\ \gamma,
\end{equation}
whenever $\ppi$ is a $q$-test plan and $f\in\s^p_{\rm loc}(X,\sfd,\mm)$.

It is obvious that $\s^p(X,\sfd,\mm)$ and $\s^p_{\rm loc}(X,\sfd,\mm)$ are vector spaces and that it holds
\begin{equation}
\label{eq:convweak}
\weakgrad{(\alpha f+\beta g)}\leq |\alpha|\weakgrad f+|\beta|\weakgrad g,\qquad\mm-a.e., \quad\forall \alpha,\beta\in\R,
\end{equation}
and that the spaces $\s^p(X,\sfd,\mm)\cap L^\infty(X,\mm)$ and $\s^p_{\rm loc}(X,\sfd,\mm)\cap L^\infty_{\rm loc}(X,\mm)$ are algebras, for which it holds
\begin{equation}
\label{eq:leibweak}
\weakgrad{(fg)}\leq |f|\weakgrad g+|g|\weakgrad f,\qquad\mm-a.e..
\end{equation}
It is also possible to check that the object $\weakgrad f$ is local in the sense that  
\begin{equation}
\label{eq:nullset}
\forall f\in \s^p_{\rm loc}(X,\sfd,\mm)\textrm{ it holds }\weakgrad f=0,\quad \mm-a.e.\ \textrm{ on }f^{-1}(\mathcal N),\quad\forall \mathcal N\subset\R,\ \textrm{ s.t.  }\mathcal L^1(\mathcal N)=0,
\end{equation}
and
\begin{equation}
\label{eq:localgrad}
\weakgrad f=\weakgrad g,\qquad\mm-a.e.\ \textrm{ on }\{f=g\},\qquad\forall f,g\in\s^p_{\rm loc}(X,\sfd,\mm).
\end{equation}
Also, for $f\in\s^p(X,\sfd,\mm)$ (resp. $\s^p_{\rm loc}(X,\sfd,\mm)$) and $\varphi:\R \to\R$ Lipschitz, the function $\varphi\circ f$ belongs to $\s^p(X,\sfd,\mm)$ (resp. $\s^p_{\rm loc}(X,\sfd,\mm)$) as well and it holds
\begin{equation}
\label{eq:chaineasy}
\weakgrad{(\varphi\circ f)}=|\varphi'\circ f|\weakgrad f\qquad\mm-a.e..
\end{equation}
The object $\weakgrad f$ is local also in the following sense (see Theorem 4.19 of \cite{Ambrosio-Gigli-Savare11bis} for the case $p=2$ and Section 8.2 of \cite{Ambrosio-Gigli-Savare-pq} for the general one).
\begin{proposition}\label{prop:srestr}
Let $(X,\sfd,\mm)$ be as in \eqref{eq:mms}, $p\in(1,\infty)$ and $\Omega\subset X$ an open set. Then the following holds.
\begin{itemize}
\item[i)] For $f\in \s^p(X,\sfd,\mm)$ (resp. $f\in \s^p_{\rm loc}(X,\sfd,\mm)$), the restriction of $f$ to $\overline\Omega$ belongs to $f\in \s^p(\overline\Omega,\sfd,\mm)$ (resp. $f\in \s^p_{\rm loc}(\overline\Omega,\sfd,\mm)$) and it holds
\begin{equation}
\label{eq:ugualiristretto}
(\weakgrad f)_X=(\weakgrad f)_{\overline\Omega},\qquad\mm-a.e.\ in\   \Omega,
\end{equation}
where by $(\weakgrad f)_X$ (resp. $(\weakgrad f)_{\overline\Omega}$) we are denoting the minimal $p$-weak upper gradient of $f$ in the space $(X,\sfd,\mm)$ (resp. of the restriction of $f$ to $\overline\Omega$ in the space $(\overline \Omega,\sfd,\mm\restr{\overline\Omega})$).
\item[ii)] Viceversa, if $f\in  \s^p(\overline\Omega,\sfd,\mm)$ (resp. $ \s^p_{\rm loc}(\overline\Omega,\sfd,\mm)$) and $\supp(f)\subset\Omega$ with $\sfd(\supp(f),X\setminus\Omega)>0$, then extending $f$ to the whole $X$ by putting $f\equiv 0$ on $X\setminus\Omega$ we have $f\in \s^p(X,\sfd,\mm)$ (resp. $f\in \s^p_{\rm loc}(X,\sfd,\mm)$) and \eqref{eq:ugualiristretto} holds.
\end{itemize}
\end{proposition}
\begin{proof}
In \cite{Ambrosio-Gigli-Savare11bis}, $(ii)$ was stated under the additional assumption that $\mm(\partial\Omega)=0$ and then the conclusion \eqref{eq:ugualiristretto} was given $\mm$-a.e. on $\overline\Omega$. To check that the statement is true also in the current version, just notice that from the local finiteness of $\mm$ we can easily build an open set $\Omega'\subset\Omega$ such that $\supp(f)\subset \Omega'$, $\sfd(\Omega',X\setminus\Omega)>0$ and $\mm(\partial\Omega')=0$.
\end{proof}
We endow $\s^p(X,\sfd,\mm)$ with the seminorm
\begin{equation}
\label{eq:semisp}
\|f\|_{\s^p(X,\sfd,\smm)}:=\|\weakgrad f\|_{L^p(X,\smm)}.
\end{equation}
We don't know if (the quotient of) $\s^p(X,\sfd,\mm)$ is complete w.r.t. $\|\cdot\|_{\s^p}$.

The Sobolev space $W^{1,p}(X,\sfd,\mm)$ is defined as  $W^{1,p}(X,\sfd,\mm):=\s^p(X,\sfd,\mm)\cap L^p(X,\mm)$ endowed with the norm 
\[
\|f\|^p_{W^{1,p}({X,\sfd,\smm})}:=\|f\|^p_{L^p(X,\smm)}+\|f\|_{\s^p(X,\sfd,\smm)}^p.
\]
$W^{1,p}(X,\sfd,\mm)$ is always a Banach space (but we remark that in general $W^{1,2}(X,\sfd,\mm)$ is \emph{not} an Hilbert space). The completeness of $W^{1,p}(X,\sfd,\mm)$ is a consequence of the following result, whose proof follows from  the definitions and Mazur's lemma.
\begin{proposition}
Let $(X,\sfd,\mm)$ be as in \eqref{eq:mms}, $p\in(1,\infty)$, $(f_n)\subset\s^p(X,\sfd,\mm)$ and $(G_n)\subset L^p(X,\mm)$. Assume that $G_n$ is a $p$-weak upper gradient of $f_n$ for every $n\in\N$, that $f_n\to f$ pointwise $\mm$-a.e. and $G_n\weakto G$ in $L^p(X,\mm)$ for some functions $f,G$.

Then $f\in\s^p(X,\sfd,\mm)$ and $G$ is a $p$-weak upper gradient of $f$.
\end{proposition}
In \cite{Ambrosio-Gigli-Savare11}, \cite{Ambrosio-Gigli-Savare-pq} the following non-trivial result has been proved.
\begin{theorem}[Density in energy of Lipschitz functions]\label{thm:lipdense}
Let $(X,\sfd,\mm)$ be as in \eqref{eq:mms}, $p\in(1,\infty)$ and assume that $\mm$ is finite on bounded sets. Then for every $f\in W^{1,p}(X,\sfd,\mm)$ there exists a sequence $(f_n)\subset W^{1,p}(X,\sfd,\mm)$ of Lipschitz functions such that
\[
\begin{split}
\lim_{n\to\infty}\|f_n-f\|_{L^p(X,\smm)}&=0,\\
\lim_{n\to\infty}\|f_n\|_{\s^p(X,\sfd,\smm)}&=\lim_{n\to\infty}\||D f_n|\|_{L^p(X,\smm)}=\lim_{n\to\infty}\|\overline{|D f_n|}\|_{L^p(X,\smm)}=\|f\|_{\s^p(X,\sfd,\smm)},
\end{split}
\]
where given $g:X\to\R$, the function $\overline{|Dg|}:X\to[0,\infty]$ is 0 by definition on isolated points and 
\[
\overline{|Dg|}(x):=\inf_{r>0}\sup_{y_1\neq y_2\in B_r(x)}\frac{|g(y_1)-g(y_2)|}{\sfd(y_1,y_2)}.
\]
on the others.
\end{theorem}
\begin{proof}
The proof given in \cite{Ambrosio-Gigli-Savare-pq} for $p\in(1,\infty)$ was based also on the assumption that $\mm$ was finite. If $\mm$ is only finite on bounded sets we can argue as follows. We pick a sequence $(\nchi_n)$ of 1-Lipschitz functions with bounded support and values in $[0,1]$ such that $\nchi_n\equiv 1$ on $B_n(\overline x)$, for some fixed $\overline x\in X$. Then we fix $f\in W^{1,p}(X,\sfd,\mm)$ and notice that by the dominate convergence theorem we have $f\nchi_n\to f$ in $L^p(X,\mm)$ as $n\to\infty$, by \eqref{eq:localgrad} that $\weakgrad{(f\nchi_n)}=\weakgrad f$ $\mm$-a.e. on $B_n(\overline x)$ and from  \eqref{eq:leibweak} that $\weakgrad{(f\nchi_n)}\leq \weakgrad f+|f|$ so that $\|f\nchi_n\|_{\s^p(X,\sfd,\smm)}\to\|f\|_{\s^p(X,\sfd,\smm)}$ as $n\to\infty$ (\eqref{eq:leibweak} was stated for bounded functions, but it is true also with $f\in L^p(X,\mm)$ if, as in our case, the other function is both bounded and Lipschitz). Thus $(f\nchi_n)$ converges in energy in $W^{1,p}(X,\sfd,\mm)$ to $f$ as $n\to\infty$. Since $f\nchi_n$ is in $W^{1,p}(X,\sfd,\mm)$ and has bounded support, taking into account Proposition \ref{prop:srestr}, the fact that $\mm$ is finite on bounded sets and the result in  \cite{Ambrosio-Gigli-Savare-pq}, a diagonalization argument gives the result. 
\end{proof}
In particular, the following holds:
\begin{corollary}\label{cor:lipdense}
Let $(X,\sfd,\mm)$ be as in \eqref{eq:mms} be such that $\mm$ is finite on bounded sets and $p\in(1,\infty)$. Assume that $W^{1,p}(X,\sfd,\mm)$ is uniformly convex. Then ${\rm LIP}(X)\cap W^{1,p}(X,\sfd,\mm)$ is dense in $W^{1,p}(X,\sfd,\mm)$ w.r.t. the $W^{1,p}$ norm.
\end{corollary}

For later use, we collect here some basic duality properties between  $\s^p(X,\sfd,\mm)$ and the set of $q$-test plans. It will be useful to introduce the $q$-energies of a curve $\gamma\in C([0,1],X)$: for $q\in(1,\infty)$, and $t\in(0,1]$, the $q$-energy  $E_{q,t}:C([0,1],X)\to[0,\infty]$ is defined by
\begin{equation}
\label{eq:qenergy}
E_{q,t}(\gamma):=\left\{
\begin{array}{ll}
t\displaystyle{\sqrt[q]{\frac1t\int_0^t|\dot\gamma_s|^q\,\d s}},&\qquad\textrm{ if }{\rm restr}_0^t(\gamma)\in AC^q([0,1],X),\\
+\infty,&\qquad\textrm{ otherwise}.
\end{array}\right.
\end{equation}
\begin{proposition}\label{prop:cambiata}
Let  $(X,\sfd,\mm)$ be as in \eqref{eq:mms}, $p\in(1,\infty)$ and $q$ the conjugate exponent, $f\in\s^p(X,\sfd,\mm)$ and $\ppi$ a $q$-test plan. Then the following are true.
\begin{align}
\label{eq:curve}
\left|\frac{f(\gamma_t)-f(\gamma_0)}{E_{q,t}}\right|^p&\leq \frac1t\int_0^t\weakgrad f^p(\gamma_s)\,\d s,\qquad\ppi-a.e.\ \gamma,\ \forall t\in[0,1],\\
\label{eq:limsup}
\lims_{t\downarrow 0}\int\left|\frac{f(\gamma_t)-f(\gamma_0)}{E_{q,t}(\gamma)}\right|^p\,\d\ppi(\gamma)&\leq \int\weakgrad f^p(\gamma_0)\,\d\ppi(\gamma),\\
\label{eq:altra}
\lims_{t\downarrow 0}\int\frac{f(\gamma_t)-f(\gamma_0)}{t}\,\d\ppi(\gamma)&\leq \frac1p{\int\weakgrad f^p(\gamma_0)\,\d\ppi(\gamma)}+\lims_{t\downarrow0}\frac1{qt}\iint_0^t|\dot\gamma_s|^q\,\d s\,\d\ppi(\gamma).
\end{align}
\end{proposition}
\begin{proof} From \eqref{eq:defsobpunt} we obtain
\[
\big|f(\gamma_t)-f(\gamma_0)\big|\leq\int_0^t\weakgrad f(\gamma_s)|\dot\gamma_s|\,\d s\leq \sqrt[p]{\int_0^t\weakgrad f^p(\gamma_s)\,\d s} \sqrt[q]{\int_0^t|\dot\gamma_s|^q\,\d s},
\]
which is a restatement of \eqref{eq:curve}. Integrate \eqref{eq:curve} w.r.t. $\ppi$ to get
\[
\int\left|\frac{f(\gamma_t)-f(\gamma_0)}{E_{q,t}(\gamma)}\right|^p\,\d\ppi(\gamma)\leq\frac1t \int_0^t\int\weakgrad f^p\,\d(\e_s)_\sharp\ppi\,\d s=\int\weakgrad f^p\left(\frac1t\int_0^t\rho_s\,\d s\right)\,\d\mm,
\]
where $\rho_s$ is the density of $(\e_s)_\sharp\ppi$ w.r.t. $\mm$. Conclude observing that $\weakgrad f^p\in L^1(X,\mm)$, that $\rho_s$ weakly converges to $\rho_0$ as $s\downarrow 0$ in duality with $C_b(X)$ and that $\|\rho_s\|_\infty$ is uniformly bounded to get weak convergence of $\rho_s$ to $\rho_0$ in duality with $L^1(X,\mm)$, thus it holds
\[
\lim_{t\downarrow 0}\int\weakgrad f^p\left(\frac1t\int_0^t\rho_s\,\d s\right)\,\d\mm=\int\weakgrad f^p\,\d\mm,
\]
and \eqref{eq:limsup} is proved. For \eqref{eq:altra} just notice that for $\ppi$-a.e. $\gamma$ it holds
\begin{equation}
\label{eq:bridge}
\begin{split}
\frac{|f(\gamma_t)-f(\gamma_0)|}{t}&=\frac{|f(\gamma_t)-f(\gamma_0)|}{E_{q,t}(\gamma)}\frac{E_{q,t}(\gamma)}{t}\leq\frac{1}{p}\left|\frac{f(\gamma_t)-f(\gamma_0)}{E_{q,t}(\gamma)}\right|^p+\frac1q\left|\frac{E_{q,t}(\gamma)}{t}\right|^q,
\end{split}
\end{equation}
then integrate w.r.t. $\ppi$ and use \eqref{eq:limsup} in passing to the limit as $t\downarrow 0$.
\end{proof}
\begin{proposition}\label{prop:bridge}
Let $(X,\sfd,\mm)$ be as in \eqref{eq:mms}, $p\in(1,\infty)$ and $q$ the conjugate exponent, $f\in\s^p(X,\sfd,\mm)$ and $\ppi$ a $q$-test plan. Then
\begin{equation}
\label{eq:convlp}
\lim_{t\downarrow0}\sqrt[p]{\frac1t\int_0^t\weakgrad f^p\circ\e_s\,\d s}= \weakgrad f\circ\e_0,\qquad\textrm{ in }L^p(\ppi).
\end{equation}
Furthermore,  if the family of functions $\frac{E_{q,t}}t$ is dominated in $L^q(\ppi)$, then the family of functions $\frac{f(\gamma_t)-f(\gamma_0)}{t}$ is dominated in $L^1(\ppi)$.
\end{proposition}
\begin{proof}
To prove \eqref{eq:convlp}  it is sufficient to show that 
\begin{equation}
\label{eq:convl1}
\lim_{t\downarrow0}{\frac1t\int_0^t\weakgrad f^p\circ\e_s\,\d s}= \weakgrad f^p\circ\e_0,\qquad\textrm{ in }L^1(\ppi).
\end{equation}
To this aim, associate to a generic $h\in L^1(X,\mm)$ the functions $H_t\in L^1(\ppi)$, $t\in[0,1]$, defined by
\[
H_t:=\frac1t\int_0^t h\circ \e_s\,\d s,\qquad\qquad H_0:=h\circ\e_0,
\]
and notice that if $h$ is Lipschitz, the inequality $|H_t(\gamma)-H_0(\gamma)|\leq \Lip(h)\frac1t\int_0^t\sfd(\gamma_s,\gamma_0)\,\d s\leq \Lip(h)\int_0^t|\dot\gamma_s|\,\d s$ easily yields that $\|H_t-H_0\|_{L^1(\sppi)}\to 0$ as $t\downarrow 0$. Now recall that Lipschitz functions are dense in $L^1(X,\mm)$ and use the uniform continuity estimates
\[
\begin{split}
\int\frac1t\left|\int_0^th(\gamma_s)\,\d s\right|\,\d\ppi(\gamma)&\leq\frac1t\iint_0^t|h(\gamma_s)|\,\d s\,\d\ppi(\gamma)\leq C\|h\|_{L^1(X,\smm)},\\
 \int|h(\gamma_0)|\,\d\ppi(\gamma)&\leq C\|h\|_{L^1(X,\smm)},
\end{split}
\]
where $C>0$ is such that $(\e_t)_\sharp\ppi\leq C\mm$ for any $t\in[0,1]$, to conclude that \eqref{eq:convl1} is true by choosing $h:=\weakgrad f^p$.

For the second part of the statement, observe that \eqref{eq:bridge} and \eqref{eq:curve} give
\[
\frac{|f(\gamma_t)-f(\gamma_0)|}{t}\leq  \frac1{pt}\int_0^t\weakgrad f^p(\gamma_s)\,\d s+\frac1q\left|\frac{E_{q,t}(\gamma)}{t}\right|^q,
\]
that \eqref{eq:convl1} yields that the family of functions $\frac1{pt}\int_0^t\weakgrad f^p(\gamma_s)\,\d s$ is dominated in $L^1(\ppi)$, and use the assumption to conclude.
\end{proof}
\subsection{The Cheeger energy and its gradient flow}
In Section \ref{se:horver}, when discussing the existence of `plans representing a gradient', we will need to call into play some non-trivial link between gradient flows in  $L^2$ and Wasserstein geometry. In this short section we recall those results which we will need later on.

We start with the following statement concerning the dependence of the Sobolev classes w.r.t. the reference measure, proven in Section 8.2 of \cite{Ambrosio-Gigli-Savare-pq}.
\begin{proposition}\label{prop:inv2}
Let $(X,\sfd)$ be a complete and separable metric space and $\mm,\mm'$ two Radon  non-negative measures on it. Assume that $\mm\ll\mm'\ll\mm$ with $\frac{\d\mm'}{\d\mm}$ locally bounded from above and from below by positive constants. Then for every $p\in(1,\infty)$ the sets $\s^p_{\rm loc}(X,\sfd,\mm)$ and $\s^p_{\rm loc}(X,\sfd,\mm')$ coincide and for $f\in \s^p_{\rm loc}(X,\sfd,\mm)=\s^p_{\rm loc}(X,\sfd,\mm')$ a function $G$ is a $p$-weak upper gradient of $f$ in $\s^p_{\rm loc}(X,\sfd,\mm)$ if and only if it is so on $\s^p_{\rm loc}(X,\sfd,\mm')$.
\end{proposition}
In summary, the notion of being a Sobolev function is unchanged if we replace the reference measure $\mm$ with an equivalent  one $\mm'$ such that  $\log(\frac{\d\mm'}{\d\mm})\in L^\infty_{\rm loc}(X,\mm)$.

Now  let $(X,\sfd,\mm)$ be a space as in \eqref{eq:mms} and $\mmz\in\prob X$ our fixed measure satisfying \eqref{eq:mmz} and \eqref{eq:momp}. As a consequence of Proposition \ref{prop:inv2} above and the inequality $\mmz\leq C\mm$ for some $C>0$, we also get  that $\s^p(X,\sfd,\mm)\subset\s^p(X,\sfd,\mmz)$.

For $p\in(1,\infty)$ define the Cheeger energy functional $\c_p:L^2(X,\mm)\to[0,\infty]$   by
\begin{equation}
\label{eq:cheegerbase}
\c_p(f):=\left\{
\begin{array}{ll}
\displaystyle{\frac1p\int\weakgrad f^p\,\d\mm,}&\qquad\textrm{ if }f\in\s^p(X,\sfd,\mm)\cap L^2(X,\mm),\\
&\\
+\infty&\qquad\textrm{  otherwise}.
\end{array}
\right.
\end{equation}
And similarly $\widetilde\c_p:L^2(X,\mmz)\to[0,\infty]$ by
\begin{equation}
\label{eq:cheegertilde}
\widetilde\c_p(f):=\left\{
\begin{array}{ll}
\displaystyle{\frac1p\int\weakgrad f^p\,\d\mmz,}&\qquad\textrm{ if }f\in\s^p(X,\sfd,\mmz)\cap L^2(X,\mmz),\\
&\\
+\infty&\qquad\textrm{  otherwise}.
\end{array}
\right.
\end{equation}
The name `Cheeger energy', introduced in \cite{Ambrosio-Gigli-Savare11} for the case $p=2$, has been preferred over `Dirichlet energy', because in general it is not a Dirichlet form - not even for $p=2$, and because its first definition, which is equivalent to the one we gave, was based on a relaxation procedure similar to the one done by Cheeger in \cite{Cheeger00}. 

The properties of $\c_p$, in particular of $\c_2$, are linked to some special case of the definition of Laplacian - see Proposition \ref{prop:comp} -, while those of $\widetilde\c_p$ will be of key importance in order to get some sort of compactness in the proof of the crucial Lemma \ref{le:planf}.

Notice that $\widetilde\c_p$ is convex, lower semicontinuous and with dense domain, thus for any $f\in L^2(X,\mmz)$ there exists a unique gradient flow of $\widetilde\c_p$ starting from $f$. 

The following theorem collects the main properties of this gradient flow. In an abstract setting it has been proved firstly in \cite{GigliKuwadaOhta10} (in Alexandrov spaces and for $p=2$), generalized in \cite{Ambrosio-Gigli-Savare11} (for a class of spaces covering in particular those as in \eqref{eq:mms}, and for $p=2$) and in its current form in  \cite{Ambrosio-Gigli-Savare-pq}. The difficult part of the statement is the fifth, where the Wasserstein geometry is linked to the dissipation of $\widetilde\c_p$: the key tool which allows to get the sharp bound on $|\dot\mu_t|$ is Kuwada's lemma, appeared at  first in \cite{Kuwada10}.
\begin{theorem}\label{thm:keyex}
Let $(X,\sfd,\mm)$ be as in \eqref{eq:mms} and $\mmz\in\prob X$ as in \eqref{eq:mmz}, \eqref{eq:momp}. Also, let $p\in(1,\infty)$, $f\in L^2(X,\mmz)$ and  $(f_t)\subset L^2(X,\mmz)$ be the gradient flow of $\widetilde\c_p$ starting from $f$. Then the following hold.
\begin{itemize}
\item[i)]\underline{\rm Mass preservation} $\displaystyle{\int f_t\,\d\mmz=\int f\,\d\mmz}$ for any $t\geq 0$.
\item[ii)]\underline{\rm Maximum principle} If $f\leq C$ (resp. $f\geq c$)  $\mmz$-a.e.,  then $f_t\leq C$ (resp. $f_t\geq c$) $\mmz$-a.e. for any $t\geq 0$.
\item[iii)]\underline{\rm Entropy dissipation} If $c\leq f\leq C$ $\mmz$-a.e. and $u:[c,C]\to\R$ is a $C^2$ map, then $t\mapsto\int u(f_t)\,\d\mm$ belongs to $C^1(0,\infty)$  and it holds
\begin{equation}
\label{eq:dissipazione}
\frac{\d}{\dt}\int u(f_t)\,\d\mmz=-\int u''(f_t)\weakgrad{f_t}^p\,\d\mmz,\qquad\textrm{ for any }t>0, 
\end{equation}
\item[iv)]\underline{\rm Dissipation at $t=0$} With the same assumptions as in $(iii)$, if furthermore $f\in\s^p(X,\sfd,\mmz)$, then \eqref{eq:dissipazione} is true also for $t=0$ and it holds
\begin{equation}
\label{eq:dissipazione0}
\lim_{t\downarrow 0}\int u''(f_t)\weakgrad{f_t}^p\,\d\mmz=\int u''(f)\weakgrad{f}^p\,\d\mmz.
\end{equation}
\item[v)]\underline{\rm Control on the Wasserstein speed} Assume that $c\leq f\leq C$ for some $c,C>0$ and that $\int f\,\d\mmz=1$. Then the curve $t\mapsto\mu_t:=f_t\mmz$ (which has values in $\probp Xq$ thanks to $(i)$, $(ii)$  and \eqref{eq:momp}) is $q$-absolutely continuous w.r.t. $W_q$, where $q\in(1,\infty)$ is the conjugate exponent of $p$, and for its metric speed $|\dot\mu_t|$  it holds
\begin{equation}
\label{eq:boundsharp}
|\dot\mu_t|^q\leq\int\frac{\weakgrad{f_t}^p}{f_t^{q-1}}\,\d\mmz,\qquad a.e.\ t. 
\end{equation}
\end{itemize}
\end{theorem}

\section{Differentials and gradients}\label{se:diff}
\subsection{Definition and basic properties}
Following the approach described in Section \ref{se:normato}, our first task toward giving the definition of distributional Laplacian on a metric measure space $(X,\sfd,\mm)$ is to give a meaning to $D^\pm f(\nabla g)$ for Sobolev functions $f,g$. We will imitate the definition given in the right hand sides of  \eqref{eq:def2}, \eqref{eq:def3}. 

Start noticing that  if $\varphi:\R\to\R^+$ is a convex function, the values of 
\[
\inf_{\eps>0 0}\frac{\varphi(\eps)^p-\varphi(0)^p}{p\eps\varphi(0)^{p-2}},\qquad\qquad \sup_{\eps<0}\frac{\varphi(\eps)^p-\varphi(0)^p}{p\eps\varphi(0)^{p-2}}
\] 
are independent on $p\in(1,\infty)$, and equal to $\varphi(0)\varphi'(0^+)$, $\varphi(0)\varphi'(0^-)$ respectively, as soon as $\varphi(0)\neq 0$. The convexity also grants that the $\inf$ and $\sup$  can be substituted by  $\lim_{\eps\downarrow0}$ and $\lim_{\eps\uparrow0}$, respectively.

Fix $p\in(1,\infty)$, let $f,g\in\s^p_{\rm loc}(X,\sfd,\mm)$ and observe that \eqref{eq:convweak} ensures that the map $\eps\mapsto \weakgrad{(g+\eps f)}$ is convex in the sense that
\[
\weakgrad{(g+((1-\lambda)\eps_0+\lambda\eps_1) f)}\leq (1-\lambda)\weakgrad{(g+\eps_0f)}+\lambda\weakgrad{(g+\eps_1f)},\qquad\mm-a.e.,
\]
for any $\lambda\in[0,1]$, $\eps_0,\eps_1\in\R$.
\begin{definition}[$D^\pm f(\nabla g)$]\label{def:dfgg}
Let $(X,\sfd,\mm)$ be as in \eqref{eq:mms}, $p\in(1,\infty)$  and $f,g\in\s^p_{\rm loc}(X,\sfd,\mm)$. The functions  $D^\pm f(\nabla g):X\to\R$ are $\mm$-a.e. well defined by   
\begin{equation}
\label{eq:defdpm}
\begin{split}
D^+f(\nabla g)&:=\inf_{\eps>0}\frac{\weakgrad{(g+\eps f)}^{p}-\weakgrad g^{p}}{p\eps\weakgrad{g}^{p-2}},\\
D^-f(\nabla g)&:=\sup_{\eps<0}\frac{\weakgrad{(g+\eps f)}^{p}-\weakgrad g^{p}}{p\eps\weakgrad{g}^{p-2}},
\end{split}
\end{equation}
on  $\{x:\weakgrad g(x)\neq 0\}$, and are taken 0 by definition on $\{x:\weakgrad g(x)=0\}$, where the $\inf/\sup$ should be intended as essential-$\inf/\sup$.
\end{definition}
As for the notation $\weakgrad f$, we are omitting the explicit dependence on the Sobolev exponent $p$ in writing $D^\pm f(\nabla g)$, which affects the definition by the potential dependence on $p$ of the objects $|D(g+\eps f)|_{w,p}$. Yet, the initial discussion and Remark \ref{re:notazione} grant that at least if the space is doubling and supports a 1-$p'$ Poincar\'e inequality, then $D^\pm f(\nabla g)$ is unambiguously defined for $f,g\in\s^p_{\rm loc}(X,\sfd,\mm)$, where $p\geq p'$.

The initial discussion also ensures that  that $\inf/\sup$ can be replaced with essential limits as $\eps\downarrow0/\eps\uparrow0$  respectively.

\medskip

Throughout this paper, the expressions $D^\pm f(\nabla g)\weakgrad g^{p-2}$ on the set $\{x:\weakgrad g(x)=0\}$ will always be taken 0 by definition. In this way it always holds
\[
\begin{split}
D^+f(\nabla g)\weakgrad g^{p-2}&=\inf_{\eps>0}\frac{\weakgrad{(g+\eps f)}^p-\weakgrad g^p}{p\eps},\\
D^-f(\nabla g)\weakgrad g^{p-2}&=\sup_{\eps<0}\frac{\weakgrad{(g+\eps f)}^p-\weakgrad g^p}{p\eps},
\end{split}
\]
$\mm$-a.e. on $X$.

Notice that the inequality $\weakgrad{(g+\eps f)}\leq\weakgrad g+|\eps|\weakgrad f$ yields
\begin{equation}
\label{eq:1}
|D^\pm f(\nabla g)|\leq\weakgrad f\weakgrad g,\qquad\mm-a.e.,
\end{equation}
and in particular $D^\pm f(\nabla g)\weakgrad g^{p-2}\in L^1(X,\mm)$ for any $f,g\in\s^p(X,\sfd,\mm)$ (resp. in $L^1_{\rm loc}(X,\mm)$ for  $f,g\in\s^p_{\rm loc}(X,\sfd,\mm)$).

The convexity of $f\mapsto\weakgrad{(g+f)}^p$ gives
\begin{equation}
\label{eq:3}
D^-f(\nabla g)\leq D^+f(\nabla g),\qquad\mm-a.e.,
\end{equation}
and
\begin{equation}
\label{eq:6}
\begin{split}
D^+f(\nabla g)\weakgrad g^{p-2}&=\inf_{\eps >0}D^-f(\nabla g_\eps)\weakgrad{g_\eps}^{p-2}=\inf_{\eps >0}D^+f(\nabla g_\eps)\weakgrad{g_\eps}^{p-2},\qquad\mm-a.e.,\\
D^-f(\nabla g)\weakgrad g^{p-2}&=\sup_{\eps <0}D^+f(\nabla g_\eps)\weakgrad{g_\eps}^{p-2}=\sup_{\eps <0}D^-f(\nabla g_\eps)\weakgrad{g_\eps}^{p-2},\qquad\mm-a.e.,
\end{split}
\end{equation}
where $g_\eps:=g+\eps f$.

Also, from the definition it directly follows that
\begin{equation}
\label{eq:4}
D^+(-f)(\nabla g)=-D^-f(\nabla g)=D^+f(\nabla(-g)),\qquad\mm-a.e.,
\end{equation}
and that
\begin{equation}
\label{eq:5}
D^\pm g(\nabla g)=\weakgrad g^2,\qquad\mm-a.e..
\end{equation}
As a consequence of the locality properties \eqref{eq:nullset} and \eqref{eq:localgrad} we also have
\begin{equation}
\label{eq:nullset2}
D^\pm f(\nabla g)=0,\quad\mm-a.e.\quad\textrm{on }f^{-1}(\mathcal N)\cup g^{-1}(\mathcal N),\qquad\forall\mathcal N\subset\R\quad\textrm{ such that }\mathcal L^1(\mathcal N)=0
\end{equation}
and
\begin{equation}
\label{eq:localgrad2}
D^\pm f(\nabla g)=D^\pm \tilde f(\nabla \tilde g),\qquad\mm-a.e.\qquad\textrm{on } \{f=\tilde f\}\cap\{g=\tilde g\}. 
\end{equation}

Some further basic properties of $D^\pm f(\nabla g)$ are collected in the following proposition.
\begin{proposition}\label{prop:convconc}
Let $(X,\sfd,\mm)$ be as in \eqref{eq:mms}, $p\in(1,\infty)$ and $g\in\s^p_{\rm loc}(X,\sfd,\mm)$. Then
\[
\s^p_{\rm loc}(X,\sfd,\mm)\ni f\qquad\mapsto\qquad D^+f(\nabla g),
\]
is positively 1-homogeneous, convex in the $\mm$-a.e. sense, i.e.
\[
D^+((1-\lambda)f_0+\lambda f_1)(\nabla g)\leq (1-\lambda)D^+f_1(\nabla g)+\lambda D^+f_2(\nabla g),\qquad\mm-a.e.,
\]
for any  $f_0,f_1\in\s^p_{\rm loc}(X,\sfd,\mm)$ and $\lambda\in[0,1]$, and 1-Lipschitz in the following sense:
\begin{equation}
\label{eq:1lip}
\big|D^+f_1(\nabla g)-D^+f_2(\nabla g)\big|\leq \weakgrad{(f_1-f_2)}\weakgrad g,\qquad \mm-a.e.\ \ \forall f_1,f_2\in\s^p_{\rm loc}(X,\sfd,\mm).
\end{equation}
Similarly,
\[
\s^p_{\rm loc}(X,\sfd,\mm)\ni f\qquad\mapsto\qquad D^-f(\nabla g)
\]
is positively 1-homogeneous, concave and 1-Lipschitz.

Conversely, for any $f\in\s^p_{\rm loc}(X,\sfd,\mm)$ it holds
\[
\begin{split}
\s^p_{\rm loc}(X,\sfd,\mm)\ni g&\quad\mapsto\quad D^+f(\nabla g)\quad\textrm{ is positively 1-homogeneous and upper semicontinuous},\\
\s^p_{\rm loc}(X,\sfd,\mm)\ni g&\quad\mapsto\quad D^-f(\nabla g)\quad\textrm{ is positively 1-homogeneous and lower semicontinuous},
\end{split}
\]
where upper semicontinuity is intended as follows: if $g_n,g\in\s^p_{\rm loc}(X,\sfd,\mm)$, $n\in\N$, and for some Borel set $E\subset X$ it holds $\int_E\weakgrad f^p\,\d\mm<\infty$,  $\sup_n\int_E\weakgrad{g_n}^p\,\d\mm<\infty$ and $\int_E\weakgrad{(g_n-g)}^p\,\d\mm\to 0$, then it holds
\[
\lims_{n\to\infty}\int_{E'} D^+f(\nabla g_n)\weakgrad{ g_n}^{p-2}\,\d\mm\leq \int_{E'}D^+f(\nabla g)\weakgrad g^{p-2}\,\d\mm,\qquad \forall \textrm{ Borel sets }E'\subset E.
\]
Similarly for lower semicontinuity.
\end{proposition}
\begin{proof}
Positive 1-homogeneity in $f,g$ is obvious.

For convexity (resp. concavity) in $f$, just notice that from \eqref{eq:convweak} we get
\[
\weakgrad{g+\eps ((1-\lambda)f_0+\lambda f_1)}-\weakgrad g\leq (1-\lambda)\Big(\weakgrad{(g+\eps f_1)}-\weakgrad g\Big) +\lambda\Big( \weakgrad{(g+\eps  f_1)}-\weakgrad g\Big),
\]
so that the conclusion follows  dividing by $\eps>0$ (resp. $\eps<0$) and letting $\eps\downarrow0 $ (resp. $\eps\uparrow 0$).

The Lipschitz continuity in $f$ is a consequence of
\[
\left|\frac{\weakgrad{(g+\eps f)}-\weakgrad g}{\eps}-\frac{\weakgrad{(g+\eps \tilde f)}-\weakgrad g}\eps\right|\leq \weakgrad{(f-\tilde f)},\qquad\forall\eps\neq 0.
\]
For the semicontinuity in $g$, let $E\subset X$ be a Borel set as in the assumptions and $V\subset \s^p_{\rm loc}(X,\sfd,\mm)$ the space of $g$'s such that $\int_E \weakgrad g^p\,\d\mm<\infty$. Endow $V$ with the seminorm $\|g\|_{\s^p,E}:=\sqrt[p]{\int_E \weakgrad g^p\,\d\mm}$ and notice that for any $\eps\neq 0$ the real valued map
\[
V\ni g \qquad\mapsto\qquad\int_{E'}\frac{\weakgrad{(g+\eps f)}^p-\weakgrad g^p}{p\eps}\,\d\mm,
\]
is continuous. Hence the map
\[
V\ni g\qquad\mapsto\qquad \int_{E'}D^+f(\nabla g)\weakgrad g^{p-2}\,\d\mm=\inf_{\eps>0}\int_{E'}\frac{\weakgrad{(g+\eps f)}^p-\weakgrad g^p}{p\eps}\,\d\mm,
\] 
is upper semicontinuous.  Similarly for $D^-f(\nabla g)$.
\end{proof}
We now propose  the following definition:
\begin{definition}[$q$-infinitesimally strictly convex spaces]
Let $(X,\sfd,\mm)$ be as in \eqref{eq:mms}, $p\in(1,\infty)$ and $q$ the conjugate exponent. We say that $(X,\sfd,\mm)$ is $q$-infinitesimally strictly convex provided
\begin{equation}
\label{eq:defsc}
\int D^+f(\nabla g)\weakgrad g^{p-2}\,\d\mm=\int D^-f(\nabla g)\weakgrad g^{p-2}\,\d\mm,\qquad\forall f,g\in\s^p(X,\sfd,\mm).
\end{equation}
\end{definition}
In the case of normed spaces analyzed in Section \ref{se:normato}, $q$-infinitesimal strict convexity is equivalent to strict convexity of the norm, whatever $q$ is, which in turn is equivalent to the differentiability of the dual norm in the cotangent space. Formula \eqref{eq:defsc} describes exactly such differentiability property, whence the terminology used.  
Notice that $q$-infinitesimal strict convexity has nothing to do with the strict convexity of $W^{1,p}(X,\sfd,\mm)$ as the former is about strict convexity of the `tangent space', while the latter is about the one of the `cotangent space'.

We don't know whether there is any general relation between infinitesimal strict convexity and non-branching (not even under additional assumptions like doubling of the measure, validity of a weak local Poincar\'e inequality or bounds from below on the Ricci curvature), but observe that the former makes sense also in non-geodesic spaces.

From inequality \eqref{eq:3} we get that the integral equality \eqref{eq:defsc} is equivalent to the pointwise one:
\begin{equation}
\label{eq:sc}
D^+f(\nabla g)=D^-f(\nabla g),\qquad\mm-a.e.,\qquad\forall f,g\in \s^p(X,\sfd,\mm).
\end{equation}
Then a simple cut-off argument and the locality \eqref{eq:localgrad2} ensures that \eqref{eq:sc} is true also for $f,g\in\s^p_{\rm loc}(X,\sfd,\mm)$. Furthermore, from Remark \ref{re:notazione} we know that if the measure is doubling and the space supports a $p'$-weak local Poincar\'e inequality, then $q'$-infinitesimal strict convexity implies $q$-infinitesimal strict convexity for any $q\in(1,q')$, where $q'$ is the conjugate exponent of $p'$.

On $q$-infinitesimally strictly convex spaces, we will denote by $Df(\nabla g)$ the common value of $D^+f(\nabla g)=D^-f(\nabla g)$, for $f,g\in\s^p_{\rm loc}(X,\sfd,\mm)$. As a direct consequence of Proposition \ref{prop:convconc} we have the following corollary:
\begin{corollary}\label{cor:dfgg}
Let $(X,\sfd,\mm)$ be as in \eqref{eq:mms}, $p,q\in (1,\infty)$ be conjugate exponents. Assume that the space is $q$-infinitesimally strictly convex. Then:
\begin{itemize}
\item For any $g\in\s^p_{\rm loc}(X,\sfd,\mm)$ the map 
\[
\s^p_{\rm loc}(X,\sfd,\mm)\ni f\qquad\mapsto\qquad Df(\nabla g)
\]
is linear  $\mm$-a.e., i.e. 
\[
D(\alpha_1f_1+\alpha_2f_2)(\nabla g)=\alpha_1 Df_1(\nabla g)+\alpha_2 Df_2(\nabla g),\qquad  \mm-a.e.\ 
\]
for any $f_1,f_2\in\s^p_{\rm loc}(X,\sfd,\mm)$,  $\alpha_1,\alpha_2\in\R$.
\item For any $f\in \s^p_{\rm loc}(X,\sfd,\mm)$ the map
\[
\s^p_{\rm loc}(X,\sfd,\mm)\ni g\qquad\mapsto\qquad Df(\nabla g),
\]
is 1-homogeneous and continuous, the latter meaning that if $g_n,g\in\s^p_{\rm loc}(X,\sfd,\mm)$, $n\in\N$, and for some Borel set $E\subset X$ it holds $\sup_n\int_E\weakgrad{g_n}^p\,\d\mm<\infty$ and $\int_E\weakgrad{(g_n-g)}^p\,\d\mm\to 0$, then 
\[
\lim_{n\to\infty} \int_{E'}Df(\nabla g_n)\weakgrad{g_n}^{p-2}\,\d\mm=\int_{E'}Df(\nabla g)\weakgrad g^{p-2}\,\d\mm,\qquad\forall\textrm{ Borel sets }E'\subset  E.
\]
\end{itemize} 
\end{corollary}
\begin{remark}[Weak semicontinuity in $f$]\label{re:weak}{\rm As a direct consequence of Proposition \ref{prop:convconc} we get the following result. For $p\in(1,\infty)$ and $g\in\s^p(X,\sfd,\mm)$ the map
\[
\s^p(X,\sfd,\mm)\ni f\qquad\mapsto\qquad\int D^+ f(\nabla g)\weakgrad g^{p-2}\,\d\mm,
\]
is weakly lower semicontinuous in $\s^p(X,\sfd,\mm)$, in the sense that if $L(f_n)\to L(f)$ as $n\to\infty$ for any bounded - w.r.t. the seminorm $\|\cdot\|_{\s^p}$ - linear functional $L:\s^p(X,\sfd,\mm)\to\R$, then 
\[
\limi_{n\to\infty}\int D^+ f_n(\nabla g)\weakgrad g^{p-2}\,\d\mm\geq \int D^+ f(\nabla g)\weakgrad g^{p-2}\,\d\mm.
\]
(in other words, we are picking the quotient of $\s^p(X,\sfd,\mm)$ w.r.t. the subspace of functions $f$ with $\|f\|_{\s^p}=0$, taking its completion and then considering the usual weak convergence in the resulting Banach space. This abstract completion has some right to be called $p$-cotangent space of $(X,\sfd,\mm)$, see Appendix \ref{app:cottan} for some results in this direction). 

To prove this weak lower semicontinuity, just notice that $f\mapsto\int D^+ f(\nabla g)\weakgrad g^{p-2}\,\d\mm $ is convex and continuous.

Similarly, $f\mapsto \int D^- f(\nabla g)\weakgrad g^{p-2}\,\d\mm$ is weakly upper semicontinuous in $\s^p(X,\sfd,\mm)$.
}\fr\end{remark}

\subsection{Horizontal and vertical derivatives}\label{se:horver}
In Section \ref{se:normato} we proved that on a flat normed space, for any two smooth functions $f,g$ it holds
\begin{equation}
\label{eq:letto1}
\begin{split}
\inf_{\eps>0}\frac{\|D(g+\eps f)(x)\|^2_*-\|Dg(x)\|^2_*}{2\eps}&=\max_{v\in\nabla g(x)}Df(v),\\
\sup_{\eps<0}\frac{\|D(g+\eps f)(x)\|^2_*-\|Dg(x)\|^2_*}{2\eps}&=\min_{v\in\nabla g(x)}Df(v).
\end{split}
\end{equation}
In this section we investigate the validity of this statement in a metric setting.  More precisely, in Theorem \ref{thm:horver} we will  prove  an analogous of
\begin{equation}
\label{eq:letto}
\inf_{\eps>0}\frac{\|D(g+\eps f)(x)\|^2_*-\|Dg(x)\|^2_*}{2\eps}\geq Df(v)\geq \sup_{\eps<0}\frac{\|D(g+\eps f)(x)\|^2_*-\|Dg(x)\|^2_*}{2\eps},
\end{equation}
for any $v\in\nabla g(x)$, given that this will be sufficient for our purposes. The non-smooth analogous of \eqref{eq:letto1} will be proved in Appendix \ref{app:cottan}, see Theorem \ref{thm:horver2}. 

The ideas behind the main results of this part of the paper, namely Theorems \ref{thm:horver} and \ref{thm:extest}, already appeared in \cite{Ambrosio-Gigli-Savare11bis} as key tools both to find a formula for the differentiation of the relative entropy along a Wasserstein geodesic, and to develop some notions of  calculus in spaces with linear heat flow. In this paper, these ideas are extended and generalized, see in particular the next section for general calculus rules.

\bigskip

The definition we gave of $D^+f(\nabla g)$ and $D^-f(\nabla g)$ are what play the role of the leftmost and rightmost sides in \eqref{eq:letto}.   In order to define a metric-measure theoretic analogous of the middle term, we need to introduce the notion of \emph{plan representing the gradient of a function}. We start with the $q$-norm of a plan. Recall that the $q$-energies $E_{q,t}$ of a curve were defined in \eqref{eq:qenergy}.
\begin{definition}[$q$-norm of a plan]
For $q\in (1,\infty)$ and $\ppi\in\prob{C([0,1],X)}$ we define the $q$-norm $\|\ppi\|_q\in[0,\infty]$ of $\ppi$  by
\begin{equation}
\label{eq:qnorm}
\|\ppi\|_q^q:=\lims_{t\downarrow 0}\int\left(\frac{E_{q,t}}{t}\right)^q\,\d\ppi
\end{equation}
so that $\|\ppi\|_q^q=\lims_{t\downarrow0}\frac1t\iint_0^t|\dot\gamma_s|^q\,\d s\,\d\ppi(\gamma)$ if $({\rm restr}_0^T)_\sharp\ppi$ is concentrated on absolutely continuous curves for some $T\in(0,1]$ and $+\infty$ otherwise.
\end{definition}
Notice that $\prob{C([0,1],X)}$ is not a vector space, so that in calling $\|\cdot\|_q$ a norm we are somehow abusing the notation. Yet, this wording well fits the discussion hereafter, and  a rigorous meaning to it can be given, see Appendix \ref{app:cottan}.

If $\ppi$ is of bounded compression and $\|\ppi\|_q<\infty$, then not necessarily $\ppi$ is a $q$-test plan, because it might be that $\iint_0^1|\dot\gamma_t|^q\,\d t\,\d\ppi(\gamma)=\infty$. However, for $T>0$ small enough, $({\rm restr}_0^T)_\sharp\ppi$ is a $q$-test plan. Hence if $p,q\in(1,\infty)$ are conjugate exponents, $\ppi$ is of bounded compression with $\|\ppi\|_q<\infty$ and $g\in \s^p(X,\sfd,\mm)$,  \eqref{eq:altra} yields  that
\begin{equation}
\label{eq:dual}
\begin{split}
\lims_{t\downarrow 0}\int\frac{g(\gamma_t)-g(\gamma_0)}{t}\,\d\ppi \leq\frac{\|\weakgrad g\|_{L^p(X,(\e_0)_\sharp\sppi)}^p}p+\frac{\|\ppi\|_q^q}q.
\end{split}
\end{equation}
We then  propose the following definition.
\begin{definition}[Plans representing gradients]\label{def:represent} Let $(X,\sfd,\mm)$ be as in \eqref{eq:mms},  $p,q\in(1,\infty)$ conjugate exponents and $g\in \s^p(X,\sfd,\mm)$. We say that $\ppi\in\prob{C([0,1],X)}$ $q$-represents $\nabla g$ if   $\ppi$ is of bounded compression, $\|\ppi\|_q<\infty$ and it holds
\begin{equation}
\label{eq:degiorgi}
\limi_{t\downarrow 0}\int\frac{g(\gamma_t)-g(\gamma_0)}{t}\,\d\ppi\geq  \frac{\|\weakgrad g\|_{L^p(X,(\e_0)_\sharp\sppi)}^p}p+\frac{\|\ppi\|_q^q}q.
\end{equation}
\end{definition}
\begin{remark}\label{re:degiorgi2}{\rm
Notice that the definition we just gave is not too far from De Giorgi's original definition of gradient flow in a metric setting (see \cite{Ambrosio-Gigli-Savare05} for a modern approach to the topic). Differences are the fact that here we care only about the behavior of the 
curves in the support of $\ppi$ for $t$ close to 0, and the fact that we are considering a Sobolev notion, rather than a metric one.}\fr\end{remark}
\begin{remark}{\rm The key information encoded by a plan $\ppi$ $q$-representing $\nabla g$ is contained in the disintegration $\{\ppi_x\}_{x\in X}$ of $\ppi$ w.r.t. $\e_0$: this provides a Borel map $X\ni x\mapsto \ppi_x\in\prob{C([0,1],X)}$ which ``associates to $(\e_0)_\sharp\ppi$-a.e. $x$ a weighted set of curves which is a gradient flow at time 0 of $g$ starting from $x$''. Once this disintegration is given, for every measure $\mu$ such that $\mu\leq C(\e_0)_\sharp\ppi$ for some $C>0$, the plan $\ppi_\mu:=\int \ppi_x\,\d\mu(x)$ also $q$-represents $\nabla g$, see also Proposition \ref{prop:altracar} below.  

Yet, it is curious to observe that the property of $q$-representing a gradient can be fully described by looking at the marginals $(\e_t)_\sharp\ppi$ only. Indeed, assume that $\ppi$ $q$-represents $\nabla g$. Then the curve $t\mapsto\mu_t:=(\e_t)_\sharp\ppi$ is $q$-absolutely continuous w.r.t. $W_q$ in a neighborhood of $0$. Now let $\tilde\ppi\in \prob{C([0,1],X)}$ be any plan associated to $(\mu_t)$ via Theorem \ref{thm:lisini}. We have
\[
\int\frac{g(\gamma_t)-g(\gamma_0)}{t}\,\d\ppi(\gamma)=\frac{1}{t}\left(\int g\,\d\mu_t-\int g\,\d\mu_0\right)=\int\frac{g(\gamma_t)-g(\gamma_0)}{t}\,\d\tilde\ppi(\gamma),
\]
$\|\weakgrad g\|_{L^p(X,(\e_0)_\sharp\sppi)}=\|\weakgrad g\|_{L^p(X,\mu_0)}=\|\weakgrad g\|_{L^p(X,(\e_0)_\sharp\tilde\sppi)}$ and
\begin{equation}
\label{eq:tantononestretta}
\lims_{t\downarrow 0}\frac1t\iint_0^t|\dot\gamma_t|^q\,\d t\,\d\ppi(\gamma)\geq \lims_{t\downarrow 0}\frac1t\int_0^t|\dot\mu_t|^q\,\d t=\lims_{t\downarrow 0}\frac1t\iint_0^t|\dot\gamma_t|^q\,\d t\,\d\tilde\ppi(\gamma),
\end{equation}
which gives that $\tilde\ppi$ also $q$-represents $\nabla g$. This argument also shows that the inequality in \eqref{eq:tantononestretta} must be an equality, or else for the plan $\tilde\ppi$ the inequality \eqref{eq:dual} would fail.
}\fr\end{remark}
The following theorem is the key technical point on which the theory developed in this paper relies. It relates the `vertical derivatives' $D^\pm f(\nabla g)$ (so called because $D^+ f(\nabla g),D^- f(\nabla g)$ are obtained by a perturbation in the dependent variable), with the `horizontal derivatives' $\limi_{t\downarrow0}\int\frac{f\circ\e_t-f\circ\e_0}{t}\,\d\ppi$, $\lims_{t\downarrow0}\int\frac{f\circ\e_t-f\circ\e_0}{t}\,\d\ppi$ (so called because $f$ is perturbed in the dependent variable).
\begin{theorem}[Horizontal and vertical derivatives]\label{thm:horver} Let $(X,\sfd,\mm)$ be as in \eqref{eq:mms},  $p,q\in(1,\infty)$ conjugate exponents and  $f,g\in\s^p(X,\sfd,\mm)$. Then for every   plan $\ppi\in\prob{C([0,1],X)}$ which $q$-represents  $\nabla g$ it holds
\begin{equation}
\label{eq:horver1}
\begin{split}
\int D^+f(\nabla g)\weakgrad g^{p-2}\,\d(\e_0)_\sharp\ppi
&\geq\lims_{t\downarrow 0}\int \frac{f(\gamma_t)-f(\gamma_0)}{t}\,\d\ppi(\gamma)\\
&\geq\limi_{t\downarrow 0}\int \frac{f(\gamma_t)-f(\gamma_0)}{t}\,\d\ppi(\gamma)\geq \int D^-f(\nabla g)\weakgrad g^{p-2}\,\d(\e_0)_\sharp\ppi.
\end{split}
\end{equation}
\end{theorem}
\begin{proof}
By the inequality \eqref{eq:dual} applied to $g+\eps f$ we have
\[
\lims_{t\downarrow0}\int\frac{(g+\eps f)(\gamma_t)-(g+\eps f)(\gamma_0)}t\,\d\ppi(\gamma)\leq \frac{\|\weakgrad{(g+\eps f)}\|^p_{L^p(X,(\e_0)_\sharp\sppi)}}{p}+\frac{\|\ppi\|_q^q}{q},
\]
and by the assumption on $g$ and $\ppi$ we know that
\[
\limi_{t\downarrow0}\int\frac{g(\gamma_t)-g(\gamma_0)}t\,\d\ppi(\gamma)\geq \frac{\|\weakgrad{g}\|^p_{L^p(X,(\e_0)_\sharp\sppi)}}{p}+\frac{\|\ppi\|_q^q}{q}.
\]
Subtract the second inequality from the first to get
\[
\lims_{t\downarrow0}\eps\int\frac{f(\gamma_t)-f(\gamma_0)}{t}\,\d\ppi(\gamma)\leq \int\frac{\weakgrad{(g+\eps f)}^p-\weakgrad g^p}{p}\,\d(\e_0)_\sharp\ppi.
\]
Divide by $\eps>0$ (resp. $\eps<0$) and use the dominate convergence theorem to get the first (resp. the third) inequality in \eqref{eq:horver1}.
\end{proof}
The rest of the section is devoted to the study of the properties of plans $q$-representing gradients. In particular, in Theorem \ref{thm:extest} we will show that they actually exist for a wide class of initial data (i.e. given value of $(\e_0)_\sharp\ppi$).
\begin{proposition}\label{prop:altracar} Let $(X,\sfd,\mm)$ be as in \eqref{eq:mms}, $p,q\in(1,\infty)$ conjugate exponents, $g\in\s^p(X,\sfd,\mm)$ and $\ppi\in\prob{C([0,1],X)}$ a plan of bounded compression.

Then $\ppi$ $q$-represents $\nabla g$ if and only if
\begin{equation}
\label{eq:repcambiata}
\lim_{t\downarrow0}\frac{g\circ\e_t -g\circ\e_0}{E_{q,t}}=\lim_{t\downarrow0}\left(\frac{E_{q,t}}{t}\right)^{\frac qp}=\weakgrad g\circ\e_0,\qquad\textrm{in }\quad L^p(C([0,1],X),\ppi).
\end{equation}
\end{proposition}
\begin{proof} The `if' is obvious, so we turn to the `only if'.

Thanks to \eqref{eq:dual} and its proof, \eqref{eq:degiorgi} and the equality case in Young's inequality, we have
\begin{equation}
\label{eq:ok}
\lims_{t\downarrow0}\int\frac{g\circ\e_t-g\circ\e_0}{t}\,\d\ppi=\limi_{t\downarrow0}\int\frac{g\circ\e_t-g\circ\e_0}{t}\,\d\ppi= \|\ppi\|_q^q=\|\weakgrad g\|_{L^p(X,(\e_0)_\sharp\sppi)}^p=:L
\end{equation}
Define the functions $A_t,B_t,C_t:C([0,1],X)\to\R\cup\{\pm\infty\}$ by
\[
A_t:=\frac{g\circ\e_t-g\circ\e_0}{E_{q,t}},\qquad \qquad B_t:=\frac{E_{q,t}}{t},\qquad\qquad  C_t:=\sqrt[p]{\frac1t\int_0^t\weakgrad g^p\circ\e_s\,\d s},
\]
and notice that from \eqref{eq:curve} we have 
\begin{equation}
\label{eq:dt}
|A_t|\leq C_t,\qquad\ppi-a.e.,
\end{equation}
and that from \eqref{eq:convlp} we know that
\begin{equation}
\label{eq:dt2}
C_t\to \weakgrad g\circ\e_0,\qquad\textrm{ in }L^p(\ppi).
\end{equation}
From \eqref{eq:ok}, \eqref{eq:dt} and \eqref{eq:dt2} we deduce
\begin{equation}
\label{eq:ok2}
\begin{split}
L&=\lim_{t\downarrow0}\int\frac{g\circ\e_t-g\circ\e_0}{t}\,\d\ppi=\lim_{t\downarrow0}\int A_tB_t\,\d\ppi\leq \limi_{t\downarrow0}\int |A_t|B_t\,\d\ppi\\
&\leq \limi_{t\downarrow0}\left(\frac{\|A_t\|^p_{L^p(\sppi)}}{p}+\frac{\|B_t\|^q_{L^q(\sppi)}}{q}\right)\leq \lims_{t\downarrow0}\frac{\|A_t\|^p_{L^p(\sppi)}}{p}+\lims_{t\downarrow0}\frac{\|B_t\|^q_{L^q(\sppi)}}{q}\\
&\leq \lims_{t\downarrow0}\frac{\|C_t\|^p_{L^p(\sppi)}}{p}+\frac{\|\ppi\|_q^q}{q}=L.
\end{split}
\end{equation}
In particular it holds $\limi_{t\downarrow0}\left(\frac{\|A_t\|^p_{L^p(\sppi)}}{p}+\frac{\|B_t\|^q_{L^q(\sppi)}}{q}\right)= \lims_{t\downarrow0}\frac{\|A_t\|^p_{L^p(\sppi)}}{p}+\lims_{t\downarrow0}\frac{\|B_t\|^q_{L^q(\sppi)}}{q}$, which forces the limits of $\|A_t\|^p_{L^p(\sppi)},\|B_t\|^q_{L^q(\sppi)}$ as $t\downarrow 0$ to exist. Then, since the last inequality of \eqref{eq:ok2} is an equality, we also get that $\lim_{t\downarrow0}\|A_t\|^p_{L^p(\sppi)}=\lim_{t\downarrow0}\|C_t\|^p_{L^p(\sppi)}$, which together with \eqref{eq:dt} and \eqref{eq:dt2} gives that $|A_t|\to \weakgrad g\circ\e_0 $ as $t\downarrow0$ in $L^p(\ppi)$. Since also the first inequality in \eqref{eq:ok2} is an equality, we deduce also that $A_t\to \weakgrad g\circ\e_0 $ as $t\downarrow0$ in $L^p(\ppi)$.

Using \eqref{eq:ok2} again we get  that $\lim_{t\downarrow0}\|B_t\|^q_{L^q(\sppi)}=L$. Let $B\in L^q(\ppi)$ be any weak limit of $B_t$ as $t\downarrow0$ in $L^q(\ppi)$. From
\begin{equation}
\label{eq:ok3}
L=\lim_{t\downarrow0}\int A_tB_t\,\d\ppi=\int  \weakgrad g\circ\e_0\ B\,\d\ppi\leq \frac{\| \weakgrad g\circ\e_0\|^p_{L^p(\sppi)}}{p}+\frac{\|B\|^q_{L^q(\sppi)}}{q}\leq L,
\end{equation}
we deduce $\|B\|^q_{L^q(\sppi)}=L$, so that the weak convergence of $B_t$ is actually strong. Finally, since the first inequality in \eqref{eq:ok3} is an equality, we deduce $B^q=\weakgrad g^p\circ\e_0$ $\ppi$-a.e., and the proof is achieved.
\end{proof}
Direct consequences of characterization \eqref{eq:repcambiata} are the following. For $g\in \s^p(X,\sfd,\mm)$ it holds:
\begin{itemize}
\item if $\ppi_1,\ppi_2$ are plans which $q$-represent $\nabla g$ then
\begin{equation}
\label{eq:somma}
\ppi:=\lambda\ppi_1+(1-\lambda)\ppi_2\qquad q-\textrm{represents }\nabla g,\qquad\textrm{for any }\lambda\in[0,1],
\end{equation}
\item if $\ppi$ is a plan  which $q$-represents $\nabla g$  and $F:C([0,1],X)\to \R$ is a non-negative bounded Borel function such that $\int F\,\d\ppi=1$, then
\begin{equation}
\label{eq:scalato}
\tilde\ppi:= F\,\ppi \qquad q-\textrm{represents }\nabla g.
\end{equation}
\end{itemize}
We turn to the existence of plans representing gradients.
\begin{remark}{\rm
In connection with Remark \ref{re:degiorgi2}, it is worth stressing the following. The general existence results for gradient flows in metric spaces are based either on some  compactness assumption, possibly w.r.t. some weak topology, of the sublevels of the function (see e.g. Chapters 2 and 3 of \cite{Ambrosio-Gigli-Savare05}), or on some structural assumption on the geometry of the space, which, very roughly said, expresses the fact that the metric space `looks like an Hilbert space on small scales' (see e.g. Chapter 4 of \cite{Ambrosio-Gigli-Savare05} and \cite{Savare07}).

In our current framework, none of the two assumptions is made. Yet, it is still possible to prove that this sort of `infinitesimal gradient flow of a Sobolev function' exists. The necessary compactness is gained by looking at the evolution of the probability densities in $L^2(X,\mmz)$, and then coming back to measures on $\prob{C([0,1],X)}$ via the use of Theorems  \ref{thm:keyex} and \ref{thm:lisini}.
}\fr\end{remark}

Recall that we fixed $\mmz\in\prob X$ satisfying \eqref{eq:mmz} and \eqref{eq:momp} and that we defined the Cheeger energy functional $\widetilde\c_p:L^2(X,\mmz)\to[0,\infty]$ in \eqref{eq:cheegertilde}.
\begin{lemma}\label{le:planf}
Let $(X,\sfd,\mm)$ be as in \eqref{eq:mms}, $\mmz$ as in \eqref{eq:mmz} and \eqref{eq:momp}, $p,q\in(1,\infty)$  conjugate exponents and $g\in\s^p(X,\sfd,\mm)\cap L^\infty(X,\mm)$. Then there exists a plan $\ppi$ which $q$-represents $\nabla g$ and such that $c\mmz\leq (\e_0)_\sharp\ppi\leq C\mmz$ for some $c,C>0$.
\end{lemma}
\begin{proof}
Define the convex function $u_q:[0,\infty)\to\R$ as
\begin{align*}
u_q(z)&:=\frac{z^{3-q}-(3-q)z}{(3-q)(2-q)},\qquad \textrm{if }q\neq2,3,\\
u_2(z)&:=z\log z-z,\\
u_3(z)&:=z-\log z.
\end{align*}
If $q=2$ put $\rho_0:=\tilde ce^{-g}$ otherwise find $a\in\R$ such that $1+(a-g)(2-q)>b>0$ $\mm$-a.e. and define
\[
\rho_0:=\tilde c\Big(1+\big(a-g\big)(2-q)\Big)^{\frac1{2-q}},
\]
where in any case $\tilde c$ is chosen so that $\int\rho_0\,\d\mmz=1$. Notice that by construction it holds $c\leq \rho_0\leq C$ $\mm$-a.e. for some $c,C>0$, $u_q'(\rho_0)=-g+constant$ and, by the chain rule \eqref{eq:chaineasy}, $\rho_0\in\s^p(X,\sfd,\mm)\subset\s^p(X,\sfd,\mmz)$.

Now consider the gradient flow $(\rho_t)$ of the Cheeger energy $\widetilde\c_p$ in $L^2(X,\mmz)$ starting from $\rho_0$. Part $(ii)$ of Theorem \ref{thm:keyex} ensures that the densities $\rho_t$ are non-negative and uniformly bounded in $L^\infty(X,\mmz)$, while $(i)$ of the same theorem grants $\int\rho_t\,\d\mmz=1$ for any $t\geq 0$. Hence we have $\mu_t:=\rho_t\mmz\in\prob X$, and the $L^\infty$ bound on $\rho_t$ in conjunction with \eqref{eq:momp} yields that the measures $\mu_t$ have finite $q$-moment for any $t\geq 0$. Furthermore, from
$(v)$ of Theorem \ref{thm:keyex} we also have that the curve $[0,1]\ni t\mapsto \mu_t$  is $q$-absolutely continuous w.r.t. $W_q$. Hence we can use Theorem \ref{thm:lisini} to associate to $(\mu_t)$ a plan $\ppi\in\prob{C([0,1],X)}$ concentrated on $AC^q([0,1],X)$ satisfying \eqref{eq:lisini}.

We claim that $\ppi$ $q$-represents $\nabla g$. Since we already know by construction that  $c\mmz\leq (\e_0)_\sharp\ppi\leq C\mmz$ for some $c,C>0$, this will give the thesis. 

Since the $\rho_t$'s are uniformly bounded and $\mmz\leq \tilde C\mm$ for some $\tilde C>0$,  $\ppi$ has bounded compression. Using again $(v)$ of Theorem \ref{thm:keyex} and \eqref{eq:lisini} we get
\[
\iint_0^t|\dot\gamma_s|^q\,\d s\,\d\ppi(\gamma)=\int_0^t|\dot\mu_s|^q\,\d s\leq  \iint_0^t\frac{\weakgrad{\rho_s}^p}{\rho_s^{q-1}}\,\d s\,\d\mmz=  \iint_0^t u_q''(\rho_s)\weakgrad{\rho_s}^p\,\d s\,\d\mmz.
\]
Recall  that $\rho_0\in\s^p(X,\sfd,\mm)$ so that applying \eqref{eq:dissipazione0} we get
\begin{equation}
\label{eq:Ct}
\|\ppi\|_q^q=\lims_{t\downarrow0}\frac1t\iint_0^t|\dot\gamma_s|^q\,\d s\leq \lims_{t\downarrow 0}\frac1t\iint_0^t u_q''(\rho_s)\weakgrad{\rho_s}^p\,\d s\,\d\mmz=\int u_q''(\rho_0)\weakgrad{\rho_0}^p\,\d\mmz.
\end{equation}
On the other hand, the  convexity of $u_q$ gives
\[
\begin{split}
\int\frac{g(\gamma_t)-g(\gamma_0)}t\,\d\ppi(\gamma)&=\int\frac{u'_q\big(\rho_0(\gamma_0)\big)-u'_q\big(\rho_0(\gamma_t)\big)}t\,\d\ppi(\gamma)= \frac1t\int u_q'(\rho_0)(\rho_0-\rho_t)\,\d\mmz\\
&\geq \frac1t\int u_q(\rho_0)-u_q(\rho_t)\,\d\mmz,
\end{split}
\]
so that using  part $(iii)$ and $(iv)$ of Theorem \ref{thm:keyex} we get
\begin{equation}
\label{eq:At}
\limi_{t\downarrow0}\int\frac{g(\gamma_t)-g(\gamma_0)}t\,\d\ppi(\gamma)\geq\int u_q''(\rho_0)\weakgrad{\rho_0}^p\,\d\mmz.
\end{equation}
Since $u_q''(\rho_0)\weakgrad{\rho_0}^p=\rho_0\weakgrad g^p$ $\mmz$-a.e., we have $\int u_q''(\rho_0)\weakgrad{\rho_0}^p\,\d\mmz=\|\weakgrad g\|^p_{L^p((\e_0)_\sharp\sppi)}$, thus the thesis follows from \eqref{eq:At}, \eqref{eq:Ct}.
\end{proof}

\begin{theorem}[Existence of plans representing $\nabla g$]\label{thm:extest}
Let $(X,\sfd,\mm)$ be as in \eqref{eq:mms}, $\mmz$ as in \eqref{eq:mmz} and \eqref{eq:momp}, $p,q\in(1,\infty)$  conjugate exponents, $g\in\s^p(X,\sfd,\mm)$ and $\mu\in\prob X$ a measure such that $\mu\leq C\mmz$ for some $C>0$. 

Then there exists $\ppi\in\prob{C([0,1],X)}$ which $q$-represents $\nabla g$ and such that $(\e_0)_\sharp\ppi=\mu$.
\end{theorem}
\begin{proof}
Let $\varphi:\R\to[0,1]$ be defined by
\[
\varphi(z):=\left\{
\begin{array}{ll}
x-2n&,\qquad\textrm{ if } x\in[2n,2n+1),\textrm{ for some }n\in\Z,\\
2n-x&,\qquad\textrm{ if } x\in[2n-1,2n),\textrm{ for some }n\in\Z,\\
\end{array}
\right.
\]
and $\psi:=1-\varphi$. Clearly, both  $\varphi$ and $\psi$ are $1$-Lipschitz. Apply Lemma \ref{le:planf} above to $\varphi\circ g$ and $\psi\circ g$ to get plans $\ppi^1,\ppi^2$ $q$-representing $\nabla(\varphi\circ g),\nabla(\psi\circ g)$ respectively and with $c\mmz\leq(\e_0)_\sharp\ppi^i\leq C\mmz$, $i=1,2$, for some $c,C>0$. Let $A:=\{g^{-1}(\cup_n[2n,2n+1))\}\subset X$, define $F^1,F^2: C([0,1],X)\to \R$ by
\[
\begin{split}
F^1(\gamma)&:=\nchi_A(\gamma_0)\frac{\d\mu}{\d(\e_0)_\sharp\ppi^1}(\gamma_0),\\
F^2(\gamma)&:=\nchi_{X\setminus A}(\gamma_0)\frac{\d\mu}{\d(\e_0)_\sharp\ppi^2}(\gamma_0),
\end{split}
\]
and notice that since $\frac{\d(\e_0)\sharp\sppi^i}{\d\mmz}$ is bounded from below, $i=1,2$, $F^1$ and $F^2$ are bounded. Also, we know from \eqref{eq:nullset} that $\mmz$-a.e. on $g^{-1}(\Z)$ it holds $\weakgrad g=\weakgrad{(-g)}=0$. Now let $\ppi:=F^1\ppi^1+F^2\ppi^2$ and use \eqref{eq:somma} and \eqref{eq:scalato} to conclude that $\ppi$ indeed $q$-represents $\nabla g$. Since by construction it holds $(\e_0)_\sharp\ppi=\mu$, the proof is concluded.
\end{proof}
\subsection{Calculus rules}
In this section we prove the basic calculus rules for $D^\pm f(\nabla g)$: the chain and Leibniz rules.
\begin{proposition}[Chain rules]\label{prop:chain}
Let $(X,\sfd,\mm)$ be as in \eqref{eq:mms}, $p\in(1,\infty)$ and $f,g\in\s^p_{\rm loc}(X,\sfd,\mm)$. Also, let $\varphi:\R\to\R$ be locally Lipschitz with the following property: for any $x\in X$ there exists a neighborhood $x\in U_x\subset X$ and an open interval $I_x\subset\R$ such that $\mm(U_x\setminus f^{-1}(I_x))=0$ and $\varphi\restr{I_x}$ is Lipschitz.

Then for any $\varphi\circ f\in \s^p_{\rm loc}(X,\sfd,\mm)$ and it holds
\begin{equation}
\label{eq:chainf}
D^\pm(\varphi\circ f)(\nabla g)=\varphi'\circ f\ D^{\pm{\rm sign}(\varphi'\circ f)}f(\nabla g),\qquad\mm-a.e.,
\end{equation}
where the value of $\varphi'\circ f$ is defined arbitrarily at points $x$ where $\varphi$ is not differentiable at $f(x)$. 

Similarly,  assume that $\varphi:\R\to\R$ is  locally Lipschitz  and with the following property: for any $x\in X$ there exists a neighborhood $x\in U_x\subset X$ and an open interval $I_x\subset\R$ such that $\mm(U_x\setminus g^{-1}(I_x))=0$ and $\varphi\restr{I_x}$ is Lipschitz.

Then $\varphi\circ g\in\s^p_{\rm loc}(X,\sfd,\mm)$ and it holds
\begin{equation}
\label{eq:chaing}
D^{\pm} f(\nabla (\varphi\circ g))=\varphi'\circ g\ D^{\pm{\rm sign}(\varphi'\circ g)}f(\nabla g),\qquad\mm-a.e.,
\end{equation}
where  the value of $\varphi'\circ g$ is defined arbitrarily at points $x$ where $\varphi$ is not differentiable at $g(x)$. 
\end{proposition}
\begin{proof} The fact that $\varphi\circ f$ (resp. $\varphi\circ g$) is in $\s^p_{\rm loc}(X,\sfd,\mm)$ is a direct consequence of the assumptions and formula \eqref{eq:chaineasy}. Also, thanks to the local nature of the claim it is sufficient to deal with the case $f,g\in \s^p(X,\sfd,\mm)$ and $\varphi:\R\to\R$ Lipschitz.

We now prove \eqref{eq:chainf}. Let $\mathcal N\subset \R$ be the $\mathcal L^1$-negligible set of points of non-differentiability of $\varphi$. Then $\varphi(\mathcal N)$ is also $\mathcal L^1$-negligible and by \eqref{eq:nullset} we deduce that \eqref{eq:chainf} holds $\mm$-a.e. on $f^{-1}(\mathcal N)$ because both sides are 0. Now observe that by the very definition of $D^\pm f(\nabla g)$, the thesis is obvious if $\varphi$ is affine. By the locality principle \eqref{eq:localgrad} and arguing as before to deal with non-differentiability points, the thesis is also true if $\varphi$ is piecewise affine. For general $\varphi$, find a sequence of piecewise affine functions $\varphi_n$ such that $\varphi_n'\to \varphi'$ $\mathcal L^1$-a.e..  Let $\mathcal N'$ be the $\mathcal L^1$-negligible set of $z$'s such that either one of $\varphi,\varphi_n$ is not differentiable at $z$ or $\varphi'_n(z)$ does not converge at $\varphi'(z)$. With the same arguments used before we get that \eqref{eq:chainf} holds $\mm$-a.e. on $f^{-1}(\mathcal N')$ because both sides are 0. On $X\setminus f^{-1}(N')$ we use the continuity property \eqref{eq:1lip} and the chain rule \eqref{eq:chaineasy} to get
\[
\begin{split}
\big|D^\pm(\varphi\circ f)(\nabla g)-D^\pm(\varphi_n\circ f)(\nabla g)\big|&\leq \weakgrad{((\varphi-\varphi_n)\circ f)}\weakgrad g=|\varphi'-\varphi_n'|\circ f\weakgrad f\weakgrad g,
\end{split}
\]
and by construction the rightmost term goes to 0 as $n\to\infty$ $\mm$-a.e. on $X\setminus f^{-1}(\mathcal N')$.

We pass to \eqref{eq:chaing}. Using  \eqref{eq:4} we can reduce ourselves to prove that $D^+f(\nabla(\varphi\circ g))=\varphi'\circ g\,D^{{\rm sign}(\varphi'\circ g)}f(\nabla g)$. Notice that the trivial identity
\[
\frac{\weakgrad{(ag+\eps f)}^p-\weakgrad{(ag)}^p}{p\eps\weakgrad{(ag)}^{p-2}}=a\frac{\weakgrad{(g+\frac\eps af)}^p-\weakgrad{g}^p}{p\frac\eps a\weakgrad g^{p-2}},\qquad\mm-a.e.,
\]
valid for any $\eps, a\neq 0 $ together with \eqref{eq:chainf}, yields \eqref{eq:chaing} for $\varphi$ linear. Hence \eqref{eq:chaing} also holds for $\varphi$ affine and, by locality, also for  $\varphi$ piecewise affine, where the same arguments as before are used to deal with the set of points of non-differentiability of $\varphi$.

Let $\tilde J\subset J$ be an interval where $\varphi$ is Lipschitz, and $(\varphi_n)$ a sequence of uniformly Lipschitz and piecewise affine functions such that $\varphi_n'\to \varphi'$ $\mathcal L^1$-a.e. on $\tilde J$. By construction, for every $E\subset g^{-1}(\tilde J)$ the sequence of functions $\varphi_n'\circ g\, D^{{\rm sign}(\varphi_n'\circ g)}f(\nabla g)\weakgrad{(\varphi_n\circ g)}^{p-2}$ is dominated in $L^1(E,\mm\restr E)$ and pointwise $\mm$-a.e. converges to $\varphi'\circ g\, D^{{\rm sign}(\varphi'\circ g)}f(\nabla g)\weakgrad{(\varphi\circ g)}^{p-2}$. From the fact that $\int_E\weakgrad{(\varphi_n\circ g-\varphi\circ g)}^p\,\d\mm\to 0$ as $n\to\infty$ and the upper semicontinuity statement of Proposition \ref{prop:convconc} we deduce
\[
\begin{split}
\int_E D^+f(\nabla(\varphi\circ g))\weakgrad{(\varphi\circ g)}^{p-2}\,\d\mm&\geq\lims_{n\to\infty}\int_E  D^+f(\nabla(\varphi_n\circ g))\weakgrad{(\varphi_n\circ g)}^{p-2}\,\d\mm\\
&=\lims_{n\to\infty}\int_E  \varphi_n'\circ g\, D^{{\rm sign}(\varphi_n'\circ g)}f(\nabla g)\weakgrad{(\varphi_n\circ g)}^{p-2}\,\d\mm\\
&=\int_E  \varphi'\circ g\, D^{{\rm sign}(\varphi'\circ g)}f(\nabla g)\weakgrad{(\varphi\circ g)}^{p-2}\,\d\mm.
\end{split}
\]
By the arbitrariness of  $\tilde J$ and $E$ we get
\begin{equation}
\label{eq:daunlato}
D^+f(\nabla(\varphi\circ g))\geq \varphi'\circ g\ D^{{\rm sign}(\varphi'\circ g)}f(\nabla g),\qquad\mm-a.e..
\end{equation}
Roughly said, the conclusion comes by applying this very same inequality with $g$ replaced by $\varphi\circ g$ and $\varphi$ replaced by $\varphi^{-1}$. To make this rigorous, assume at first that $\varphi$ is in $C^1_{\rm loc}$. Notice that equality in \eqref{eq:daunlato} holds $\mm$-a.e. on $g^{-1}(\{\varphi'=0\})$. Now pick $z$ such that $\varphi'(z)\neq 0$, say $\varphi'(z)>0$ (the other case is similar) and let $(a,b)\ni z$ be an open interval such that $\inf_{(a,b)}\varphi'>0$. Put $\tilde g:=\min\{\max\{g,a\},b\}$, notice that  $\varphi$ is invertible on $\tilde g(X)$ with $C^1_{\rm loc}$ inverse and use \eqref{eq:daunlato} to get
\[
D^+ f(\nabla \tilde g)=D^+  f(\nabla(\varphi^{-1}\circ(\varphi\circ \tilde g)))\geq (\varphi^{-1})'\circ(\varphi\circ \tilde g) D^+  f(\nabla(\varphi\circ \tilde g))=\frac{1}{\varphi'\circ \tilde g}D^+  f(\nabla(\varphi\circ \tilde g)),
\]
$\mm$-a.e.. From the locality principle \eqref{eq:localgrad} and the arbitrariness of $z$ we conclude that \eqref{eq:chaing} holds for generic $C^1_{\rm loc}$ functions $\varphi$. The general case follows by approximating $\varphi$ with a sequence $(\varphi_n)$ of $C^1_{\rm loc}$ functions such that $\mathcal L^1(\{\varphi_n\neq\varphi\}\cup\{\varphi_n'\neq\varphi'\})\to 0$ as $n\to\infty$ and using again the locality principle.
\end{proof}
The following simple lemma will be useful in proving the Leibniz rule.
\begin{lemma}\label{le:perleib} Let $(X,\sfd,\mm)$ be as in \eqref{eq:mms}, $p,q\in(1,\infty)$ conjugate exponents, $f_1,f_2\in\s^p(X,\sfd,\mm)\cap L^\infty(X,\mm)$, $g\in \s^p(X,\sfd,\mm)$ and $\ppi\in\prob{C([0,1],X)}$ a  plan $q$-representing $\nabla g$.

Then
\[
\lim_{t\downarrow 0}\int\left|\frac{\big(f_1(\gamma_t)-f_1(\gamma_0)\big)\big(f_2(\gamma_t)-f_2(\gamma_0)\big)}{t}\right|\,\d\ppi(\gamma)=0.
\] 
\end{lemma}
\begin{proof}
Let $F^i_t(\gamma):=f_i(\gamma_t)-f_i(\gamma_0)$, $i=1,2$, $t\in(0,1]$. From \eqref{eq:repcambiata} we know that the family $\frac{E_{q,t}}{t}$ is dominated in $L^q(\ppi)$, so that from the second part of Proposition \ref{prop:bridge} we get that  the family $\frac{F^1_t}t$ is dominated in $L^1(\ppi)$. By definition we also know that $\|F^2_t\|_{L^\infty(\sppi)}\leq 2\|f_2\|_{L^\infty(\smm)}$ for any $t\in(0,1]$. Hence to conclude it is sufficient to check that $\ppi$-a.e. it holds $F^2_t\to 0$ as $t\downarrow 0$. But this is obvious, because $\frac{F^2_t}{t}$ is dominated in $L^1(\ppi)$ as well.    
\end{proof}

\begin{proposition}[Leibniz rule]\label{prop:leibniz}  Let $(X,\sfd,\mm)$ be as in \eqref{eq:mms}, $p\in(1,\infty)$,  $f_1,f_2\in \s^p_{\rm loc}(X,\sfd,\mm)\cap L^\infty_{\rm loc}(X,\mm)$ and $g\in\s^p_{\rm loc}(X,\sfd,\mm)$. 

Then $\mm$-a.e. it holds
\begin{equation}
\label{eq:leibniz}
\begin{split}
 D^+(f_1f_2)(\nabla g)&\leq f_1D^{s_1}f_2(\nabla g)+f_2D^{s_2}f_1(\nabla g),\\
 D^-(f_1f_2)(\nabla g)&\geq f_1D^{-s_1}f_2(\nabla g)+f_2D^{-s_2}f_1(\nabla g),
\end{split}
\end{equation}
where $s_i:={\rm sign}f_i$, $i=1,2$.
\end{proposition}
\begin{proof} With a cut-off argument and from the locality property \eqref{eq:localgrad2}, we can assume that $f_1,f_2\in \s^p(X,\sfd,\mm)\cap L^\infty(X,\mm)$ and $g\in\s^p(X,\sfd,\mm)$. Also, replacing $f_1,f_2$ with $|f_1|, |f_2|$ and using Proposition \ref{prop:chain} above, we reduce to the case $f_1,f_2\geq 0$. 

Under this assumptions, we will prove that for any measure $\mu\in\prob X$ with $\mu\leq C\mmz$ for some $C>0$ (with $\mmz$  as in \eqref{eq:mmz}, \eqref{eq:momp})  it holds
\begin{equation}
\label{eq:mu}
\int D^+(f_1f_2)(\nabla g)\weakgrad g^{p-2}\,\d\mu\leq \int \Big(f_1D^{+}f_2(\nabla g)+f_2D^{+}f_1(\nabla g)\Big)\weakgrad g^{p-2}\,\d\mu, 
\end{equation}
which implies, since $\mu$ is arbitrary and $\mm\ll\mmz$, the first  inequality in \eqref{eq:leibniz}. The other one will then follow putting $-g$ in place of $g$ and using \eqref{eq:4}. 

Fix such $\mu$ and notice that if $f_1f_2=0$ $\mu$-a.e., then \eqref{eq:mu} is obvious, thus we can assume $\int f_i\,\d\mu>0$, $i=1,2$.

Fix $\eps>0$, put $g_\eps:=g+\eps f_1f_2\in \s^p(X,\sfd,\mm)$ and let $\ppi^\eps$ be a plan $q$-representing $\nabla g_\eps$ with $(\e_0)_\sharp\ppi=\mu$ (Theorem \ref{thm:extest}, here $q$ is the conjugate exponent of $p$) and recall \eqref{eq:6}, \eqref{eq:horver1} to get 
\[
\begin{split}
\int D^+(f_1f_2)(\nabla g)\weakgrad g^{p-2}\,\d\mu&\leq \int D^-(f_1f_2)(\nabla g_\eps)\weakgrad{g_\eps}^{p-2}\,\d\mu\\
&\leq \limi_{t\downarrow 0}\int\frac{f_1(\gamma_t)f_2(\gamma_t)-f_1(\gamma_0)f_2(\gamma_0)}{t}\,\d\ppi^\eps(\gamma)\\
&\leq \lims_{t\downarrow 0}\int\frac{f_1(\gamma_t)f_2(\gamma_t)-f_1(\gamma_0)f_2(\gamma_0)}{t}\,\d\ppi^\eps(\gamma).
\end{split}
\]
Taking into account Lemma \ref{le:perleib} we have
\[
\begin{split}
&\lims_{t\downarrow 0}\int\frac{f_1(\gamma_t)f_2(\gamma_t)-f_1(\gamma_0)f_2(\gamma_0)}{t}\,\d\ppi^\eps(\gamma)\\
&=\lims_{t\downarrow 0}\int f_1(\gamma_0)\frac{f_2(\gamma_t)-f_2(\gamma_0)}{t}+ f_2(\gamma_0)\frac{f_1(\gamma_t)-f_1(\gamma_0)}{t}\,\d\ppi^\eps(\gamma)\\
&\leq\lims_{t\downarrow 0}\int f_1(\gamma_0)\frac{f_2(\gamma_t)-f_2(\gamma_0)}{t}\,\d\ppi^\eps(\gamma)+\lims_{t\downarrow 0}\int f_2(\gamma_0)\frac{f_1(\gamma_t)-f_1(\gamma_0)}{t}\,\d\ppi^\eps(\gamma).
\end{split}
\]
From \eqref{eq:scalato} we know that the plans $\ppi^\eps_i:=\frac{1}{\int f_i\,\d\mu}f_i\circ\e_0\ppi^\eps$, $i=1,2$, $q$-represent $\nabla g_\eps$ with $(\e_0)_\sharp\ppi^\eps_i=\frac{1}{\int f_i\,\d\mu}f_i\circ\e_0\,\mu$. Hence using again \eqref{eq:horver1}  we  obtain
\[
\begin{split}
\lims_{t\downarrow 0}\int f_1(\gamma_0)\frac{f_2(\gamma_t)-f_2(\gamma_0)}{t}\,\d\ppi^\eps(\gamma)&\leq \int f_1 D^+f_2(\nabla g_\eps )\weakgrad{g_\eps }^{p-2}\,\d\mu,\\
\lims_{t\downarrow 0}\int f_2(\gamma_0)\frac{f_1(\gamma_t)-f_1(\gamma_0)}{t}\,\d\ppi^\eps(\gamma)&\leq \int f_2 D^+f_1(\nabla g_\eps)\weakgrad{g_\eps}^{p-2}\,\d\mu.
\end{split}
\]
In summary, we proved that for any $\eps>0$ it holds
\[
\int D^+(f_1f_2)(\nabla g)\weakgrad g^{p-2}\,\d\mu\leq \int f_1 D^+f_2(\nabla g_\eps)\weakgrad{g_\eps}^{p-2}+ f_2 D^+f_1(\nabla g_\eps)\weakgrad{g_\eps}^{p-2}\,\d\mu.
\]
Notice now that the integrand in the right hand side is dominated in $L^1(X,\mu)$  and use \eqref{eq:6} in considering its limit  as $\eps\downarrow0$. The thesis follows.
\end{proof}
\begin{remark}[The infinitesimally strictly convex case]{\rm
If $(X,\sfd,\mm)$ is $q$-infinitesimally strictly convex, $q$ being the conjugate exponent of $p$, then under the same assumptions of Propositions \ref{prop:chain} and \ref{prop:leibniz} we have the following $\mm$-a.e. equalities:
\begin{align}
\label{eq:chainfs}
D(\varphi\circ f)(\nabla g)&=\varphi'\circ f\ Df(\nabla g),\\
\label{eq:chaings}
Df(\nabla(\varphi\circ g))&=\varphi'\circ g\ Df(\nabla g),\\
\label{eq:leibs}
D(f_1f_2)(\nabla g)&=f_1 Df_2(\nabla g)+f_2Df_1(\nabla g).
\end{align}

}\fr\end{remark}
\begin{remark}{\rm
It is natural to question whether it is valid a Leibniz rule of a form like
\begin{equation}
\label{eq:leibno}
D f(\nabla(g_1g_2))=g_1\,Df(\nabla g_2)+g_2\,Df(\nabla g_1),\qquad\forall f,g_1,g_2\textrm{ ``smooth''}
\end{equation}
provided possibly we replace the equality with an inequality and do appropriate sign choices in $D^\pm$.

This is well known to be false even in a smooth world. For instance, in the flat normed setting considered in Section \ref{se:normato}, the above equality holds for any couple of smooth $f,g$ if and only if the norm comes from a scalar product.  Indeed, recall that $\nabla g={\rm Dual}^{-1}(Dg)$, so that \eqref{eq:leibno} holds for any $f,g_1,g_2$ if and only if
\[
{\rm Dual}^{-1}(D(g_1g_2))=g_1{\rm Dual}^{-1}(Dg_2)+g_2{\rm Dual}^{-1}(Dg_1).
\]
By the Leibniz rule for the differentials, we know that the left hand side is equal to ${\rm Dual}^{-1}(g_1Dg_2+g_2Dg_1)$. Therefore, since ${\rm Dual}^{-1}$  is always 1-homogeneous, \eqref{eq:leibno} is true if and only if ${\rm Dual}^{-1}$ is linear. That is, if and only if the norm $\|\cdot\|$ comes from a scalar product (compare this fact with the Leibniz rule \eqref{eq:leibb} obtained in Section \ref{se:lineare} under the assumption that the space is infinitesimally Hilbertian). 
}\fr\end{remark}
\begin{remark}[The different roles of horizontal and vertical derivation]\label{re:horver}{\rm In developing this abstract differential calculus, we proved the chain rule as a consequence of the `vertical' derivations and the Leibniz rule as a byproduct of the `horizontal' ones. We chose this path to present a first direct application of both the approaches and to show how they relate, but actually this is unnecessary, because the Leibniz rule can also be proved via the `vertical' approach only. Indeed, notice that as in the proof of Proposition \ref{prop:leibniz} we can reduce to prove that
\[
D^+(f_1f_2)(\nabla g)\leq f_1 D^+f_2(\nabla g)+f_2 D^+f_1(\nabla g),\qquad\mm-a.e.,
\]
for positive  $f_1,f_2\in \s^p(X,\sfd,\mm)\cap L^\infty(X,\mm)$ and $g\in\s^p(X,\sfd,\mm)$. Then use the chain rule\footnote{I saw this approach to the proof of  the Leibniz rule at the page \texttt{http://planetmath.org/?op=getobj\&from=objects\&id=8438}, where it is  acknowledged to W. Buck} \eqref{eq:chainf} and the convexity and positive 1-homogeneity of $f\mapsto D^+f(\nabla g)$ to get
\[
\begin{split}
D^+(f_1f_2)(\nabla g)&=f_1f_2\,D^+(\log(f_1f_2))(\nabla g)\leq f_1f_2\,\big(D^+(\log(f_1))(\nabla g)+D^+(\log(f_2))(\nabla g)\big)\\
&=f_1 D^+f_2(\nabla g)+f_2 D^+f_1(\nabla g),\qquad\mm-a.e..
\end{split}
\]
On the other hand, a calculus can also be developed starting from the `horizontal' approach only (but this would be much more cumbersome to do w.r.t. Definition \ref{def:dfgg}).

Therefore the point is not which one of the two approaches is the best, because they produce an equally powerful calculus. The point is that they produce the same calculus. Notice that the `horizontal' derivation is common when working with the Wasserstein geometry, while the `vertical' one is typical of the $L^p$ world. So that every time these two geometries are put in relation, in a way or another the same principle underlying Theorem \ref{thm:horver} comes into play. For instance, the relation between horizontal and vertical derivation was already present, although well hidden, in the first proof in a non-smooth context of the equivalence between the gradient flow of $\c_2$ in $L^2(X,\mm)$ and the one of the relative entropy in the Wasserstein space $(\probt X,W_2)$, provided in the case of Alexandrov spaces with lower curvature bounds in \cite{GigliKuwadaOhta10}.

In several  situations, one specific approach works better than the other, and their interaction has fruitful consequences. For instance, in getting the Laplacian comparison out of the $\CD(K,N)$ condition (see the introduction and Theorem \ref{thm:controllo}) it is crucial to take an horizontal derivative, which is not surprising given that the $\CD(K,N)$ condition is given in terms of the Wasserstein geometry: the same result seems hardly achievable working with the vertical approach only. On the other hand, the vertical approach is needed if one wants to know under which conditions the limit of
\[
\int \frac{f(\gamma_t)-f(\gamma_0)}{t}\,\d\ppi(\gamma),
\]
exists as $t\downarrow0$, where here $f\in\s^p(X,\sfd,\mm)$ and $\ppi$  $q$-represents $\nabla g$ for some $g\in\s^p(X,\sfd,\mm)$. In other words, if we want to know up to which extent we can consider plans $\ppi$ as differentiation operators. Using Theorem \ref{thm:horver} we know that this is the case as soon as $(X,\sfd,\mm)$ is $q$-infinitesimally strictly convex, a fact which would otherwise be not at all obvious from the horizontal approach only. Also, consider the following question: let $f_1,f_2\in\s^2(X,\sfd,\mm)$ and $\ppi_1,\ppi_2$ plans of bounded compression with $(\e_0)_\sharp\ppi^1=(\e_0)_\sharp\ppi^2=:\mu$ such that $\ppi_i$ $2$-represents $\nabla f_i$, $i=1,2$. Assume that the seminorm $\|\cdot\|_{\s^2}$ satisfies the parallelogram rule. Can we deduce that the equality
\[
\lim_{t\downarrow0}\int\frac{f_1(\gamma_t)-f_1(\gamma_0)}{t}\,\d\ppi_2(\gamma)=\lim_{t\downarrow0}\int\frac{f_2(\gamma_t)-f_2(\gamma_0)}{t}\,\d\ppi_1(\gamma),
\]
holds? This is expected, as the fact that  $\|\cdot\|_{\s^2}$ satisfies the parallelogram rule should mean that we can identify differentials and gradients, so that both the expressions above should be equal to something like 
\[
\int \nabla f_1\cdot\nabla f_2\,\d\mu.
\]
This is indeed the case, as we will see in the Chapter  \ref{se:lineare}, and once again this will follow from the `vertical' approach, rather than from the `horizontal' one (studies in directions have been firstly made in \cite{Ambrosio-Gigli-Savare11bis}).
}\fr\end{remark}
\begin{remark}[Comparison with Cheeger's differential structure]\label{re:cheeger}{\rm In the seminal paper \cite{Cheeger00}, Cheeger built a differential structure and defined the differential of Lipschitz functions in a completely different way. Working under the assumption that the measure is doubling and that the space supports a weak-local Poincar\'e inequality, he proved (among other things) that the space has a finite dimensional cotangent space, in the following sense. There exists $N\in\N$, a measurable decomposition  $X=\sqcup_{i\in \N} U_i$ and  Lipschitz functions ${\bf x}_i:U_i\to\R^{N_i}$, $N_i\leq N$,  such that for any $i\in\N$ and any Lipschitz function $f:X\to\R$ there exists a unique - up to negligible sets - map $df_i:U_i\to\R^{N_i}$ such that 
\begin{equation}
\label{eq:taylor}
\lim_{y\to x}\left|\frac{f(y)-f(x)-df_i(x)\cdot\big({\bf x}_i(y)-{\bf x}_i(x)\big)}{\sfd(y,x)}\right|=0,\qquad\textrm{ for }\mm-a.e.\ x\in U_i.
\end{equation}
This result can be interpreted as: the $U_i$ are coordinate charts, the map ${\bf x}_i$ provides the coordinate maps of the chart and formula \eqref{eq:taylor} is the first order Taylor expansion of $f$ in this chart. 

Relying on this result, he then interpreted (i.e. defined) the function $df_i$  as the differential of the Lipschitz map $f$ on the set $U_i$. From the construction, it is obvious that it satisfies the Leibniz rule in the following form, sharper than \eqref{eq:leibniz}:
\begin{equation}
\label{eq:leibch}
d(f_1f_2)_i=f_1(df_2)_i+f_2(df_1)_i.
\end{equation}
In particular, this holds without any sort of infinitesimal strict convexity. Yet, even in Cheeger's paper the issue of possibly not strictly convex norms in tangent spaces arose, in particular in connection with the surjectivity of a certain natural map on tangent cones. This issue was worked out using the finite dimensionality of the cotangent spaces to define equivalent well-behaved norms: his resulting Theorem 13.4 in \cite{Cheeger00} should be compared with our, weaker but valid in higher generality, Theorem \ref{thm:figata}.

Using the tools described in Appendix \ref{app:cottan} (see Theorems \ref{thm:horver2} and \ref{thm:figata}) one can check that \eqref{eq:leibniz} is a consequence of \eqref{eq:leibch}, which means in particular that the two constructions are compatible.

The fact that formula \eqref{eq:leibniz} is in form of inequalities rather than being an equality, is due to the fact that  with the approach we proposed `we don't really know who is the differential of a function before applying it to a gradient', and gradients in general are not uniquely defined. On the contrary, Cheeger's approach allows for a more intrinsic definition of differential, but it passes through the finite dimensionality of the cotangent space, a property which on arbitrary spaces certainly is not satisfied.
}\fr\end{remark}
\section{Laplacian}\label{se:lap}
\subsection{Definition and basic properties}
Having in mind the discussion made in Section \ref{se:normato} and the tools developed in the previous chapter, we will now discuss the definition of Laplacian. We start noticing that Proposition \ref{prop:srestr} allows to define the Sobolev class $\s^p_{\rm loc}(\Omega)$  in the following way:
\begin{definition}[The Sobolev class $\s^p_{\rm loc}(\Omega)$]
Let $(X,\sfd,\mm)$ be as in \eqref{eq:mms}, $\Omega\subset X$ an open set and $p\in(1,\infty)$. The space $\s^p_{\rm loc}(\Omega)$ is the space of Borel functions $g:\Omega\to\R$ such that $g\nchi\in\s^p_{\rm loc}(X,\sfd,\mm)$ for any Lipschitz function $\nchi:X\to[0,1]$ such that $\sfd(\supp(\nchi),X\setminus\Omega)>0$, where $g\nchi$ is taken 0 by definition on $X\setminus\Omega$.
\end{definition}
Here and in the following for $A,B\subset X$, we denote  by $\sfd(A,B)$ the value of $\inf\sfd(x,y)$, where the infimum is taken among all $x\in A$ and $y\in B$.
 
It is immediate to verify that this is a good definition, and that for $g\in\s^p_{\rm loc}(\Omega)$ the function $\weakgrad g\in L^p_{\rm loc}(\Omega,\mm\restr\Omega)$ is well defined, thanks to \eqref{eq:localgrad},  by
\begin{equation}
\label{eq:vabene}
\weakgrad g:=\weakgrad{(g\nchi)},\qquad\mm-a.e.\ \textrm{ on }\{\nchi=1\},
\end{equation}
where $\nchi:X\to[0,1]$ is any Lipschitz function  such that $\sfd(\supp(\nchi),X\setminus\Omega)>0$. Using the locality property \eqref{eq:localgrad2} it is easy to see that the objects $D^\pm f(\nabla g)$ are well defined $\mm$-a.e. on $\Omega$ and that the chain rules \eqref{eq:chainf}, \eqref{eq:chaing} and the Leibniz rule \eqref{eq:leibniz} hold $\mm$-a.e. on $\Omega$ for functions $f,g\in \s^p_{\rm loc}(\Omega)$, $p\in(1,\infty)$.

To define the Laplacian of $g$ we will need to ask for an integrability of $\weakgrad g$ which on non-proper spaces is slightly stronger than $L^p_{\rm loc}(\Omega,\mm\restr\Omega)$:
\begin{definition}[The class $\Int\Omega$ and `internal' Sobolev spaces]\label{def:int}
Let $(X,\sfd,\mm)$ be as in \eqref{eq:mms} and $\Omega\subset X$ an open set. The class $\Int\Omega$ is the collection of all open sets $\Omega'\subset\Omega$ such that
\begin{itemize}
\item[i)] $\Omega'$ is bounded,
\item[ii)] $\sfd(\Omega',X\setminus \Omega)>0$,
\item[iii)] $\mm(\Omega')<\infty$.
\end{itemize}
For $p\in(1,\infty)$ we define the space $\s^p_{\rm int}(\Omega)$ of functions Sobolev internally in $\Omega$ as 
\[
\begin{split}
\s^p_{\rm int}(\Omega):=\Big\{g\in\s^p_{\rm loc}(\Omega) \ :\ & \int_{\Omega'}\weakgrad{g}^p\,\d\mm<\infty,\ \forall \Omega'\in\Int\Omega\Big\}.
\end{split}
\]
\end{definition}
We also introduce the class $\test\Omega$ of test functions in $\Omega$ as follows.
\begin{definition}[Test functions]\label{def:testf}
Let $(X,\sfd,\mm)$ be as in \eqref{eq:mms} and $\Omega\subset X$ an open set. The space $\test\Omega$ is defined by
\[
\begin{split}
\test\Omega:=\Big\{f\in {\rm LIP}(X):\  \supp(f)\subset\Omega',\ \textrm{ for some }\ \Omega'\in\Int\Omega\}.
\end{split}
\]
\end{definition}
Recall that if $f:X\to\R$ is a Lipschitz function, then $f\in\s^p_{\rm loc}(X,\sfd,\mm)$ for any $p\in(1,\infty)$ and the $p$-minimal weak upper gradient $\weakgrad f$ is  uniformly bounded by $\Lip(f)$.

Clearly, for $g\in\s^p_{\rm int}(\Omega)$ and $f\in\test\Omega$ it holds $g+\eps f\in\s^p_{\rm int}(\Omega)$ for any $\eps\in\R$. Thus  for $g\in\s^p_{\rm int}(\Omega)$ and $f\in\test\Omega$, thanks to \eqref{eq:vabene} the functions $D^\pm f(\nabla g)$ are $\mm$-a.e. defined on $\Omega$ by
\begin{equation}
\label{eq:defrestr2}
D^\pm f(\nabla g):= D^\pm f(\nabla (g\nchi)),\qquad \textrm{ on }\{\nchi=1\},
\end{equation}
for any $\nchi:X\to[0,1]$ Lipschitz with $\supp(\nchi)\subset \Omega'$  for some $\Omega'\in\Int\Omega$.
From the bound \eqref{eq:1} and the fact that $\mm(\supp(f))<\infty$ we have $D^\pm f(\nabla g)\in L^1(\Omega,\mm\restr{\Omega})$.

We are now ready to give the definition of distributional Laplacian.
\begin{definition}[Laplacian]\label{def:lap}
Let $(X,\sfd,\mm)$ be as in \eqref{eq:mms},  $\Omega\subset X$ an open set and $g:\Omega\to\R$ a Borel function. We say that $g$ is in the domain of the Laplacian in $\Omega$, and write $g\in D(\bd,\Omega)$, provided $g\in\s^p_{\rm int}(\Omega)$ for some $p>1$ and there exists a Radon   measure $\mu$ on $\Omega$ such that for any $f\in\test\Omega\cap L^1(\Omega,|\mu|)$ it holds
\begin{equation}
\label{eq:deflap}
-\int_\Omega D^+f(\nabla g)\,\d\mm\leq\int_\Omega f\,\d\mu\leq-\int_\Omega D^-f(\nabla g)\,\d\mm.
\end{equation}
In this case we write $\mu\in \bd g\restr{\Omega}$. If $\Omega=X$ we simply write $g\in D(\bd)$ and $\mu\in \bd g$.
\end{definition}
We are using the bold font for the Laplacian to emphasize that in our discussion it  will always be a measure, possibly multivalued.

\begin{remark}\label{re:occhiolap}{\rm
As for the objects $\weakgrad f$ and $D^\pm f(\nabla g)$, the choice of the Sobolev exponent $p$ might affect the definition of Laplacian, but in our writing we will omit such explicit dependence (see also Remark \ref{re:notazione}). In this direction, notice that we require in the definition that the function $g$ belongs to $\s^p_{\rm int}(\Omega)$ for some $p>1$: the exponent $p$ is then the one chosen to compute  $\weakgrad{(g+\eps f)}$, which in turn  is the tool used to introduce $D^\pm f(\nabla g)$.
}\fr\end{remark}
\begin{remark}[A coincidence about a priori regularity]\label{re:apriori}{\rm  
We already observed in Section \ref{se:normato} that in a non-Riemannian situation, in order to define the distributional Laplacian of a function $g$, we need to impose some a priori Sobolev regularity on it. The minimal one is then that - in the appropriate local sense - $g\in\s^1(\Omega)$. Yet, we didn't define the Sobolev class $\s^1$, because its definition requires some care (for recent results in this direction, see \cite{Ambrosio-DiMarino}), so that we just asked for $g$ to locally belong to $\s^p$ for some $p>1$. 

Notice that not only this allows to write down $D^\pm f(\nabla g)$, but is in perfect duality with the integrability of $\weakgrad f$ for the smoothest functions available in this setting, i.e. Lipschitz $f$'s, because as we already remarked it holds $\weakgrad f\in L^\infty$. 

Here there is an interesting coincidence. Definition \ref{def:lap} differs from the one of distributional Laplacian  on the Euclidean setting $\R^d$, as in this framework in order for the expression $\Delta g=\mu$ to make sense, the only a priori condition that is imposed on $g$ is to be locally in $L^1$. Yet, standard regularity theory gives that if $\Delta g=\mu$ in the sense of distributions, then the distributional gradient $\nabla g$ is a function which locally belongs to $L^p$ for any $p\in[1,1+\frac1d)$. Therefore if we are in a Euclidean setting and we `don't know the dimension', we are anyway sure that $\Delta g=\mu$ implies $\nabla g\in L^p_{\rm loc}(X,\mm)$ for some $p>1$. 

This is precisely the a priori condition on $g$ we asked for in Definition \ref{def:lap}. Which means that we are, in practice, not asking more than what is already true in the classical situation. In particular, our definition reduces to the standard one in a Euclidean setting.
}\fr\end{remark}
A drawback of Definition \ref{def:lap} is that in general  the Laplacian might be non-unique: in Proposition \ref{prop:comp} below we will actually see that explicit counterexamples to uniqueness can be produced as soon as the space is not 2-infinitesimally strictly convex.  Yet, it is immediate to verify that the Laplacian is one-homogeneous, i.e.
\begin{equation}
\label{eq:unohom}
g\in D(\bd,\Omega),\ \mu\in \bd g\restr\Omega\qquad\Rightarrow\qquad \lambda g \in D(\bd,\Omega),\ \lambda\mu\in \bd (\lambda g)\restr\Omega,\quad\forall \lambda\in\R,
\end{equation}
invariant by addition of a constant, i.e.
\[
g\in D(\bd,\Omega),\ \mu\in \bd g\restr\Omega\qquad\Rightarrow\qquad c+g \in D(\bd,\Omega),\ \mu\in \bd (c+ g)\restr\Omega,\quad\forall c\in\R,
\]
that the set $\bd g\restr\Omega$ is convex and that it is weakly closed in the following sense: if $\mu_n\in \bd g\restr\Omega$ for any $n\in\N$ and $\mu$ is a Radon  measure on $\Omega$ such that $f\in\test\Omega\cap L^1(\Omega,|\mu|)$ implies $f\in L^1(\Omega,|\mu_n|)$ for $n$ large enough and
\[
\lim_{n\to\infty}\int f\,\d\mu_n=\int f\,\d\mu,
\]
then $\mu\in\bd g\restr\Omega$ (see also the proof of Proposition \ref{prop:laplstab}).

Another direct consequence of the definition is the following `global-to-local' property
\begin{equation}
\label{eq:globloc}
\widetilde\Omega\subset\Omega\textrm{ open sets, }g\in D(\bd,\Omega),\ \mu\in\bd g\restr\Omega\ \qquad\Rightarrow\qquad g\in D(\bd,\widetilde\Omega),\ \mu\restr{\widetilde\Omega}\in \bd g\restr{\widetilde\Omega},
\end{equation}
and the fact that measures in  $\bd g\restr\Omega$ are always concentrated on $\supp(\mm)$.

By the definition of $D^\pm f(\nabla g)$, we easily get the following sufficient condition for the Laplacian to be unique.
\begin{proposition}\label{prop:lapunique}
Let $(X,\sfd,\mm)$ be as in \eqref{eq:mms}, $\Omega\subset X$ an open set, $p\in(1,\infty)$, $g\in\s^p_{\rm int}(\Omega)\cap  D(\bd,\Omega)$. Assume that  $(X,\sfd,\mm)$ is $q$-infinitesimally strictly convex, $q$ being the conjugate exponent of $p$.

Then $\bd g\restr\Omega$ contains only one measure.
\end{proposition}
\begin{proof} By the definition \eqref{eq:defrestr2} and \eqref{eq:sc}, we deduce that for any $f\in\test\Omega$ it holds $D^+f(\nabla g)=D^-f(\nabla g)$ $\ \mm$-a.e. on $\Omega$. The thesis follows.
\end{proof}
The definition of Laplacian we just proposed is compatible with those available on Alexandrov spaces with curvature bounded from below, see Remark \ref{re:alex}.

We conclude this introduction by comparing our definition with the one appeared in \cite{Ambrosio-Gigli-Savare11}. Recall the definition given in \eqref{eq:cheegerbase} of the Cheeger energy functional $\c_2:L^2(X,\mm)\to[0,\infty]$, and that $\c_2$ is convex and lower semicontinuous.
\begin{definition}[Laplacian as defined in \cite{Ambrosio-Gigli-Savare11}]
Let $(X,\sfd,\mm)$ be as in \eqref{eq:mms} and $g\in L^2(X,\mm)$.  We say that $g$ is in the domain of the Laplacian provided $\c_2(g)<\infty$ and the subdifferential $\partial^-\c_2(g)$ of $\c_2$ at $g$ is non-empty. In this case the Laplacian of $g$ is the element of minimal $L^2(X,\mm)$ norm in (the closed convex set) $-\partial^-\c_2(g)$.
\end{definition}
\begin{proposition}[Compatibility]\label{prop:comp}
Let $(X,\sfd,\mm)$ be as in \eqref{eq:mms}, $g\in L^2(X,\mm)$ and $h\in L^2(X,\mm)$. Assume that $\c_2(g)<\infty$ and $-h\in\partial^-\c_2(g)$. Then  $g\in D(\bd)$ and $h\mm\in \bd g$.
\end{proposition}
\begin{proof}
Fix $f\in L^2(X,\mm)$ with $\c_2(f)<\infty$ and notice that the definition of subdifferential yields
\[
\c_2(g)-\eps\int fh\,\d\mm\leq \c_2(g+\eps f),\qquad\forall \eps\in\R.
\]
Hence for $\eps>0$ it holds
\[
-\int fh\,\d\mm\leq\int\frac{\weakgrad{(g+\eps f)}^2-\weakgrad g^2}{2\eps}\,\d\mm.
\]
Passing to the limit as $\eps\downarrow 0$ and noticing that the integrand at the right hand side is dominated, we get
\[
-\int fh\,\d\mm\leq\int D^+f(\nabla g)\,\d\mm.
\]
In particular, this holds for $f\in\test X\subset\{\c_2<\infty\}\cap L^1(X,|h|\mm)$. Replacing $f$ with $-f$ we also get the other inequality in \eqref{eq:deflap} and the conclusion.
\end{proof}
Therefore from the perspective of Definition \ref{def:lap} there is nothing special about the element of minimal norm in $-\partial^-\c_2(g)$, as any of its elements is (the density of) an admissible distributional Laplacian. This is not so surprising, as in \cite{Ambrosio-Gigli-Savare11} the choice of the element of minimal norm was done essentially only for cosmetic reasons, as it allowed to identify uniquely who the Laplacian was, and to write the equation of the gradient flow of $\c_2$ in the form $\frac{\d}{\d t}f_t=\Delta f_t$. Yet, the weak integration by part rules provided in \cite{Ambrosio-Gigli-Savare11} are valid, as the proof shows, for all the elements in $-\partial^-\c_2(g)$.

The main technical difference between the approach in \cite{Ambrosio-Gigli-Savare11} and the current one is then - beside the integrability assumption on $g$ and $\weakgrad g$ - that in \cite{Ambrosio-Gigli-Savare11} the Laplacian is well defined `only when it belongs to $L^2$', while here we are giving a meaning to the abstract distributional Laplacian, allowing it to be a measure.
\begin{remark}{\rm It is natural to ask if the converse of Proposition \ref{prop:comp} holds, i.e. if the following is true: suppose that $g\in D(\c_2)\cap D(\bd)$ and that for some $\mu\in\bd g$ it holds $\mu\ll\mm$ with density $h\in L^2(X,\mm)$. Can we deduce that $-h\in\partial^-\c_2(g)$? Without further assumptions, it is unclear to us whether this is true or not, the point being that we would like to deduce from
\[
\c_2(g)-\int fh\,\d\mm\leq \c_2(g+f),\qquad\forall f\in\test X,
\]
that
\[
\c_2(g)-\int fh\,\d\mm\leq \c_2(g+f),\qquad\forall f\in L^2(X,\mm).
\]
This implication seems to rely on a density result of Lipschitz functions in $W^{1,2}(X,\sfd,\mm)$ finer than the one expressed in Theorem \ref{thm:lipdense}. Indeed, we would like to know that for any $f,g\in D(\c_2)$ there exists a sequence of Lipschitz functions $(f_n)$ converging to $f$ in $L^2(X,\mm)$ such that $\c_2(g+f_n)$ converges to $\c_2(g+f)$ (Theorem \ref{thm:lipdense} ensures this fact only when $g$ is itself Lipschitz).

If $W^{1,2}(X,\sfd,\mm)$ is uniformly convex and $\mm$ is finite on bounded sets, then such approximation result is indeed true, because Corollary \ref{cor:lipdense} ensures  the density of $\test X$  in $W^{1,2}(X,\sfd,\mm)$, see for instance the argument in Proposition \ref{prop:complin} for the case in which $W^{1,2}$ is Hilbert.
}\fr\end{remark}

\subsection{Calculus rules}\label{se:lapcalc}
In this section we collect the basic calculus properties of the Laplacian.

\begin{proposition}[Chain rule]\label{prop:chainlapl} Let $(X,\sfd,\mm)$ be as in \eqref{eq:mms}, $\Omega\subset X$ an open set, $g\in D(\bd,\Omega)$. Assume that $g$  is Lipschitz on $\Omega'$ for any $\Omega'\in\Int\Omega$  and let $\varphi:g(\Omega)\to\R$ a $C^{1,1}_{\rm loc}$ map. Then $\varphi\circ g\in D(\bd,\Omega)$ and for any $\mu\in \bd g\restr{\Omega}$ it holds
\begin{equation}
\label{eq:chainlapl}
\bd(\varphi\circ g)\restr\Omega\ni\tilde\mu:= \varphi'\circ g\,\mu+\varphi''\circ g\,\weakgrad g^2\mm.
\end{equation}
\end{proposition}
\begin{proof}
Notice that since $g$ is continuous, the expression $\varphi'\circ g\,\mu$ makes sense and defines a locally finite measure.  Similarly, $\varphi''\circ g\,\weakgrad g^2\in L^\infty_{\rm loc}(\Omega,\mm\restr\Omega)$, so that $\tilde\mu$ is a locally finite measure and the statement makes sense. 

More precisely, $\varphi'\circ g$ is Lipschitz on $\Omega'$ and $\varphi''\circ g\,\weakgrad g^2$ is bounded on $\Omega'$ for any $\Omega'\in\Int\Omega$. Therefore if $f\in\test\Omega\cap L^1(\Omega,|\tilde\mu|)$, we have $f\varphi'\circ g\in\test\Omega\cap L^1(|\mu|)$. Fix such $f$, recall  the chain rules \eqref{eq:chainf}, \eqref{eq:chaing}, the Leibniz rule \eqref{eq:leibniz} and \eqref{eq:5} to get that $\mm$-a.e. on $\Omega$ it holds
\[
\begin{split}
D^+f\nabla(\varphi\circ g)&=\varphi'\circ g D^{{\rm sign}\varphi'\circ g}f\nabla g\geq D^+(f\varphi'\circ g)(\nabla g)-fD^{{\rm sign}f}(\varphi'\circ g)(\nabla g)\\
&= D^+(f\varphi'\circ g)(\nabla g)-f\varphi''\circ g D^{{\rm sign}(f\varphi''\circ g)}(g)(\nabla g)\\
&= D^+(f\varphi'\circ g)(\nabla g)-f\varphi''\circ g\weakgrad g^2.
\end{split}
\]
Integrating we obtain $\int D^+f(\nabla(\varphi\circ g))\,\d\mm\geq -\int f\,\d\tilde\mu$. Replacing $f$ with $-f$ we conclude.
\end{proof}
\begin{remark}{\rm
The assumption that $g$ was Lipschitz in $\Omega'$ was needed to ensure that $f\in\test\Omega$ implies $f\varphi'\circ g\in\test\Omega$, so that from assumptions one could deduce
\[
\int D^+(f\varphi'\circ g)(\nabla g)\,\d\mm\geq -\int f\,\d\mu.
\]
Yet, it would be more natural to have the chain rule to hold under the only assumption that $g\in \s^2_{\rm int}(\Omega)\cap C(\Omega)$, because these are the minimal regularity requirements needed for the right hand side of \eqref{eq:chainlapl} to make sense.

To get the chain rule under this weaker assumption on $g$, a suitable generalization of the integration by parts rule is needed, which looks linked to the uniform convexity of $W^{1,2}(X,\sfd,\mm)$. See for instance the proof of Lemma \ref{le:finalmente} for the case of proper spaces such that $W^{1,2}$ is Hilbert.
}\fr\end{remark}
The following proposition provides a sufficient condition on $g$ which ensures that it is in the domain of the Laplacian and gives a bound on elements of $\bd g\restr\Omega$. Technicalities apart, it can be seen as a rewording of the statement  `a non-negative distribution is always a non-negative measure' valid in a Euclidean context.
\begin{proposition}[Laplacian existence and comparison]\label{prop:comparison}
Let $(X,\sfd,\mm)$ be as in \eqref{eq:mms} and assume also that $(X,\sfd)$ is proper. Let $\Omega\subset X$ be an open set, $g\in \s^p_{\rm int}(\Omega)$, $p\in(1,\infty)$,   and $\tilde\mu$ a Radon  measure on $\Omega$. Assume that for any non-negative $f\in\test\Omega\cap L^1(\Omega,|\tilde\mu|)$  we have
\begin{equation}
\label{eq:boundalto}
-\int_\Omega D^-f(\nabla g)\,\d\mm\leq \int_\Omega f\,\d\tilde\mu.
\end{equation}
Then $g\in D(\bd,\Omega)$ and for any $\mu\in\bd g\restr\Omega$ it holds $\mu\leq\tilde\mu$.
\end{proposition}
\begin{proof} Consider the real valued map $\test\Omega\ni f\mapsto T(f):=-\int_\Omega D^-f(\nabla g)\,\d\mm$. From Proposition \ref{prop:convconc} we know that it  satisfies
\[
\begin{split}
T(\lambda f)&=\lambda T(f),\qquad\qquad\forall f\in \test\Omega,\ \lambda\geq 0,\\
T(f_1+f_2)&\leq T(f_1)+T(f_2),\qquad\forall f_1,f_2\in \test\Omega.
\end{split}
\]
Hence by the Hahn-Banach theorem there exists a linear map $L:\test\Omega\to\R$ such that $L(f)\leq T(f)$ for any $f\in \test\Omega$. By \eqref{eq:4} we then have
\[
-\int_\Omega D^+f(\nabla g)\,\d\mm\leq L(f)\leq -\int_\Omega D^-f(\nabla g)\,\d\mm,\qquad\forall f\in \test\Omega.
\]
Our goal is to prove that $L$ can be represented as integral w.r.t. some measure $\mu$. By \eqref{eq:boundalto} we get
\[
\int_\Omega f\,\d\tilde\mu-L(f)\geq 0,\qquad\forall f\in \test\Omega,\ f\geq 0.
\]
Fix a compact set $K\subset\Omega$ and a  function $\nchi_K\in \test\Omega$ such that $0\leq\nchi_K\leq 1$ everywhere and $\nchi_K= 1$ on $K$. Let $V_K\subset \test\Omega$ be the set of those functions with support contained in $K$ and observe that for any non-negative $f\in V_K$, the fact that $(\max f)\nchi_K-f$ is in $\test\Omega$ and non-negative yields
\[
\begin{split}
L(f)=-L((\max f)\nchi_K-f)+L((\max f)\nchi_K)\geq&  -\int(\max f) \nchi_K-f\,\d\tilde\mu+L((\max f)\nchi_K)\\
\geq &-(\max f)\big(\tilde\mu(\supp(\nchi_K))+L(\nchi_K)\big).
\end{split}
\]
Thus for a generic $f\in V_K$ it holds
\[
\begin{split}
L(f)=L(f^+-f^-)=L(f^+)-L(f^-)&\leq \int_\Omega f^+\,\d\tilde\mu+(\max f^-)\big(\tilde\mu(\supp(\nchi_K))+L(\nchi_K)\big)\\
&\leq(\max|f|)\Big(\tilde\mu(K)+\tilde\mu(\supp(\nchi_K))+L(\nchi_K)\Big),
\end{split}
\]
i.e. $L:V_K\to\R$ is continuous w.r.t. the $\sup$ norm. Hence it can be extended to a (unique, by the density of Lipschitz functions in the uniform norm) linear bounded functional on the set $C_K\subset C(X)$ of continuous functions with support contained in $K$. Since $K$ was arbitrary, by the Riesz theorem we get that there exists a Radon  measure $\mu$ such that
\[
L(f)=\int f\,\d\mu,\qquad\forall f\in \test\Omega.
\]
Thus $g\in D(\bd,\Omega)$ and $\mu\in\bd g\restr\Omega$.

Finally, from \eqref{eq:boundalto} it is immediate to get that for any $\mu'\in\bd g\restr\Omega$ it holds $\mu'\leq\tilde\mu$.
\end{proof}
\begin{remark}{\rm
In the proof of Proposition \ref{prop:comparison} we used the Hahn-Banach theorem, which means that some weak form of the Axiom of Choice (but still stronger than the Axiom of Dependent Choice) is required in the argument. In full generality, it seems hard to avoid such use of Hahn-Banach.

Yet, if furthermore   $(X,\sfd,\mm)$ is $q$-infinitesimally strictly convex and $g\in\s^p_{\rm int}(\Omega)$ - $p,q$ being conjugate exponents - then the map $T$ is linear and the proof goes on without any use of Choice. This is the situation we will work on when proving our main Laplacian comparison result in Theorem \ref{thm:controllo}.
}\fr\end{remark}
We now turn to the stability of $\bd g$ under convergence of $g$. Since in our definition of Laplacian we required some Sobolev regularity of $g$, in order to get stability we need to impose some sort of Sobolev convergence as well. This is clearly a much stronger assumption than the usual convergence in $L^1_{\rm loc}$ which is sufficient in the Euclidean context, but looks unavoidable in this general non-linear setting. 
\begin{proposition}[Stability]\label{prop:laplstab} Let $(X,\sfd,\mm)$ be as in \eqref{eq:mms}, $\Omega\subset X$ an open set and $p\in(1,\infty)$. Let $(g_n)\subset \s^p_{\rm int}(\Omega)$ be a sequence and $g\in \s^p_{\rm int}(\Omega)$ be such that  $\int_{\Omega'}\weakgrad{(g_n-g)}\,\d\mm\to 0$ for every $\Omega'\in\Int\Omega$. Assume also that $g_n\in D(\bd,\Omega)$ for every $n\in\N$, let $\mu_n\in \bd g_n\restr\Omega$ and suppose that for some locally finite measure $\mu$ on $\Omega$ it holds
\begin{equation}
\label{eq:assstrana}
f\in\test\Omega\cap L^1(\Omega,|\mu|)\qquad\Rightarrow\qquad 
\left\{\begin{array}{l}
f\in L^1(\Omega,|\mu_n|)\textrm{ for $n$ large enough and }\\
\phantom{ciccia}\\
\displaystyle{ \lim_{n\to\infty}\int f\,\d\mu_n=\int f\,\d\mu.}
\end{array}
\right.
\end{equation}
Then $g\in D(\bd,\Omega)$ and $\mu\in \bd g\restr\Omega$.
\end{proposition}
Note: if $(X,\sfd,\mm)$ is proper and the measures $\mu_n$ are equibounded - in total variation norm - on compact subsets of $\Omega$, then the convergence in \eqref{eq:assstrana} is nothing but weak convergence in duality with $C_c(\Omega)$. 
\begin{proof} The argument is similar to that of the proof of Proposition \ref{prop:convconc}. It is easy to check that for any $f\in\test\Omega$, $g\in\s^p_{\rm int}(\Omega)$ it holds
\[
\begin{split}
\int D^+f(\nabla g)\,\d\mm&=\inf_{\eps>0}\int\weakgrad g\frac{\weakgrad{(g+\eps f)}-\weakgrad g}\eps\,\d\mm,\\
\int D^-f(\nabla g)\,\d\mm&=\sup_{\eps<0}\int\weakgrad g\frac{\weakgrad{(g+\eps f)}-\weakgrad g}\eps\,\d\mm.
\end{split}
\]
By our assumptions we have that  for any $f\in\test\Omega$ it holds $\int_{\supp(f)}\weakgrad{(g_n-g)}\,\d\mm\to 0$,   the sequence $\frac{\weakgrad{(g_n+\eps f)}-\weakgrad {g_n}}\eps$ is uniformly bounded in $L^\infty(\Omega,\mm\restr\Omega)$ (by $\Lip(f)$) and converges to $\frac{\weakgrad{(g+\eps f)}-\weakgrad {g}}\eps$ in $L^1(\Omega,\mm\restr\Omega)$ for any $\eps\neq 0$. It easily follows that
\begin{equation}
\label{eq:checulo}
\begin{split}
\lims_{n\to\infty}\int_\Omega D^+f(\nabla g_n)\,\d\mm&\leq\int_\Omega D^+f(\nabla g)\,\d\mm,\\
\limi_{n\to\infty}\int_\Omega D^-f(\nabla g_n)\,\d\mm&\geq\int_\Omega D^-f(\nabla g)\,\d\mm.
\end{split}
\end{equation}
If furthermore $f\in L^1(\Omega,|\mu|)$,  \eqref{eq:assstrana} ensures that $f\in L^1(\Omega,|\mu_n|)$  for large $n$, so that from $\mu_n\in \bd g_n\restr\Omega$ we have
\[
-\int D^+f(\nabla g_n)\,\d\mm\leq \int f\,\d\mu_n\leq-\int D^-f(\nabla g_n)\,\d\mm,\qquad\forall n\gg 0.
\]
Using \eqref{eq:checulo} and \eqref{eq:assstrana} we can pass to the limit in these inequalities and get the thesis.
\end{proof}
The next statement shows that the definition of Laplacian is compatible with local minimizers of the Cheeger energy $\c_2$.
\begin{proposition}\label{prop:divano}
Let $(X,\sfd,\mm)$ be as in \eqref{eq:mms}, $\Omega\subset X$ an open set  and $g\in\s^2_{\rm int}(\Omega)$ with $\int_\Omega\weakgrad g^2\,\d\mm<\infty$. Assume that it holds
\begin{equation}
\label{eq:harm}
\int_\Omega\weakgrad g^2\,\d\mm\leq \int_\Omega\weakgrad{(g+h)}^2\,\d\mm,\qquad\forall h\in\s^2(X,\sfd,\mm) \textrm{ such that }\supp(h)\subset\Omega.
\end{equation}
Then $g\in D(\bd,\Omega)$ and $0\in \bd g\restr\Omega$
\end{proposition}
\begin{proof}
Notice that for $f\in\test\Omega$ we certainly have $\supp(f)\subset\Omega$ and $f\in \s^2(X,\sfd,\mm)$. Thus for   $f\in\test\Omega$ and $\eps\in\R$,  \eqref{eq:harm} yields
\[
\int_\Omega\weakgrad{(g+\eps f)}^2\,\d\mm\geq\int_\Omega\weakgrad{g}^2\,\d\mm.
\]
Therefore
\[
\begin{split}
\int_\Omega\frac{\weakgrad{(g+\eps f)}^2-\weakgrad g^2}{2\eps}\,\d\mm&\geq0,\qquad\qquad\forall\eps >0,\\
\int_\Omega\frac{\weakgrad{(g+\eps f)}^2-\weakgrad g^2}{2\eps}\,\d\mm&\leq0,\qquad\qquad\forall\eps <0.
\end{split}
\]
Letting $\eps\downarrow 0$ and $\eps\uparrow0$ we get the thesis.
\end{proof}
Another expected property of the Laplacian is the `local-to-global' property. Shortly said, one would like a theorem of the following form: if for some measure $\mu$ on $\Omega_1\cup\Omega_2$ it holds $\bd g\restr{\Omega_i}\ni\mu\restr{\Omega_i}$, $i=1,2$, then $\bd g\restr{\Omega_1\cup\Omega_2}\ni\mu$. Unfortunately, it is not clear - to us - whether this is true in the general case or not (actually, this is not so obvious  even in normed spaces, when the norm is not strictly convex). Yet, at least for infinitesimally strictly convex and proper spaces the result holds, and follows by a standard partition of the unit argument.
\begin{proposition}[Local-to-global]\label{prop:laplocal}
Let $(X,\sfd,\mm)$ be as in \eqref{eq:mms} and assume also that $(X,\sfd)$ is proper. Let $p\in(1,\infty)$ and $q$ the conjugate exponent, assume that $(X,\sfd,\mm)$ is $q$-infinitesimally strictly convex, let $\Omega\subset X$ be an open set and $\{\Omega_i\}_{i\in I}$ a family of open sets such that $\Omega=\cup_i\Omega_i$. Let $g\in \s^p_{\rm int}(\Omega)$, with  $g\in D(\bd,\Omega_i)$ for every $i\in I$ and let $\mu_i$ be the only element of $\bd g\restr{\Omega_i}$. Then
\begin{equation}
\label{eq:inters}
\mu_i\restr{\Omega_i\cap\Omega_j}=\mu_j\restr{\Omega_i\cap\Omega_j},\qquad\forall i,j\in I,
\end{equation}
$g\in D(\bd,\Omega)$ and the measure $\mu$ on $\Omega$ defined by
 \begin{equation}
\label{eq:defmu}
\mu\restr{\Omega_i}:=\mu_i,\qquad\forall i\in I,
\end{equation}
is the only element of $\bd g\restr\Omega$.
\end{proposition}
\begin{proof} Since $(X,\sfd)$ is proper, for any $\tilde\Omega\subset X$ open and Radon measure $\nu$ on $\tilde\Omega$ and any $f\in\test{\tilde\Omega}$, it holds $f\in L^1(\tilde\Omega,|\nu|)$, because $\supp(f)$ is compact.

Let $i,j\in I$,   $f\in\test{\Omega_i\cap\Omega_j}$  and notice that by definition it holds
\[
-\int_{\Omega_i\cap\Omega_j} f\,\d\mu_i=\int_{\Omega_i\cap\Omega_j} Df(\nabla g)\,\d\mm=-\int_{\Omega_i\cap\Omega_j} f\,\d\mu_j,
\]
which yields \eqref{eq:inters}. In particular, the measure $\mu$ is well defined by \eqref{eq:defmu}.

Fix $f\in\test\Omega$. Since $\supp(f)$ is compact,  there exists a finite set $I_f\subset I$ of indexes such that $\supp(f)\subset \cup_{i\in I_f}\Omega_i$. From the fact that $(X,\sfd)$ is proper it is easy to build a family $\{\nchi_i\}_{i\in I_f}$ of Lipschitz functions such that $\sum_{i\in I_f}\nchi_i\equiv 1$ on $\supp(f)$ and  $\supp(\nchi_i)$ is compact and contained in $\Omega_i$ for any $i\in I_f$. Hence $f\nchi_i\in\test{\Omega_i}$ for any $i\in I_f$ and  taking into account the linearity of the differential expressed in Corollary \ref{cor:dfgg}, we have
\[
\begin{split}
\int Df(\nabla g)\,\d\mm&=\int D\Big(\sum_{i\in I_f}\nchi_if\Big)(\nabla g)\,\d\mm=\sum_{i\in I_f}\int D(f\nchi_i)(\nabla g)\,\d\mm\\
&=-\sum_{i\in I_f}\int f\nchi_i\,\d\mu_i=-\int f\,\d\Big(\sum_{i\in I_f}\nchi_i\mu_i\Big),
\end{split}
\]
which is the thesis.
\end{proof}

We conclude the section discussing the effect of a change of the reference measure. As for the locality, the correct formula shows up under the assumption that the space is $q$-infinitesimally strictly convex.
\begin{proposition}[Change of the reference measure]
Let $(X,\sfd,\mm)$ be a metric measure space and $V:X\to\R$ a locally Lipschitz function, which is Lipschitz when restricted to bounded sets. Define the measure $\mm':=e^{-V}\mm$ and let $\bd'$ be the Laplacian in $(X,\sfd,\mm')$, where in Definition \ref{def:lap} we replace $\mm$ with $\mm'$. Let $\Omega\subset X$ be an open set, $g\in D(\bd)\cap \s^p_{\rm int}(\Omega)$ for some $p\in(1,\infty)$ and assume that $(X,\sfd,\mm)$ is $q$-infinitesimally strictly convex, $q$ being the conjugate exponent of $p$. Then $g\in D(\bd',\Omega)$ and the measure
\[
\mu':=e^{-V}\mu-DV(\nabla g)e^{-V}\mm,
\]
is the only element in $\bd'g\restr\Omega$, where here $\mu$ is the only element of $\bd g\restr\Omega$ (Proposition \ref{prop:lapunique}). 
\end{proposition}
\begin{proof} Let ${\rm Test'}(\Omega)$ be defined as $\test\Omega$ with the measure $\mm'$ replacing $\mm$ (the role of the measure is in the definition of $\Int\Omega$).

Since $e^{-V}$ and $\weakgrad V$ are locally bounded, $\mu'$ is a locally finite measure and the statement makes sense. For  $f\in{\rm Test'}(\Omega)\cap L^1(\Omega,|\mu'|)$, the fact that $\supp(f)$ is bounded yields that the restriction of $V$ to $\supp(f)$ is Lipschitz and bounded. It easily follows that $fe^{-V}\in\test\Omega\cap L^1(\Omega,|\mu|)$, so that from the chain rule \eqref{eq:chainfs} and the Leibniz rule \eqref{eq:leibs} we get
\[
\begin{split}
\int f\,\d\mu'&=\int fe^{-V}\,\d\mu-\int f DV(\nabla g)e^{-V}\,\d\mm\\
&=-\int D(fe^{-V})(\nabla g)-f D(e^{-V})(\nabla g)\,\d\mm=-\int e^{-V}Df(\nabla g)\,\d\mm,
\end{split}
\]
which is the thesis.
\end{proof}

\subsection{The linear case}\label{se:lineare}
In this section we introduce a sufficient, and in most cases also necessary, condition in order for the Laplacian to be linear, and analyze its properties also in connection with Ricci curvature bounds.

\begin{definition}[Infinitesimally Hilbertian spaces]
We say that $(X,\sfd,\mm)$ is infinitesimally Hilbertian provided it is as in \eqref{eq:mms} and the seminorm $\|\cdot\|_{\s^2(X,\sfd,\smm)}$ on $\s^2(X,\sfd,\smm)$ satisfies the parallelogram rule.
\end{definition}
In particular, infinitesimally Hilbertian spaces are 2-infinitesimally strictly convex, hence for any $f,g\in\s^2_{\rm loc}(X,\sfd,\mm)$ the object $Df(\nabla g)$ is well defined $\mm$-a.e., and for $g\in D(\bd)\cap\s^2(X,\sfd,\mm)$ the set $\bd g$ contains only one element, which - abusing a bit the notation - we will denote by $\bd g$.

\begin{remark}[About stability]{\rm In the class of compact and normalized (i.e. $\mm\in\prob X$) metric measure spaces, the class of infinitesimally Hilbertian spaces is definitively \emph{not} closed w.r.t. measured Gromov-Hausdorff convergence. Actually it is dense, because any compact normalized space is the mGH limit of a sequence of discrete spaces, and on discrete spaces Sobolev analysis trivializes, i.e. $\weakgrad f\equiv 0$ for any $f$. The fact that infinitesimal Hilbertianity is not a closed condition w.r.t. mGH convergence is not at all surprising, because the former is a first order condition on the space, while the latter is a zeroth order convergence.

It is worth underlying, however, that the class of spaces which are both infinitesimally Hilbertian and $\CD(K,N)$ for some $K\in\R$, $N\in(1,\infty]$, \emph{is} closed w.r.t. mGH convergence. This can be seen - at the very heuristic level - as the stability of a first order condition w.r.t. a zeroth order convergence under a uniform second order bound (the curvature bound). The proof of this stability passes through the identification of the gradient flow of $\c_2$ in $L^2$ with the one of the relative entropy in the Wasserstein space (proved in \cite{Ambrosio-Gigli-Savare11}) and the stability of gradient flows of a sequence of $K$-geodesically convex functionals which $\Gamma$-converges (proved in \cite{Gigli10}) along the following lines:
\begin{itemize}
\item[i)] $(X,\sfd,\mm)$ is infinitesimally Hilbertian if and only if $\c_2$ is a Dirichlet form (obvious, see Proposition \ref{prop:ovvio} below and the discussion thereafter for the details)
\item[ii)] $\c_2$ is a Dirichlet form if and only if its  gradient flow linearly depends on the initial datum (obvious)
\item[iii)] The gradient flow of $\c_2$ in $L^2$ coincides with the gradient flow of the relative entropy in the Wasserstein space (Theorem 9.3 of \cite{Ambrosio-Gigli-Savare11})
\item[iv)] If $(X_n,\sfd_n,\mm_n)$ converges to $(X,\sfd,\mm)$ in the mGH sense, then the relative entropies $\Gamma$-converge (this is the heart of the proof of stability of the $\CD(K,\infty)$ condition proved, up to minor technical differences, by Lott-Villani in \cite{Lott-Villani09} and Sturm in \cite{Sturm06I})
\item[v)] If a sequence of $K$-geodesically convex functionals $\Gamma$-converges to a limit functional, then gradient flows of the approximating sequence converge to gradient flows of the limit functional (Theorem 21 of \cite{Gigli10})
\end{itemize}
So that in the end we know that on the limit space the gradient flow of the relative entropy linearly depends on the initial datum and we conclude using again $(iii),(ii),(i)$. 

We remark that the 
combination of $(iv)$ and $(v)$ above is sufficient to prove that \emph{a compact smooth Finsler manifold is the mGH-limit of a sequence of Riemannian manifolds with Ricci curvature 
uniformly bounded from below if and only if it is Riemannian itself}. To prove it just recall that the heat flow on a Finsler manifold is linear if and only if it is Riemannian and that $(iv)$, $(v)$ yield that 
the heat flow passes to the limit under a uniform lower Ricci bound. 

The fact that Finsler manifolds cannot arise as limits of Riemannian manifolds with a uniform Ricci bound from below is certainly not a new  result, at least under the additional assumption that the dimension of the approximating sequence is uniformly bounded: in \cite{Cheeger-Colding97I}, \cite{Cheeger-Colding97II}, \cite{Cheeger-Colding97III} Cheeger and Colding  proved the 
same thing (and much more). The proof we propose has the advantage of being simple and of avoiding the use of the almost splitting theorem, although clearly to (hope to) recover the results in \cite{Cheeger-Colding97I}, \cite{Cheeger-Colding97II}, \cite{Cheeger-Colding97III}  in the non-smooth context a lot of work has still to be done.

It is natural to think that the right condition to add to $\CD(K,N)$ spaces in order to enforce a Riemannian-like behavior on small scales and to rule out Finsler geometries, is infinitesimal Hilbertianity. Yet, from the  technical point of view it is not so clear whether this is sufficient or not to derive the expected properties of the heat flow (e.g., $K$-contractivity of $W_2$ along two flows): the first attempt in this direction has been done in \cite{Ambrosio-Gigli-Savare11bis}, where infinitesimal Hilbertianity was enforced with a notion of $K$-geodesic convexity of the entropy stronger than the standard one (see Definition \ref{def:rcd} below and Theorem \ref{thm:rcd} for a - non-complete - overview of the results of \cite{Ambrosio-Gigli-Savare11bis}). Recent results by Rajala (see \cite{Rajala11bis}) suggest that infinitesimal Hilbertianity plus $\CD(K,\infty)$ are actually sufficient to recover all the properties obtained in  \cite{Ambrosio-Gigli-Savare11bis}, without any strengthening of the curvature condition, but the situation, as of today, is not yet completely clear\footnote{this turned out to be true: in the recent paper \cite{AmbrosioGigliMondinoRajala12} it has been shown that infinitesimal Hilbertianity plus $\CD(K,\infty)$ is sufficient to deduce all the properties of the heat flow stated in  \cite{Ambrosio-Gigli-Savare11bis}}.
}\fr\end{remark}
We start proving that  infinitesimal Hilbertianity can be checked by looking at the symmetry of  the map $\s^2\ni f,g\mapsto Df(\nabla g)$. The proof follows the same lines  of Section 4.3 of \cite{Ambrosio-Gigli-Savare11bis}.
\begin{proposition}\label{prop:simm}
Let $(X,\sfd,\mm)$ be as in \eqref{eq:mms}. Then it is an infinitesimally Hilbertian space if and only if  it is 2-infinitesimally strictly convex and  for every $f,g\in\s^2_{\rm loc}(X,\sfd,\mm)$ it holds
\begin{equation}
\label{eq:simmetrico}
Df(\nabla g)=Dg(\nabla f),\qquad\mm-a.e..
\end{equation}
\end{proposition}
\begin{proof} 
Assume at first that the space is 2-infinitesimally strictly convex and that \eqref{eq:simmetrico} holds. Fix $f,g\in\s^2(X,\sfd,\mm)$ and notice that
\[
\begin{split}
\|g+f\|^2_{\s^2}-\|g\|^2_{\s^2}&=\|\weakgrad{(g+f)}\|^2_{L^2}-\|\weakgrad g\|^2_{L^2}=\int_0^1\frac{\d}{\d t}\|\weakgrad{(g+tf)}\|^2_{L^2}\,\d t\\
&=2\iint_0^1Df(\nabla(g+tf))\,\d t\,\d\mm\stackrel{\eqref{eq:simmetrico}}=2\iint_0^1D(g+tf)(\nabla f )\,\d t\,\d\mm\\
&= 2\int Dg(\nabla f)\,\d\mm+\|  f\|^2_{\s^2},
\end{split}
\]
having used the linearity of the differential (Corollary \ref{cor:dfgg}) in the last step. Exchange $f$ with $-f$ and add up to conclude.

We pass to the converse implication. With a cut-off and truncation argument we can reduce to the case $f,g\in \s^2(X,\sfd,\mm)\cap L^\infty(X,\mm)$.
The fact that $\|\cdot\|_{\s^2}$ satisfies the parallelogram rule easily yields
\[
\frac{\|g+\eps f\|^2_{\s^2}-\|g\|^2_{\s^2}}{\eps}=\frac{\|f+\eps g\|^2_{\s^2}-\|f\|^2_{\s^2}}{\eps}+O(\eps),\qquad\forall f,g\in\s^2(X,\sfd,\mm),
\] 
so that  by the very definition of $Df(\nabla g)$ and \eqref{eq:5} we get
\begin{equation}
\label{eq:simmint}
\int Df(\nabla g)\,\d\mm=\int Dg(\nabla f)\,\d\mm,\qquad\forall f,g\in \s^2(X,\sfd,\mm).
\end{equation}
We want to pass from this integrated equality, to the pointwise statement \eqref{eq:simmetrico}.   The thesis will follow if we prove that 
\[
\s^2(X,\sfd,\mm)\cap L^\infty(X,\mm)\ni f\qquad\mapsto\qquad\int h\weakgrad f^2\,\d\mm\ \in\R
\]
is a quadratic form for any $h\in \s^2(X,\sfd,\mm)\cap L^\infty(X,\mm)$. Fix such $h$ and use \eqref{eq:5}, the Leibniz rule \eqref{eq:leibs} and the chain rule  to get
\[
\begin{split}
\int h\weakgrad f^2\,\d\mm&=\int hDf(\nabla f)\,\d\mm=\int D(hf)(\nabla f)- fDh(\nabla f)\,\d\mm\\
&=\int D(hf)(\nabla f)\,\d\mm-\frac12\int Dh(\nabla(f^2))\,\d\mm\\
&=\int D(hf)(\nabla f)\,\d\mm-\frac12\int D(f^2)(\nabla h)\,\d\mm,
\end{split}
\]
having used \eqref{eq:simmint} in the last step. Now recall from Corollary \ref{cor:dfgg} that the map $\s^2(X,\sfd,\mm)\ni g\mapsto \int Dg(\nabla h)\,\d\mm$ is linear, hence the map $\s^2(X,\sfd,\mm)\cap L^\infty(X,\mm)\ni f\mapsto\int D(f^2)(\nabla h)\,\d\mm$ is a quadratic form. Similarly, from the fact that both the maps $\s^2(X,\sfd,\mm)\cap L^\infty(X,\mm)\ni g\mapsto \int D(hg)(\nabla f)\,\d\mm$ and $\s^2(X,\sfd,\mm)\ni g\mapsto \int D(hf)(\nabla g)\,\d\mm=\int Dg(\nabla(hf))\,\d\mm$ are linear, we get that $\s^2(X,\sfd,\mm)\cap L^\infty(X,\mm)\ni f\mapsto\int D(hf)(\nabla f)\,\d\mm$ is a quadratic form.
\end{proof}
On infinitesimally Hilbertian spaces we will denote $D f(\nabla g)$ by $\nabla f\cdot\nabla g$, in order to highlight the symmetry of this object. A direct consequence of \eqref{eq:simmetrico} and the linearity of the differential expressed in Corollary \ref{cor:dfgg}, is the bilinearity of $\nabla f\cdot\nabla g$, i.e.
\begin{equation}
\label{eq:bilin}
\begin{split}
\nabla(\alpha_1f_1+\alpha_2f_2)\cdot\nabla g&=\alpha_1\nabla f_1\cdot\nabla g+\alpha_2\nabla  f_2\cdot\nabla g,\qquad\forall f_1,f_2,g\in \s^2_{\rm loc},\ \alpha_1,\alpha_2\in\R,\\
\nabla f\cdot\nabla (\beta_1g_1+\beta_2g_2)&=\beta_1\nabla f\cdot\nabla g_1+\beta_2\nabla  f\cdot\nabla g_2,\qquad\forall f,g_1,g_2\in \s^2_{\rm loc},\ \beta_1,\beta_2\in\R,
\end{split}
\end{equation}
and from  \eqref{eq:leibs} we also get
\begin{equation}
\label{eq:leibb}
\begin{split}
\nabla(f_1f_2)\cdot\nabla g&=f_1\nabla f_2\cdot\nabla g+f_2\nabla f_1\cdot\nabla g,\qquad\forall f_1,f_2\in \s^2_{\rm loc}\cap L^\infty_{\rm loc},\ g\in\s^2_{\rm loc},\\
\nabla f\cdot\nabla(g_1g_2)&=g_1\nabla f\cdot\nabla g_2+g_2\nabla f\cdot\nabla g_1,\qquad\forall f\in \s^2_{\rm loc},\ g_1,g_2\in\s^2_{\rm loc}\cap L^\infty_{\rm loc},
\end{split}
\end{equation}
all the equalities being intended $\mm$-a.e..

The following proposition collects some different characterizations of infinitesimally Hilbertian spaces. 
\begin{proposition}\label{prop:ovvio}
Let $(X,\sfd,\mm)$ be as in \eqref{eq:mms}. Then the following are equivalent.
\begin{itemize}
\item[i)] $(X,\sfd,\mm)$ is infinitesimally Hilbertian.
\item[ii)] For any $\Omega\subset X$ open with $\mm(\partial\Omega)=0$ the space $(\overline\Omega,\sfd,\mm\restr\Omega)$ is infinitesimally Hilbertian.
\item[iii)] $W^{1,2}(X,\sfd,\mm)$ is an Hilbert space.
\item[iv)] The Cheeger energy $\c_2:L^2\to[0,\infty]$ is a quadratic form, i.e.
\begin{equation}
\label{eq:parallelogramma}
\c_2(f+g)+\c_2(f-g)=2\Big(\c_2(f)+\c_2(g)\Big),\qquad\forall f,g\in L^2(X,\mm).
\end{equation}
\end{itemize}
\end{proposition}
\begin{proof}\\*
\noindent$\mathbf{(i)\Rightarrow (ii)}$ Follows from  Proposition \ref{prop:srestr} and the equivalence stated in Proposition \ref{prop:simm}.\\
\noindent$\mathbf{(ii)\Rightarrow (i)}$ Obvious: just take $\Omega:=X$.\\
\noindent$\mathbf{(i)\Rightarrow (iii)}$ Obviously follows from the definition of the  $W^{1,2}$ norm.\\
\noindent$\mathbf{(iii)\Rightarrow (i)}$ We know that $W^{1,2}(X,\sfd,\mm)\ni f\mapsto \|\weakgrad f\|^2_{L^2}$ is a quadratic form, and we want to prove that the same its true for its extension to $\s^2(X,\sfd,\mm)$. Arguing exactly as in the proof of Proposition \ref{prop:simm} above we have that 
\begin{equation}
\label{eq:wq}
W^{1,2}(X,\sfd,\mm)\ni f\qquad\mapsto\qquad \weakgrad f^2\in L^1(X,\mm),\qquad\textrm{ is a quadratic form}.
\end{equation}
Now use the local finiteness of $\mm$ and the Lindel\"of property of $(X,\sfd)$ to build an increasing sequence $(K_n)$ of compact sets such that $\mm(X\setminus \cup_nK_n)=0$ and an increasing sequence $(\nchi_n)$ of Lipschitz bounded functions such that $\mm(\supp(\nchi_n))<\infty$ and $\nchi_n\equiv 1$ on $K_n$ for every $n\in\N$. Then for every $f\in \s^2(X,\sfd,\mm)\cap L^\infty(X,\mm)$ and $n\in\N$ it holds $f\nchi_n\in W^{1,2}(X,\sfd,\mm)$. Since $\weakgrad{f}=\weakgrad{(f\nchi_n)}$ $\mm$-a.e. on $K_n$, from \eqref{eq:wq} and letting $n\to\infty$ we get that $\s^2(X,\sfd,\mm)\cap L^\infty(X,\mm)\ni f\mapsto \weakgrad f^2\in L^1(X,\mm)$ is a quadratic form as well. With a truncation argument we get that also $\s^2(X,\sfd,\mm)\ni f\mapsto \weakgrad f^2\in L^1$ is a quadratic form. Integrate to conclude.\\
\noindent$\mathbf{(iii)\Leftrightarrow (iv)}$ The formula $\|f\|_{W^{1,2}}^2=\|f\|_{L^2}^2+2\c_2(f)$ shows that  $\c_2$ satisfies the parallelogram rule if and only if   so does the $W^{1,2}$ norm.
\end{proof}

A property of  particular importance of infinitesimally Hilbertian spaces is that the Laplacian is linear, as explained in the following proposition. We will denote by $D_{{\rm ifm}}(\bd,\Omega)\subset D(\bd,\Omega)$ the set of those $g$'s whose Laplacian has `internally finite mass', i.e.   for any   $\mu\in \bd g\restr\Omega$ it holds $|\mu|(\Omega')<\infty$ for any $\Omega'\in\Int \Omega$. If $(X,\sfd)$ is proper, then $D_{{\rm ifm}}(\bd,\Omega) =D(\bd,\Omega)$.
\begin{proposition}[Linearity of the Laplacian]
Let $(X,\sfd,\mm)$ be an infinitesimally Hilbertian space. Then for any $\Omega\subset X$ open, the set $D_{{\rm ifm}}(\bd,\Omega)\cap \s^2_{\rm int}(\Omega)$ is a vector space and for $g\in D_{{\rm ifm}}(\bd,\Omega)\cap \s^2_{\rm int}(\Omega)$, $\bd g\restr\Omega$ is single valued and linearly depends on $g$.
\end{proposition}
\begin{proof}
Fix $\Omega\subset X$ open. Being $(X,\sfd,\mm)$ 2-infinitesimally strictly convex, by Proposition \ref{prop:lapunique} we  know that $\bd g\restr\Omega$ is single valued for any $g\in\s^2_{\rm int}(\Omega)\cap D(\bd,\Omega)$. With a cut-off argument   we deduce that \eqref{eq:bilin} is satisfied also for functions in $\s^2_{\rm int}(\Omega)$. Pick $g_1,g_2\in D_{{\rm ifm}}(\bd,\Omega)\cap \s^2_{\rm int}(\Omega)$, $\beta_1,\beta_2\in\R$. Observe that $|\bd g_i|(\supp(f))<\infty$ for any $f\in\test\Omega$, so that in particular $\test\Omega\subset L^1(\Omega,|\bd g\restr\Omega|)$, $i=1,2$. Also, it holds  $\test\Omega\subset \s^2(X,\sfd,\mm)$. Let $f\in\test\Omega$ and conclude with
\[
\begin{split}
\int_\Omega f\,\d(\beta_1\bd g_1\restr\Omega+\beta_2\bd g_2\restr\Omega)&=\beta_1\int_\Omega f\,\d\bd g_1\restr\Omega+\beta_2\int_\Omega f\,\d\bd g_2\restr\Omega\\
&=-\beta_1\int_\Omega \nabla f\cdot \nabla g_1\,\d\mm-\beta_2\int_\Omega \nabla f\cdot \nabla g_2\,\d\mm\\
&=-\int_\Omega\nabla f\cdot\nabla (\beta_1g_1+\beta_2 g_2)\,\d\mm.
\end{split}
\]
\end{proof}
The fact that on infinitesimally Hilbertian spaces the Cheeger energy $\c_2$ satisfies \eqref{eq:parallelogramma} yields that the formula
\[
\mathcal E(f,g):=\c_2(f+g)-\c_2(f)-\c_2(g),
\]
defines a symmetric bilinear form on $W^{1,2}(X,\sfd,\mm)$. The chain rule \eqref{eq:chaineasy} gives that the form is Markovian and the semicontinuity of $\c_2$ means that the form is closed. Hence, $\mathcal E$ is (more precisely, can be extended to) a Dirichlet form on $L^2(X,\mm)$. As a consequence of the locality principle \eqref{eq:localgrad}, $\mathcal E$ is strongly local. The generator of $\mathcal E$ will be denoted by $\Delta$, so that $\Delta:D(\Delta)\subset L^2(X,\mm)\to L^2(X,\mm)$ and $\Delta$ and its domain $D(\Delta)$ are defined by
\begin{equation}
\label{eq:defdelta}
g\in D(\Delta),\ h=\Delta g\qquad\textrm{ if and only if }\qquad \mathcal E(f,g)=-\int fh\,\d\mm,\qquad\forall f\in D(\mathcal E) .
\end{equation}
Notice that $D(\mathcal E)=D(\c_2)=W^{1,2}(X,\sfd,\mm)$.

In the following proposition we compare the various definitions of `the Laplacian of $g$ exists and is in $L^2$'. Notice that the results are sharper than those of Proposition \ref{prop:comp}.
\begin{proposition}[Compatibility]\label{prop:complin}
Let $(X,\sfd,\mm)$ be an infinitesimally Hilbertian space, $g\in W^{1,2}(X,\sfd,\mm)$ and $h\in L^2(X,\mm)$. Consider the  following statements:
\begin{itemize}
\item[i)] $g\in D(\bd)$ and $\bd g=h\mm$,
\item[ii)]  $h\in -\partial^-\c_2(g)$,
\item[iii)] $g\in D(\Delta)$ and $\Delta g=h$.
\end{itemize}
Then
\[
(ii)\quad\Leftrightarrow \quad(iii)\quad\Rightarrow \quad (i),
\]
and if $\mm$ is finite on bounded sets it also holds
\[
(i) \quad \Rightarrow \quad  (ii),(iii).
\]
\begin{proof}\\
\noindent $\mathbf{ (iii)\Rightarrow (ii).}$ We need to prove that for any $f\in L^2(X,\mm)$ it holds
\begin{equation}
\label{eq:iiiaii}
\c_2(g)-\int (f-g)h\,\d\mm\leq \c_2(f). 
\end{equation}
If $f\notin D(\c_2)=D(\mathcal E)$ there is nothing to prove. Otherwise $f-g$ is in $D(\mathcal E)$ and we can use the definition of $\Delta$ to get
\[
-\int  (f-g)h\,\d\mm=\mathcal E(f-g,g)=\mathcal E(f,g)-\mathcal E(g,g)\leq\frac12\mathcal E(f,f)-\frac12\mathcal E(g,g),
\]
which is \eqref{eq:iiiaii}.\\
\noindent$\mathbf {(ii)\Rightarrow (iii})$. Pick $f\in D(\mathcal E)$ and notice that by definition of $\partial^-\c_2(g)$ it holds
\[
\c_2(g)-\int \eps f h\,\d\mm\leq\c_2(g+\eps f),\qquad\forall \eps\in\R.
\]
Conclude observing that $\c_2(g+\eps f)=\c_2(g)+\eps^2\c_2(f)+\mathcal E(f,g)$, dividing by $\eps$ and letting $\eps\downarrow0$ and $\eps\uparrow 0$.\\
\noindent$\mathbf {(ii)\Rightarrow (i})$. This is a particular case of Proposition \ref{prop:comp}.\\
\noindent$\mathbf{ (i)\Rightarrow (iii)}$. Here we assume that $\mm$ is finite on bounded sets. We know that
\[
-\int\nabla f\cdot\nabla g\,\d\mm=\int fh\,\d\mm,\qquad\forall f\in\test X\cap L^1(X,|h|\mm)
\]
and we want to conclude that the same is true for any $f\in W^{1,2}(X,\sfd,\mm)$. Pick $f\in W^{1,2}(X,\sfd,\mm)$ and assume for the moment that $\supp(f)$ is bounded. Let $\nchi$ be a Lipschitz bounded function with bounded support and identically 1 on   $\supp(f)$. Also, let $(f_n)$ be a sequence of Lipschitz functions converging to $f$ in $W^{1,2}(X,\sfd,\mm)$ (Corollary \ref{cor:lipdense}). Then $f_n\nchi\to f$ in $W^{1,2}(X,\sfd,\mm)$ and $f_n\nchi\in\test X\cap L^1(|h|\mm)$  for any $n\in\N$. Thus passing to the limit in $-\int\nabla( f_n\nchi)\cdot\nabla g\,\d\mm=\int f_n\nchi h\,\d\mm$ we get $-\int\nabla f\cdot\nabla g\,\d\mm=\int fh\,\d\mm$ for any $f\in W^{1,2}(X,\sfd,\mm)$ with bounded support.

To achieve the general case, let $(\nchi_n)$ be an increasing sequence of non-negative, 1-Lipschitz functions with bounded support and such that  $\nchi_n\equiv 1$ on $B_n(x_0)$, where $x_0\in X$ is some fixed point, for any $n\in\N$. Then fix an arbitrary $f\in W^{1,2}(X,\sfd,\mm)$, notice that $f\nchi_n\in W^{1,2}(X,\sfd,\mm)$ as well and has bounded support. We claim that $f\nchi_n\to f$ in $W^{1,2}(X,\sfd,\mm)$. Indeed, by the dominate convergence theorem we immediately get that $\|f-f\nchi_n\|_{L^2}\to 0$ as $n\to\infty$. Also, we have $\weakgrad{(f-f\nchi_n)}\to 0$ $\mm$-a.e. as $n\to\infty$ and $\weakgrad{(f-f\nchi_n)}\leq\weakgrad f|1-\nchi_n|+|f|$ $\mm$-a.e., so that again by dominate convergence we get $\weakgrad{(f-f\nchi_n)}\to 0$ in $L^2(X,\mm)$ as $n\to\infty$. Conclude by letting $n\to\infty$ in
\[
-\int\nabla (f\nchi_n)\cdot\nabla g\,\d\mm=\int f\nchi_nh\,\d\mm.
\]
\end{proof}
\end{proposition}
\begin{remark}[Compatibility with the theory of Alexandrov spaces]\label{re:alex}{\rm
In  \cite{KMS}, using key results of \cite{OtsuShioya} and \cite{PerelmanDC}, it has been proved that finite dimensional Alexandrov spaces with curvature bounded from below are infinitesimally Hilbertian.

On these spaces there are four - compatible - ways of defining a Laplacian:
\begin{itemize}
\item[i)] To use the fact that $W^{1,2}$ is Hilbert to build a natural Dirichlet form and then considering its generator.
\item[ii)] To use the DC structure of the space to define the Laplacian of DC$^1$ functions on $X\setminus S_\delta$  by direct calculus. Recall that DC functions are functions which are locally the Difference of Concave functions, DC$^1$ functions are also $C^1$, that a DC chart is a chart made of maps with DC components and that Perelman showed in \cite{PerelmanDC} that `most' of an Alexandrov space (i.e. all of it with the exception of a small singular set $S_\delta$) can be covered by DC charts.
\item[iii)] To use the Riemannian structure of the space to give a meaning to the integration by parts formula for functions which are not necessarly DC.
\item[iv)] To equip these spaces with their natural Hausdorff measure $\mathcal H^n$ and use the definition given here. Notice that the resultant metric measure space is always doubling and supports a 1-1 Poincar\'e inequality, hence in this case there is no ambiguity due to the choice of the Sobolev exponent $p$ in the definition of distributional Laplacian (see also Remarks \ref{re:notazione} and \ref{re:occhiolap}).
\end{itemize}
For the first two approaches  see \cite{KMS}. Due to Proposition \ref{prop:complin}, the approach $(i)$  is clearly in  line with $(iv)$. The same is true for the approach $(ii)$, because $(iv)$  clearly reduces to the standard definition of distributional Laplacian on Riemannian manifolds equipped with a metric tensor locally uniformly elliptic, which is the case for DC manifolds.

The approach $(iii)$ was introduced by Petrunin (\cite{Pet}) for the case of DC functions. To define the Laplacian of these functions, he used the integration by parts rule: for $g$ which is DC on some open domain $\Omega$ he proposed
\begin{equation}
\label{eq:pet}
\int_{\Omega} f\,\d\bd g:=-\int_\Omega \la{\rm grad}\, f,{\rm grad}\, g\ra \,\d\mathcal H^n,\qquad\forall f:\overline\Omega\to\R \textrm{ Lipschitz with }f\restr{\partial\Omega}=0,
\end{equation}
where $n$ is the dimension and $\mathcal H^n$ the corresponding Hausdorff measure and the object $ \la{\rm grad}\, f,{\rm grad}\, g\ra $ is defined $\mathcal H^n$-a.e. via the use of DC charts.  He used this definition of Laplacian to prove that  for DC functions the measure $\bd$ actually exists and that it is non-negative for convex functions. 

To see that  \eqref{eq:pet} defines the same measure as the one given in \eqref{def:lap}, start observing that  the requirements  `$f\restr{\partial\Omega}=0$' and `$\supp(f)\subset\subset\Omega$' are easily seen to be interchangeable in this setting, so that it is sufficient to prove that $\mathcal H^n$-a.e. it holds $\la{\rm grad}\, f,{\rm grad}\, g\ra=\nabla f\cdot\nabla g$. This follows from the fact that $\mathcal H^n$-almost all the space can be covered by DC charts, and that on these charts the two objects are clearly the same.
}\fr\end{remark}
In the next two lemmas we prove that the higher regularity of $g$ and $\bd g$ is, the wider is the class of functions for which the integration by part rules holds.
\begin{lemma}\label{le:finalmente}
Let $(X,\sfd,\mm)$ be a proper infinitesimally Hilbertian space, $\Omega\subset X$ an open set and $g\in D(\bd,\Omega)\cap\s^2_{\rm int}(\Omega)$. Then for every   $\psi\in \s^2(X,\sfd,\mm)\cap C_c(X)$ with support contained in $\Omega$ it holds
\[
-\int_\Omega \nabla\psi\cdot\nabla g\,\d\mm=\int_\Omega \psi\,\d\bd g\restr\Omega.
\]
\end{lemma}
\begin{proof}
Since $(X,\sfd)$ is proper, $\supp(\psi)$ is compact and therefore $|\bd g\restr\Omega|(\supp(\psi))<\infty$. Similarly, we have  $\mm(\supp(\psi))<\infty$  and thus $\psi\in L^2(X,\mm)$. Hence $\psi\in W^{1,2}(X,\sfd,\mm)$ and from Corollary \ref{cor:lipdense} we know there exists a sequence $(\psi_n)\subset W^{1,2}(X,\sfd,\mm)$ of Lipschitz functions converging to $\psi$ in $W^{1,2}(X,\sfd,\mm)$. Define $\psi^{t,+}, \psi^{t,-}:X\to\R$ by
\[
\psi^{t,+}(x):=\inf_y \psi(y)+\frac{\sfd^2(x,y)}{2t},\qquad\qquad \psi^{t,-}(x):=\sup_y \psi(y)-\frac{\sfd^2(x,y)}{2t}.
\]
It is not hard to show  that since $\psi\in C_b(X)$, the functions $\psi^{t,+},\psi^{t,-}$ are Lipschitz, equibounded and it holds   $\psi^{t,+}(x)\uparrow \psi(x) $, $\psi^{t,-}(x)\downarrow \psi(x) $ as $t\downarrow0$ for any $x\in X$ (see for instance Chapter 3 of \cite{Ambrosio-Gigli-Savare11}). 

Put $\psi_{n,t}:=\min\{\max\{\psi_n,\psi^{t,+}\},\psi^{t,-}\}$ and observe that $\psi_{n,t}$ is Lipschitz for any $n\in\N$, $t>0$. Let $\nchi\in\test\Omega$ be identically 1 on $\supp(\psi)$, and consider the functions $\nchi\psi_{n,t}\in\test\Omega$.

Since $\psi^{t,+},\psi^{t,-},\nchi$ are Lipschitz, an application of the dominate convergence theorem ensures that for any $t>0$ the sequence $n\mapsto \psi_{n,t}$ converges to $\psi$ in energy 
$W^{1,2}(X,\sfd,\mm)$ as $n\to\infty$. Thus this convergence is also w.r.t. the $W^{1,2}$ norm and from $\supp(\nchi\psi_{n,t})\subset\supp(\nchi)$, $\mm(\supp(\nchi))<\infty$ we get 
\begin{equation}
\label{eq:euno}
\lim_{n\to\infty}\int_\Omega\nabla(\nchi\psi_{n,t})\cdot\nabla g\,\d\mm=\int_\Omega\nabla\psi\cdot\nabla g\,\d\mm,\qquad\forall t>0.
\end{equation}
By construction $\psi_{n,t}$ is uniformly bounded in $n,t$ and pointwise converges to $\psi$ as $t\to 0$ uniformly on $n$. Taking into account that $\supp(\nchi\psi_{n,t})\subset\supp(\nchi)$, $|\bd g\restr\Omega|(\supp(\nchi))<\infty$ and the dominate convergence theorem we get
\begin{equation}
\label{eq:eddue}
\lim_{t\downarrow0}\int_\Omega\nchi\psi_{n,t}\,\d\bd g\restr\Omega=\int\psi\,\d\bd g\restr\Omega,\qquad\textrm{ uniformly on }n.
\end{equation}
Since $\nchi\psi_{n,t}\in\test\Omega$, \eqref{eq:euno} and \eqref{eq:eddue} together with a diagonalization argument give the thesis.
\end{proof}

\begin{lemma}\label{le:l2}
Let $(X,\sfd,\mm)$ be a proper infinitesimally Hilbertian space, $\Omega\subset X$ an open set and $g\in D(\bd,\Omega)\cap \s^2_{\rm int}(\Omega)$ with $\bd g\restr\Omega\ll\mm$, $\bd g\restr\Omega=h\mm$ and $h\in L^2_{\rm loc}(\Omega,\mm\restr{\Omega})$. Then for every $\psi\in W^{1,2}(X,\sfd,\mm)$ with bounded support contained in $\Omega$ it holds
\[
-\int_\Omega \nabla\psi\cdot\nabla g\,\d\mm=\int_\Omega \psi h\,\d\mm.
\]
\end{lemma}
\begin{proof}
By Corollary \ref{cor:lipdense} we know there exists a sequence $(\psi_n)\subset W^{1,2}(X,\sfd,\mm)$ of Lipschitz functions converging to $\psi$ in $W^{1,2}(X,\sfd,\mm)$. Let $\nchi\in\test\Omega$ be identically 1 on $\supp(\psi)$ and notice that $\nchi\psi_n\in\test\Omega\cap L^1(\Omega,|h|\mm)$ and $\nchi\psi_n\to \psi$ in $W^{1,2}(X,\sfd,\mm)$. Hence we can pass to the limit in
\[
-\int_\Omega \nabla(\nchi\psi_n)\cdot\nabla g\,\d\mm=\int_\Omega \nchi\psi_n h\,\d\mm,
\]
and get the thesis.
\end{proof}
\begin{proposition}[Chain rule]\label{prop:chainlinear}
Let $(X,\sfd,\mm)$ be an infinitesimally Hilbertian metric measure space,  $\Omega\subset X$ an open set, $g\in D(\bd,\Omega)\cap\s^2_{\rm int}(\Omega)$, $I\subset\R$ an open set such that $\mm(g^{-1}(\R\setminus I))=0$ and $\varphi\in C^{1,1}_{\rm loc}(I)$. Then the following holds.
\begin{itemize}
\item[i)] Assume that $g\restr{\Omega'}$ is Lipschitz for every $\Omega'\in\Int\Omega$. Then 
\begin{equation}
\label{eq:chainlaplh}
\varphi\circ g\in D(\bd,\Omega),\qquad\textrm{ and }\qquad \bd(\varphi\circ g)\restr\Omega= \varphi'\circ g\,\bd g\restr{\Omega}+\varphi''\circ g\,\weakgrad g^2\mm\restr\Omega.
\end{equation}
\item[ii)] Assume that $(X,\sfd)$ is proper and  $g\in C(\Omega)$. Then \eqref{eq:chainlaplh} holds.
\item[iii)] Assume that $(X,\sfd)$ is proper, $g\in L^2_{\rm loc}(X,\mm)$ and $\bd g\restr\Omega\ll\mm$  with $\frac{\d\bd g\restr\Omega}{\d\mm}\in   L^2_{\rm loc}(\Omega,\mm\restr\Omega)$. Then  \eqref{eq:chainlaplh} holds.
\end{itemize}
\end{proposition}
\begin{proof}\\*
\noindent{$\mathbf{(i)}$} This is a particular case of Proposition \ref{prop:chainlapl}.

\noindent{$\mathbf{(ii)}$} Let $\Omega'\in\Int\Omega$ and observe that since $\overline{\Omega'}$ is compact, $g(\overline{\Omega'})$ is compact as well and thus $\varphi''$ is bounded on $g(\Omega')$. It easily follows that the formula $\tilde\mu:= \varphi'\circ g\,\bd g\restr{\Omega}+\varphi''\circ g\,\weakgrad g^2\mm\restr\Omega$ defines  a locally finite measure on $\Omega$, so that the statement makes sense. 

Now notice that  $\test\Omega\subset L^1(\Omega,|\tilde\mu|)$, pick   $f\in\test\Omega$ and use \eqref{eq:leibb} to get
\begin{equation}
\label{eq:sonnolenza1}
\begin{split}
\nabla f\cdot\nabla (\varphi\circ g)= \varphi'\circ g\,\nabla f\cdot\nabla g &= \nabla (f\varphi'\circ g)\cdot\nabla g-f\,\nabla(\varphi'\circ g)\cdot\nabla g\\
&=\nabla (f\varphi'\circ g)\cdot\nabla g-f\varphi''\circ g\weakgrad g^2\qquad \mm-a.e.,
\end{split}
\end{equation}
which by integration gives
\begin{equation}
\label{eq:sonnolenza2}
-\int\nabla f\cdot\nabla (\varphi\circ g)\,\d\mm=-\int \nabla (f\varphi'\circ g)\cdot\nabla g\,\d\mm+\int f\varphi''\circ g\weakgrad g^2\,\d\mm.
\end{equation}
Hence to conclude it is sufficient to show that
\begin{equation}
\label{eq:show}
-\int \nabla (f\varphi'\circ g)\cdot\nabla g\,\d\mm=\int  f\varphi'\circ g\,\d\bd g\restr{\Omega}.
\end{equation}
This is a consequence of Lemma \ref{le:finalmente} applied to $\psi:=f\varphi'\circ g$, which by our assumptions belongs to $\s^2(X,\sfd,\mm)\cap C_c(\Omega)$.

\noindent{$\mathbf{(iii)}$} From the assumptions we know that $\varphi'\circ g\in L^2_{\rm loc}(\Omega,\mm\restr\Omega)$ and $\varphi''\circ g\in L^\infty_{\rm loc}(\Omega,\restr\Omega)$. Therefore, since  $\bd g\restr\Omega\ll\mm$ with $L^2_{\rm loc}$ density, the formula  $\tilde\mu:= \varphi'\circ g\,\bd g\restr{\Omega}+\varphi''\circ g\,\weakgrad g^2\mm\restr\Omega$ defines  a locally finite measure on $\Omega$ and  the statement makes sense. As before, we have $\test\Omega\subset L^1(\Omega,|\tilde\mu|)$. With the same computations done in \eqref{eq:sonnolenza1}, \eqref{eq:sonnolenza2} we reduce to show that \eqref{eq:show} holds also in this case. This is a consequence of Lemma \ref{le:l2} applied to $\psi:=f\varphi'\circ g$, which belongs to $W^{1,2}(X,\sfd,\mm)$ and has bounded support contained in $\Omega$.
\end{proof}
The next theorem shows that in this linear case, the Leibniz rule holds also for the Laplacian, a property which, being consequence of the second identity in \eqref{eq:leibb}, is not available in the general non-linear setting. As for the chain rule, the theorem is stated in 3 different ways, depending on the assumptions on $g_i$, $\bd g_i\restr\Omega$ and the space.
\begin{theorem}[Leibniz rule for the Laplacian]\label{thm:leiblap} Let $(X,\sfd,\mm)$ be an infinitesimally Hilbertian space,  $\Omega\subset X$ an open set and  $g_1,g_2\in D(\bd,\Omega)\cap \s^2_{\rm int}(\Omega)$. Then the following is true.
\begin{itemize}
\item[i)] If $g_1,g_2$ are Lipschitz on  $\Omega'$ for every $\Omega'\in\Int\Omega$, and $g_1,g_2\in D_{\rm ifm}(\bd,\Omega)$, then $g_1g_2\in D(\bd,\Omega)$ and 
\begin{equation}
\label{eq:leinlapl}
\bd(g_1g_2)\restr\Omega=g_1\bd  g_2\restr\Omega+g_2\bd  g_1\restr\Omega+2\nabla g_1\cdot\nabla g_2\,\mm.
\end{equation}
\item[ii)] If $(X,\sfd)$ is proper and $g_1,g_2\in C(\Omega)$, then  $g_1g_2\in D(\bd,\Omega)$ and \eqref{eq:leinlapl} holds.
\item[iii)] If $(X,\sfd)$ is proper, $g_1,g_2\in L^2_{\rm loc}(\Omega,\mm\restr\Omega)\cap L^\infty_{\rm loc}(\Omega,\mm\restr\Omega)$ and $\bd g_i\restr\Omega\ll\mm$ with $L^2_{\rm loc}(\Omega,\mm\restr\Omega)$ density, $i=1,2$, then $g_1g_2\in D(\bd,\Omega)$ and \eqref{eq:leinlapl} holds.
\end{itemize} 
\end{theorem}
\begin{proof}

\noindent{$\mathbf{(i)}$} Being Lipschitz, $g_1,g_2$ are bounded on $\Omega'$ for any $\Omega'\in\Int\Omega$, hence  $g_1g_2\in\s^2_{\rm int}(\Omega)$. It is also clear  that the right hand side of \eqref{eq:leinlapl} defines a locally finite measure $\mu$ on $\Omega$, so that the statement makes sense. The fact that $|\bd g_i\restr\Omega|(\Omega')$ is finite for every $\Omega'\in\Int\Omega $, $i=1,2$, grants that  $\test\Omega\subset  L^1(\Omega,|\bd g_i\restr\Omega|)$, $i=1,2$, and $\test\Omega\subset L^1(\Omega,|\mu|)$. To conclude, pick $f\in\test{\Omega}$, notice that $fg_1,fg_2\in\test\Omega$ and take the Leibniz rule \eqref{eq:leibb} into account to get
\begin{equation}
\label{eq:contoleib}
\begin{split}
\nabla f\cdot\nabla(g_1g_2)&=g_1\nabla f\cdot \nabla g_2 + g_2\nabla f\cdot \nabla g_1 =\nabla(fg_1)\cdot\nabla g_2+\nabla(fg_2)\cdot\nabla g_1-2f\,\nabla g_1\cdot\nabla g_2,
\end{split}
\end{equation}
which integrated gives the thesis.

\noindent{$\mathbf{(ii)}$} As before, the right hand side of \eqref{eq:leinlapl} defines a locally finite measure $\mu$, and as before  $g_1g_2\in\s^2_{\rm int}(\Omega)$, $\test\Omega\subset  L^1(\Omega,|\bd g_i\restr\Omega|)$, $i=1,2$, and $\test\Omega\subset L^1(\Omega,|\mu|)$. Pick $f\in\test\Omega$ and notice that with the same computations done in \eqref{eq:contoleib} the thesis follows if we show that
\begin{equation}
\label{eq:fineleib}
\begin{split}
\int_\Omega \nabla(fg_1)\cdot\nabla g_2\,\d\mm&=-\int_\Omega fg_1\,\d\bd g_2\restr\Omega,\\
\int_\Omega \nabla(fg_2)\cdot\nabla g_1\,\d\mm&=-\int_\Omega fg_2\,\d\bd g_1\restr\Omega.
\end{split}
\end{equation}
These are a consequence of Lemma \ref{le:finalmente} applied to $\psi:=fg_1$, $g:=g_2$ and $\psi:=fg_2$, $g:=g_1$ respectively.

\noindent{$\mathbf{(iii)}$} Same as $(ii)$ with Lemma \ref{le:l2} in place of Lemma \ref{le:finalmente} to justify \eqref{eq:fineleib}
\end{proof}

\bigskip

In the discussion we made up to now, the Laplacian has always been understood by means of the integration by parts formula in Definition \ref{def:lap}. On a Euclidean setting there is at least another completely different approach: to look at the behavior of the heat flow for small times. It is therefore natural to try to understand whether a similar approach is possible also in the non-smooth context. Although it is always possible to define the heat flow as the gradient flow of $\c_2$ in $L^2$, in practice, in order for it to have at least some of the regularization properties available in a Euclidean context, it seems better to restrict the attention to spaces with \emph{Riemannian Ricci curvature bounded from below}.  This class of spaces was introduced in \cite{Ambrosio-Gigli-Savare11bis} as a strengthening of the standard $\CD(K,\infty)$ class which rules out Finsler geometries. On such spaces the heat flow can be equivalently regarded as gradient flow of $\c_2$ in $L^2$ or as gradient flow of the relative entropy in $(\probt X,W_2)$, it is linear, self adjoint and has some very general regularization properties, see Theorem \ref{thm:rcd} below.

Recall that for $\mm\in\prob X$ the relative entropy functional ${\rm Ent}_\smm:\prob X\to[0,+\infty]$ is defined as
\[
{\rm Ent}_\smm(\mu):=\left\{
\begin{array}{ll}
\displaystyle{\int\rho\log\rho\,\d\mm},&\qquad\textrm{ if }\mu=\rho\mm,\\
+\infty,&\qquad\textrm{ otherwise}.
\end{array}
\right.
\]
By $D({\rm Ent}_\smm)$ we denote the set of $\mu$'s such that ${\rm Ent}_\smm(\mu)<\infty$.
\begin{definition}[Riemannian Ricci curvature bounds]\label{def:rcd}
Let $(X,\sfd,\mm)$ be as in \eqref{eq:mms} with $\mm\in\prob X$ and $K\in\R$. We say that it is a $\RCD(K,\infty)$ space provided it is infinitesimally Hilbertian and for any $\mu,\nu\in D({\rm Ent}_\smm)\cap \probt X$  there exists $\ppi\in\gopt(\mu,\nu)$ such that the inequality
\begin{equation}
\label{eq:weight}
{\rm Ent}_\smm((\e_t)_\sharp\ppi_F)\leq (1-t){\rm Ent}_\smm((\e_0)_\sharp\ppi_F)+t{\rm Ent}_\smm((\e_1)_\sharp\ppi_F)-\frac K2 t(1-t)W_2^2\big((\e_0)_\sharp\ppi_F,(\e_1)_\sharp\ppi_F\big),
\end{equation}
holds for any $t\in[0,1]$, where $F:\geo(X)\to\R$ is any Borel non-negative bounded function such that $\int F\,\d\ppi=1$ and $\ppi_F:= F\ppi$.
\end{definition}
Notice that the usual $K$-geodesic convexity of the entropy is enforced by asking it along all weighted plans $\ppi_F$ in \eqref{eq:weight}\footnote{in the recent paper \cite{AmbrosioGigliMondinoRajala12} it is shown that this assumption is in fact redundant, as it follows by the standard $\CD(K,\infty)$ condition plus infinitesimal Hilbertianity}.

Let $(X,\sfd,\mm)$ be a $\RCD(K,\infty)$ space. Denote by $\heatl_t:L^2(X,\mm)\to L^2(X,\mm)$ the gradient flow of $\c_2$ (which clearly exists and is unique), i.e. for any $f\in L^2(X,\mm)$ the curve $t\mapsto\heatl_t(f)$ is the gradient flow of $\c_2$ in $L^2(X,\mm)$ starting from $f$. Similarly, denote by $\heatw_t:D({\rm Ent}_\smm)\cap\probt X \to D({\rm Ent}_\smm)\cap\probt X$ the gradient flow of ${\rm Ent}_\smm$ in $(\probt X,W_2)$ (uniqueness has been established in \cite{Gigli10}, existence in \cite{Ambrosio-Gigli-Savare05}, \cite{Gigli10}, \cite{Ambrosio-Gigli-Savare11bis} for, respectively, compact / locally compact / extended Polish spaces).

In the next theorem we collect those results proved in \cite{Ambrosio-Gigli-Savare11},  \cite{Ambrosio-Gigli-Savare11bis} which we will use in the rest of the section.
\begin{theorem}[Properties of the heat flow]\label{thm:rcd}
Let $K\in \R$ and $(X,\sfd,\mm)$ a $\RCD(K,\infty)$ space. Then the following hold.
\begin{itemize}
\item[i)] \underline{Coincidence of the two flows} Let $f\in L^2(X,\mm)$ be such that $f\mm\in\probt X$. Then for any $t\geq 0$ it holds $(\heatl_tf)\mm=\heatw_t(f\mm)$.
\item[ii)] \underline{Linearity and extendibility} $\heatl_t$ is a linear semigroup which can be uniquely extended to a continuous semigroup on $L^1(X,\mm)$, still denoted by $\heatl_t$. Similarly, $\heatw_t$ is a linear semigroup which can be uniquely extended to a continuous linear semigroup, still denoted by $\heatw_t$, on the space $\mathcal M(\supp(\mm))$ of Radon measures with finite mass on $\supp(\mm)$, where continuity is meant in duality with $C_b(\supp(\mm))$.
\item[iii)]\underline{Regularization properties} For $t>0$, $\heatw_t$ maps $\mathcal M(\supp(\mm))$ on $D({\rm Ent}_\smm)$ and  $\heatl_t$ maps contractively $L^\infty(X,\mm)$ in $C_b(\supp(\mm))$, with $L^\infty\to \Lip$ regularization:
\begin{equation}
\label{eq:inflip}
\Lip(\heatl_tf)\leq \frac{\|f\|_{L^\infty}}{\sqrt{2\int_0^te^{2Ks}\,\d s}}.
\end{equation}
\item[iv)]\underline{Self-adjointness} For any $f\in L^\infty(X,\mm)$, $\mu\in \mathcal M(\supp(\mm))$ and $t> 0$ it holds
\begin{equation}
\label{eq:self}
\int f\,\d\heatw_t\mu=\int\heatl_tf\,\d\mu.
\end{equation}
\item[v)]\underline{Compatibility with the Dirichlet form theory} The intrinsic distance $\sfd_\mathcal E$ associated to $\mathcal E$ coincides with $\sfd$.
\item[vi)] \underline{Validity of the Bakry-Emery curvature condition} For any $f\in W^{1,2}(X,\sfd,\mm)$  it holds
\begin{equation}
\label{eq:BE}
\weakgrad{(\heatl_tf)}^2\leq e^{-2Kt}\heatl_t(\weakgrad f^2),\qquad \mm-a.e.\ \forall t\geq 0.
\end{equation}
\end{itemize}
\end{theorem}
As expected,  the Laplacian commutes with the heat flow:
\begin{proposition}[$\bd\heatl_tg=\heatw_t\bd g$]
Let $K\in\R$, $(X,\sfd,\mm)$ a $\RCD(K,\infty)$ space and $g\in W^{1,2}(X,\sfd,\mm)\cap D(\bd)$ with $\|\bd g\|_{{\rm TV}}<\infty$. Then for every $t\geq 0$ it holds
\begin{equation}
\label{eq:commuta}
\bd\big(\heatl_tg\big)=\heatw_t\big(\bd g\big).
\end{equation}
\end{proposition}
\begin{proof}
Taking into account \eqref{eq:self} and the definition of $\bd$ we need to prove that for any $f\in\test X$ it holds
\[
\int \nabla f\cdot \nabla\heatl_t g\,\d\mm=\int\nabla \heatl_tf\cdot\nabla g\,\d\mm.
\]
Given that $f,g\in D(\mathcal E)$, such equality is a standard consequence of the self-adjointness  of the evolution semigroup associated to a Dirichlet form, we omit the details.
\end{proof}
Having a well defined heat flow allows, as we discussed before, to give a definition of Laplacian based on its infinitesimal behavior.
\begin{definition}[A variant of the definition of Laplacian]
Let $K\in\R$, $(X,\sfd,\mm)$ be an $\RCD(K,\infty)$ space and $g\in L^1(X,\mm)$. We say that $g\in D(\tilde\bd)$  provided there exists a locally finite measure $\mu$ such that 
\[
\lim_{t\downarrow 0}\int f\frac{\heatl_tg-g}{t}\,\d\mm=\int f\,\d\mu,
\] 
for any $f\in C_b(X)\cap L^1(X,|\mu|)$ with bounded support. In this case we write $\tilde\bd g=\mu$ (notice that $\mu$ is certainly unique).
\end{definition}
Given that the heat semigroup is linear, it is trivial that $D(\tilde \bd)\subset L^1$ is a vector space and that $\tilde\bd$  is linear. Our first task is  to compare this version of the Laplacian with the one previously introduced in Definition \ref{def:lap}. Notice that Proposition \ref{prop:complin} grants that 
\[
\begin{split}
g\in D(\Delta),\quad h=\Delta g&\qquad\Leftrightarrow \qquad g\in W^{1,2}(X,\sfd,\mm)\cap D(\bd),\quad \bd g=h\mm,\\
&\qquad\Leftrightarrow \qquad g\in W^{1,2}(X,\sfd,\mm)\cap D(\tilde \bd),\quad \tilde\bd g=h\mm.
\end{split}
\]
The following statement generalizes such result to situations  where  the Laplacian is a measure.
\begin{proposition}\label{prop:stessolap}
Let $K\in\R$, $(X,\sfd,\mm)$ a $\RCD(K,\infty)$ space and $g\in W^{1,2}(X,\sfd,\mm)$. Assume that $g\in D(\tilde\bd )$. Then $g\in D(\bd)$ and $\bd g=\tilde\bd g$. Viceversa, if $g\in D(\bd)$ and $\|\bd g\|_{\rm TV }<\infty$, then $g\in D(\tilde\bd)$ and $\tilde \bd g=\bd g$.
\end{proposition}
\begin{proof}
We know that $\heatl_sg\to g$ in $L^2(X,\mm)$ as $s\downarrow 0$ and, from the lower semicontinuity of $\c_2$ and the fact that $\c_2$ decreases along the heat flow, $\c_2(\heatl_sg)\to\c_2(g)$ as $s\to 0$. Hence, since $W^{1,2}(X,\sfd,\mm)$ is Hilbert, we deduce $\heatl_sg\to g$ in $W^{1,2}(X,\sfd,\mm)$ as $s\to 0$.

Let $f\in\test X\cap L^1(X,|\tilde\bd g|)\subset W^{1,2}(X,\sfd,\mm)$ and notice that
\[
\begin{split}
\int f\frac{\heatl_tg-g}{t}\,\d\mm&=\frac1t\iint_0^tf\Delta\heatl_sg\,\d s\,\d\mm=-\frac1t\int_0^t\int \nabla f\cdot \nabla \heatl_sg\,\d\mm\,\d s\\
&=-\int \nabla f\cdot\nabla g\,\d\mm+{\rm Rem}_t,
\end{split}
\] 
where the reminder term  ${\rm Rem}_t$ is bounded by
\[
\big|{\rm Rem}_t\big|\leq\|\weakgrad f\|_{L^2(X,\smm)} \frac1t\int_0^t\|\weakgrad{(\heatl_sg-g)}\|_{L^2(X,\smm)}\,\d s.
\]
The first part of the statement follows.

For the second, fix $f\in C_b(X)$  and use \eqref{eq:commuta} and \eqref{eq:self}  to get
\[
\int f\frac{\heatl_tg-g}{t}\,\d\mm=\frac1t\int_0^t\int f\d\bd\heatl_sg\,\d s=\int\left(\frac1t\int_0^t  \heatl_s f\,\d s\right)\d\bd g.
\]
By $(iii)$ of Theorem \ref{thm:rcd} we have $\|\heatl_sf\|_{L^\infty}\leq \|f\|_{L^\infty}$, thus by the dominate convergence theorem, to conclude it is sufficient to show that $\heatl_sf(x)\to f(x)$ for any $x\in \supp(\mm)$ as $s\downarrow 0$ (recall that $|\bd g|(X\setminus\supp(\mm))=0$). Using again \eqref{eq:self}  we have
\[
\heatl_sf(x)=\int \heatl_sf(y)\,\d\delta_x(y)=\int f(y)\,\d\heatw_s(\delta_x)(y),
\]
and the conclusion follows from the weak continuity  statement in $(ii)$ of Theorem \ref{thm:rcd}.
\end{proof}
The following is a natural variant of Proposition \ref{prop:comparison}. Notice that the use of the theory of semigroups allows us to conclude that the $\limi$ at the left hand side of \eqref{eq:Lm} is always a limit. The compactness assumption can be dropped at least if the measure is doubling (in this case, also a weak local Poincar\'e inequality holds, see \cite{Rajala11}).
\begin{proposition}\label{prop:boundext}
Let $K\in\R$, $(X,\sfd,\mm)$ a compact $\RCD(K,\infty)$ space, $g\in L^1(X,\mm)$ and $\mu$  a Radon  measure on $X$. Assume that 
\begin{equation}
\label{eq:Lm}
\limi_{t\downarrow0}\int f\frac{\heatl_tg-g}t\,\d\mm\geq \int f\,\d\mu,\qquad\forall f\in C(X),\ f\geq 0.
\end{equation}
Then $g\in D(\tilde \bd)$ and $\tilde \bd g\geq \mu$.
\end{proposition}
\begin{proof} It is obvious that \eqref{eq:Lm} holds also with $\mu\restr{\supp(\mm)}$ in place of $\mu$ and that it implies $\mu\restr{X\setminus\supp(\smm)}\leq 0$, thus we can assume $\supp(\mm)=X$.

Define the functionals $L^\pm:C(X)\to\R\cup\{\pm\infty\}$ by
\begin{equation}
\label{eq:lpm}
L^-(f):=\limi_{t\downarrow 0}\int f\frac{\heatl_tg-g}t\,\d\mm,\qquad L^+(f):=\lims_{t\downarrow 0}\int f\frac{\heatl_tg-g}t\,\d\mm.
\end{equation}
It is clear that both are positively 1-homogeneous, that $L^-$ is concave and $L^+$ convex. Moreover,  since $\int \heatl_tg\,\d\mm=\int g\,\d\mm$, it holds $L^\pm(f+c)=L^\pm(f)$ for any $f\in C(X)$, $c\in\R$.

Pick a generic $f\in C(X)$, notice that $f+\|f\|_{L^\infty}\geq 0$ so that from \eqref{eq:Lm} we derive
\begin{equation}
\label{eq:Lp}
L^-(f)=L^-(f+\|f\|_{L^\infty})\geq  \int f+\|f\|_{L^\infty}\,\d\mu\geq -2\|f\|_{L^\infty}\|\mu\|_{{\rm TV}}.
\end{equation}
Then using the identity $L^-(-f)=-L^+(f)$ we get
\[
-2\|f\|_{L^\infty}\|\mu\|_{{\rm TV}}\leq L^-(f)\leq L^+(f)\leq 2\|f\|_{L^\infty}\|\mu\|_{{\rm TV}},\qquad\forall f\in C(X).
\]
Hence the inequality $L^-(f)-L^-(g)\geq L^-(f-g)\geq -2\|f-g\|_{L^\infty}\|\mu\|_{{\rm TV}}$, and the similar one written with reversed roles, gives that $L^-:C(X)\to \R$ is Lipschitz. Similarly for $L^+$.

Now fix  a non-negative function $\varphi\in C^\infty_c(0,1)$ such that $\int_0^\infty\varphi(t)\,\d t=1$. For any $f\in {\rm LIP}(X)$,  $r>0$ define $\tau_rf:X\to\R$ by
\[
\tau_rf:=\frac1r\int_0^\infty \heatl_sf\,\varphi\big(\frac sr\big)\,\d s.
\]
From \eqref{eq:BE} it follows that $\Lip(\tau_rf)\leq \Lip(f)\int_0^\infty e^{-Ksr}\varphi(s)\,\d s$ for any $r>0$ (see Section 6.1 of \cite{Ambrosio-Gigli-Savare11bis} for details), and therefore the $L^2(X,\mm)$ convergence of $\tau_rf$ to $f$ implies uniform convergence.

By standard semigroups computations we know that $\tau_rf\in D(\Delta)$ and that $\Delta\tau_rf=\frac1r\int_0^\infty \Delta \heatl_sf\varphi(\frac sr)\,\d s$. Hence for any $r>0$ we have
\[
L^-(\tau_r f)=\int \Delta \tau_rf g\,\d\mm=L^+(\tau_rf).
\]
Since $L^-,L^+$ are both continuous on $C(X)$, by letting $r\downarrow 0$ we deduce that $L^+(f)=L^-(f)$ for any $f\in{\rm LIP}(X)$ and hence they coincide also on $C(X)$. The conclusion follows from the Riesz theorem.
\end{proof}
The notion of Laplacian $\tilde\bd$ allows to prove the Bochner inequality with $N=\infty$ in compact $\RCD(K,\infty)$ spaces\footnote{in the recent paper \cite{Savare13}, Savar\'e proved that for `many' functions $g$ as in the statement of Theorem \ref{thm:boch} one has $\weakgrad g\in W^{1,2}(X,\sfd,\mm)$, so that thanks to Proposition \ref{prop:stessolap} in the left hand side we can write  the Laplacian $\bd$ instead of $\tilde\bd$. The regularity result  $\weakgrad g\in W^{1,2}(X,\sfd,\mm)$ is non-trivial and follows from a generalization of works of Barky in the context of Dirichlet forms \cite{Bakry06}}: shortly said, this will follow by the standard argument consisting in derivating at time $t=0$ the Bakry-Emery condition \eqref{eq:BE}. Recall that by $\Delta$ we intend the generator of the Dirichlet form $\mathcal E$, see \eqref{eq:defdelta}.
\begin{theorem}[Bochner inequality with $N=\infty$]\label{thm:boch}
Let $K\in\R$, $(X,\sfd,\mm)$ a compact $\RCD(K,\infty)$ space and $g\in D(\Delta)$ with  $\Delta g\in W^{1,2}(X,\sfd,\mm)$. Then $\weakgrad g^2\in D(\tilde\bd)$ and it holds
\begin{equation}
\label{eq:be}
\tilde\bd\frac{\weakgrad g^2}2\geq \big(\nabla\Delta g\cdot\nabla g+K\weakgrad g^2\big)\mm.
\end{equation}
\end{theorem}
Note: for any $f\in L^2(X,\mm)$, $t>0$, the function $g:=\heatl_tf$ fulfills the hypothesis.
\begin{proof}
By standard semigroup computation we get that  for $g$ as in the hypothesis it holds
\begin{equation}
\label{eq:claimboch}
\lim_{t\downarrow 0}\frac{\weakgrad{(\heatl_tg)}^2-\weakgrad g^2}t=2\nabla\Delta g\cdot\nabla g,\qquad\textrm{ in }L^2(X,\mm),
\end{equation}
we omit the details. Now recall that from \eqref{eq:BE} we know that  it holds
\[
\weakgrad{(\heatl_tg)}^2\leq e^{-2Kt}\heatl_t(\weakgrad g^2),\qquad\forall t\geq 0.
\]
Notice that for $t=0$ this is an equality, multiply both sides by  a continuous and non-negative function $f:X\to\R$ to get
\[
\int f\frac{\weakgrad{(\heatl_tg)}^2-\weakgrad g^2}t\,\d\mm\leq\int f\frac{e^{-2Kt}\heatl_t(\weakgrad g^2)-\weakgrad g^2}t\,\d\mm.
\]
Using \eqref{eq:claimboch} we get
\[
2\int f\nabla\Delta g\cdot\nabla g\,\d\mm\leq\limi_{t\downarrow0}\int f\frac{e^{-2Kt}\heatl_t(\weakgrad g^2)-\weakgrad g^2}t\,\d\mm.
\]
since $\heatl_t(\weakgrad g^2)\mm$ weakly converges to $\weakgrad g^2\mm$, it is immediate to check that for the right hand side it holds
\[
\limi_{t\downarrow0}\int f\frac{e^{-2Kt}\heatl_t(\weakgrad g^2)-\weakgrad g^2}t\,\d\mm=-2K\int f\weakgrad g^2\,\d\mm+ \limi_{t\downarrow0}\int f\frac{\heatl_t(\weakgrad g^2)-\weakgrad g^2}t\,\d\mm
\]
therefore the conclusion follows from Proposition \ref{prop:boundext} with $\weakgrad g^2$ in place of $g$ and $\mu:=2\big(\nabla\Delta g\cdot\nabla g+K\weakgrad g^2\big)\mm$.
\end{proof}

\begin{remark}{\rm The derivation of the Bochner inequality with $N=\infty$ from the contraction estimate \eqref{eq:BE} is classical. What is interesting to observe is the link between \eqref{eq:BE} and the $K$-geodesic convexity of the entropy. Indeed, even on a smooth Riemannian setting, in deriving the former from the latter, the fact that the tangent space is Hilbert is used twice:
\begin{itemize}
\item[i)] A first time to deduce the $K$-contractivity of $W_2$ along two heat flows.
\item[ii)] A second time to pass, by a duality argument, from the $K$-contractivity of $W_2$ to \eqref{eq:BE} (see e.g. the general argument in \cite{Kuwada10}).
\end{itemize}
It is known that on Finsler manifolds the $K$-convexity of the entropy does \emph{not} imply the $K$-contractivity of $W_2$ along two heat flows (as proved by Ohta-Sturm in \cite{Sturm-Ohta10}), and the duality argument in $(ii)$ works only for linear flows, so that when applied to the heat flow, it works only on Riemannian manifolds. 

The fact that a duality argument is used twice, suggests that there should be a way to state and prove the Bochner inequality also on a Finsler setting satisfying the $\CD(K,\infty)$ condition. This is actually the case: more generally,  the Bochner inequality, also with finite $N$, holds also on a Finsler structure, provided an appropriate reformulation of the term $\Delta\frac{|\nabla\varphi|^2}2$ is considered (see \cite{Ohta-Sturm-Bochner} and \cite{Bochner-CD} for two different approaches).
}\fr\end{remark}

\section{Comparison estimates}\label{se:comp}
In this chapter we recall the definition of metric measure spaces with Ricci curvature $\geq K$ and dimension $\leq N$, and prove that on these spaces the sharp Laplacian comparison estimates for the distance function holds, at least on the infinitesimally strictly convex case.
\subsection{Weak Ricci curvature bounds}
Here we recall the definition of $\CD(K,N)$, $N<\infty$, spaces given by Sturm in \cite{Sturm06II} and Lott-Villani in \cite{Lott-Villani09} (the latter reference deals with the case $K=0$ only). 

Let $u:[0,\infty)\to\R$ be a convex continuous and sublinear (i.e. $\lim_{z\to+\infty}\frac{u(z)}{z}=0$) function satisfying $u(0)=0$. Let $\mathcal M^+(X)$ the space of finite non-negative Radon measures on $X$. The \emph{internal energy} functional  $\u:\mathcal M^+(X)\to\R\cup\{+\infty\}$ associated to $u$ is well defined by the formula
\[
\u (\mu):=\int u(\rho)\,\d\mm,\qquad \mu=\rho\mm+\mu^s,\ \mu^s\perp\mm.
\]
Jensen's inequality ensures that if $\mm(\supp(\mu))<\infty$, then $\u(\mu)>-\infty$. More generally, the functional $\u$ is lower semicontinuous in $\prob{B}\subset\mathcal M^+ (X)$ w.r.t. convergence in duality with $C_b(B)$, for any closed set $B$ such that $\mm(B)<\infty$.

Functions $u$ of interest for us are 
\[
\begin{split}
u_N(z)&:=-z^{1-\frac1N},\qquad N\in(1,\infty),\\
\end{split}
\]
and we will denote the associated internal energies by $\u_N$ respectively.

For $N\in(1,\infty)$, $K\in\R$ the \emph{distortion coefficients} $\tau_{K,N}^{(t)}(\theta)$ are functions $[0,1]\times[0,\infty)\ni (t,\theta)\mapsto \tau_{K,N}^{(t)}(\theta)\in[0, +\infty]$ defined by
\[
\tau^{(t)}_{K,N}(\theta):=\left\{
\begin{array}{ll}
+\infty,&\qquad\textrm{if }K\theta^2\geq (N-1)\pi^2,\\
\displaystyle{t^{\frac1N}\left(\frac{\sin(t\theta\sqrt{K/(N-1)})}{\sin(\theta\sqrt{K/(N-1)})}\right)^{1-\frac1N}},&\qquad\textrm{if }0<K\theta^2<(N-1)\pi^2,\\
t,&\qquad\textrm{if }K\theta^2=0,\\
\displaystyle {t^{\frac1N}\left(\frac{\sinh(t\theta\sqrt{-K/(N-1)})}{\sinh(\theta\sqrt{-K/(N-1)})}\right)^{1-\frac1N}},&\qquad\textrm{if }K\theta^2<0.
\end{array}
\right.
\]

\begin{definition}[Weak Ricci curvature bound] Let $(X,\sfd,\mm)$ be a boundedly finite metric measure space. 
We say that $(X,\sfd,\mm)$ is a $\CD(K,N)$ space, $K\in\R$, $N\in(1,\infty)$ provided for any $\mu,\nu\in\prob{\supp(\mm)}$ with bounded support there exists $\ppi\in\gopt(\mu,\nu)$ such that
\begin{equation}
\label{eq:cd}
\u_{N'}((\e_t)_\sharp\ppi)\leq-\int\tau^{(1-t)}_{K,N'}\big(\sfd(\gamma_0,\gamma_1)\big)\rho^{-\frac1{N'}}(\gamma_0)+\tau^{(t)}_{K,N'}\big(\sfd(\gamma_0,\gamma_1)\big)\eta^{-\frac1{N'}}(\gamma_1)\,\d\ppi(\gamma),\quad\forall t\in[0,1],
\end{equation}
for any $N'\geq N$, where $\mu=\rho\mm+\mu^s$ and $\nu=\eta\mm+\nu^s$, with $\mu^s,\nu^s\perp\mm$.
\end{definition}
When dealing with a $\CD(K,N)$ space, we will always assume that $\supp(\mm)=X$ (this is dangerous only when discussing stability issues), and to avoid trivialities we also assume that $X$ contains more than one point. In the following proposition we collect those basic properties of $\CD(K,N)$ spaces we will use later on.
\begin{proposition}[Basic properties of $\CD(K,N)$ spaces]\label{prop:basecd}
Let $(X,\sfd,\mm)$ be a $\CD(K,N)$ space, $K\in\R$, $N<\infty$. Then $(X,\sfd)$ is proper and geodesic, $\mm$ is doubling and $(X,\sfd,\mm)$ supports a weak local 1-1 Poincar\'e inequality, i.e. for any bounded Borel function $f:X\to\R$ and any upper gradient $G$ of $f$ and $R>0$ it holds
\[
\frac{1}{\mm(B_r(x))}\int_{B_r(x)}\Big|f-\langle f\rangle_{B_r(x)} \Big|\,\d\mm\leq C(R)r\frac{1}{\mm(B_{2r}(x))}\int_{B_{2r}(x)}G\,\d\mm,\qquad\forall 0<r<R,
\]
where $\langle f\rangle_{B_r(x)}:=\frac{1}{\mm(B_r(x))}\int_{B_r(x)} f\,\d\mm$, for some constant $C(R)$ depending only on $K,N$ and the upper bound on the radius $R>0$. Also, the Bishop-Gromov comparison estimates holds, i.e. for any $x\in \supp(\mm)$ it holds
\begin{equation}
\label{eq:BGv}
\begin{split}
\frac{\mm(B_r(x))}{\mm(B_R(x))}&\geq\left\{
\begin{array}{ll}
\displaystyle{\frac{\int_0^r\sin(t\sqrt{K/(N-1)})^{N-1}\,\d t}{\int_0^R\sin(t\sqrt{K/(N-1)})^{N-1}\,\d t}}&\qquad\textrm{ if }K>0\\
&\\
\displaystyle{\frac {r^N}{R^N}}&\qquad\textrm{ if }K=0\\
&\\
\displaystyle{\frac{\int_0^r\sinh(t\sqrt{K/(N-1)})^{N-1}\,\d t}{\int_0^R\sinh(t\sqrt{K/(N-1)})^{N-1}\,\d t}}&\qquad\textrm{ if }K<0\\
\end{array}
\right.
\end{split}
\end{equation}
for any $0<r\leq R\leq \pi\sqrt{(N-1)/(\max\{K,0\})}$.

Finally, if $K>0$ then $(\supp(\mm),\sfd)$ is compact and with diameter at most $\pi\sqrt{\frac{N-1}K}$.
\end{proposition}
\begin{proof}
For the Poincar\'e inequality see \cite{Lott-Villani-Poincare} for the original argument requiring the non-branching condition and the more recent paper \cite{Rajala11} for the same result without such assumption. For the other properties, see \cite{Sturm06II} or the final Chapter of \cite{Villani09}.
\end{proof}
We conclude this introduction recalling that on general metric measure spaces $(X,\sfd,\mm)$, given a Lipschitz function $f:X\to\R$, typically the objects $\lip(f),\lip^+(f),\lip^-(f),\weakgrad f$ are all different one from each other, the only information available being $\weakgrad f\leq \lip^\pm(f)$ $\mm$-a.e. and $\lip^\pm(f)\leq\lip(f)$ everywhere. On doubling spaces supporting a weak local 1-1 Poincar\'e inequality, the fine analysis done by Cheeger in \cite{Cheeger00} ensures that these objects can be identified and, as already recalled in Remark \ref{re:notazione}, that the $p$-minimal weak upper gradient does not depend on $p$.

In particular,  we have the  following:
\begin{theorem}\label{thm:cheeger}
Let $(X,\sfd,\mm)$ be a $\CD(K,N)$ space, $K\in \R$, $N\in(1,\infty)$. Then we have
\begin{equation}
\label{eq:cheeger2}
\weakgrad f==\lip^+(f)=\lip^-(f)\lip(f),\qquad \mm-a.e.\qquad \forall f:X\to\R\textrm{ locally Lipschitz}.
\end{equation}
Furthermore, the $p$-minimal weak upper gradient of s Sobolev function does not depend on $p>1$ and in particular also the objects $D^\pm f(\nabla g)$ and $\bd$ do not depend on $p$.
\end{theorem}

\subsection{Variants of calculus rules}
In the introduction we saw that the heart of the proof of Laplacian comparison estimates for the distance function on $\CD(K,N)$ space is to differentiate the internal energy functional $\u_N$ along a Wasserstein 
geodesic. To make this computation on a metric measure space is complicated by the lack of a change of variable formula, which prevents the use of the standard analytical tools. In this abstract setting, this sort of 
computation has been done for the first time in \cite{Ambrosio-Gigli-Savare11bis} via an analogous of Theorem \ref{thm:horver} relating horizontal and vertical derivatives. In order to apply such theorem, it was 
important to be sure that the Wasserstein geodesic along which the derivative was computed was made of absolutely continuous measures with uniformly bounded densities (at least for $t$ close to 0), so that the plan associated to it was of 
bounded compression. Geodesics of this kind were obtained by assuming some strong form of geodesic convexity for the relative entropy.

Our purpose here is to show that as a consequence of Theorem \ref{thm:cheeger}, in the finite dimensional case no form of strong geodesic convexity is needed for the internal energy functionals in order for the argument to be carried out. Shortly said, the assumption that $\ppi$ has bounded compression can be replaced by $\u_N((\e_t)_\sharp\ppi)\to\u_N((\e_0)_\sharp\ppi)$ as $t\downarrow0$.

The arguments of this section are based on the following simple lemma.
\begin{lemma}\label{le:scemo}
Let $(X, \sfd,  \mm)$ be as in \eqref{eq:mms},   $(\mu_n)\subset\prob{X }$ a sequence of measures and $\mu\in\prob X$. Assume that for some closed set $B$ with $\mm(B)<\infty$ it holds $\supp(\mu_n)\subset B$ for every $n\in\N$, $\supp(\mu)\subset B$ and that $(\mu_n)$ weakly converges to $\mu$ in duality with $C_b(B)$. Assume furthermore that 
\[
\mu\ll\mm,\qquad\textrm{and}\qquad \u_N(\mu_n)\to\u_N(\mu),\quad\textrm{ as }n\to\infty.
\]
Then for every Borel bounded function $f:B\to\R$ it holds
\begin{equation}
\label{eq:bb}
\lim_{n\to\infty}\int f\,\d\mu_n=\int f\,\d\mu.
\end{equation}
\end{lemma}
\begin{proof}
Write $\mu=\rho\mm$,  $\mu_n=\rho_n\mm+\mu^s_n$ with $\mu_n^s\perp\mm$. We need to prove that $\mu^s_n(B)\to 0$ and that $\rho_n\to\rho$ in duality with $L^\infty(B,\mm)$. We argue by contradiction and assume that there exists a sequence $c_n\uparrow\infty$ such that 
\begin{equation}
\label{eq:lossmass}
\int_{\{\rho_n\geq c_n\}}\rho_n\,\d\mm+\mu^s_n(B)\geq C>0,\qquad\forall n\in\N.
\end{equation}
Define the measures $\nu_n:=\rho_n\nchi_{\{\rho_n\leq c_n\}}\mm\leq \mu_n$ and notice that since $(\mu_n)$ is tight, so is $(\nu_n)$. Hence up to pass to a subsequence, not relabeled, we can assume that $(\nu_n)$ converges to some Borel measure $\nu\in\mathcal M^+( B)$ in duality with $C_b(B)$. From
\[
\begin{split}
\big|\u_N(\mu_n)-\u_N(\nu_n)\big|=\int_{\{\rho_n\geq c_n\}}|u_N(\rho_n)|\,\d\mm\leq \left(\int_{\{\rho_n\geq c_n\}}\rho_n\,\d\mm\right)^{\frac{N-1}{N}}\mm(\{\rho_n\geq c_n\})^{\frac1N},
\end{split}
\]
and Chebyshev's inequality $\mm(\{\rho_n\geq c_n\})\leq\frac{1}{c_n}\downarrow 0$, we get $\lims_{n\to\infty}|\u_N(\mu_n)-\u_N(\nu_n)|=0$ and thus
\begin{equation}
\label{eq:ass}
\u_N(\nu)\leq\limi_{n\to\infty}\u_N(\nu_n)=\limi_{n\to\infty}\u_N(\mu_n)=\u_N(\mu).
\end{equation}
Write $\nu=\eta\mm+\nu^s$, with $\nu^s\perp\mm$ and notice that by construction it holds $\nu\leq\mu$, so that $\nu^s=0$ and $\eta\leq\rho$ $\mm$-a.e.. Also, since $u_N$ is strictly decreasing, \eqref{eq:ass} forces $\eta=\rho$ $\mm$-a.e., so that  $\nu(B)=\int\eta\,\d\mm=\int\rho\,\d\mm=1$. But this is a contradiction, because from \eqref{eq:lossmass} we have $\nu_n(B)\leq 1-C$ for every $n\in\N$, so that $\nu(B)\leq1-C$ as well.
\end{proof}
Notice that as a consequence of the convexity of $\u_N$ (w.r.t. affine interpolation) and its weak lower semicontinuity, the following implication holds:
\begin{equation}
\label{eq:sempre}
\begin{array}{l}
\left\{
\begin{array}{l}
\ppi\in\prob{\geo(X)}\\
\supp((\e_t)_\sharp\ppi)\subset B,\quad\forall t\in[0,1], \ \textrm{ for some bounded closed set } B\textrm{ s.t. }\mm(B)<\infty\\
\u_N((\e_t)_\sharp\ppi)\to\u_N((\e_0)_\sharp\ppi),\quad\textrm{ as }t\downarrow0
\end{array}
\right\}\\
\phantom{ciccia}\\
\qquad\qquad\displaystyle{\Rightarrow\qquad\qquad \u_N\left(\frac1t\int_0^t(\e_s)_\sharp\ppi\,\d s\right)\to\u_N((\e_0)_\sharp\ppi),\quad\textrm{ as }t\downarrow0}
\end{array}
\end{equation}

A direct consequence of Lemma \ref{le:scemo}, is the following version of the metric Brenier theorem proved in \cite{Ambrosio-Gigli-Savare11}. Notice that the argument we use here is essentially the same used in  \cite{Ambrosio-Gigli-Savare11}.
\begin{proposition}[A variant of the metric Brenier theorem]\label{prop:varbrenier} Let $(X,\sfd,\mm)$ be as in \eqref{eq:mms}, $B\subset X$ a bounded closed set such that $\mm(B)<\infty$, $\mu_0,\mu_1\in\prob{ B}$ and  $N\in(1,\infty)$. Assume that $\mu_0\ll\mm$, that there exists $\ppi\in\gopt(\mu_0,\mu_1)$ such that  $\lim_{t\downarrow 0}\u_N((\e_t)_\sharp\ppi)=\u_N(\mu_0)$ and $\supp((\e_t)_\sharp\ppi)\subset B$ for any $t\in[0,1]$.  Let $\varphi$ be a Kantorovich potential associated to $\mu_0,\mu_1$ and assume that $\varphi$  is Lipschitz on bounded subsets of $X$. 

Then for every $\tilde\ppi\in\gopt(\mu_0,\mu_1)$ it holds
\begin{equation}
\label{eq:bm}
\sfd(\gamma_0,\gamma_1)=\lip^+(\varphi)(\gamma_0),\qquad\tilde\ppi-a.e.\ \gamma.
\end{equation}
\end{proposition}
\begin{proof}
Fix $x\in  X$ and notice that for every $y\in\partial^c\varphi(x)$ it holds
\[
\begin{split}
\varphi(x)&=\frac{\sfd^2(x,y)}2-\varphi^c(y),\\
\varphi(z)&\leq\frac{\sfd^2(z,y)}2-\varphi^c(y),\qquad\forall z\in X.
\end{split}
\]
Therefore $\varphi(z)-\varphi(x)\leq \frac{\sfd^2(z,y)}2-\frac{\sfd^2(x,y)}2\leq\sfd(x,z)\frac{\sfd(z,y)+\sfd(x,y)}2$, so that dividing by $\sfd(x,z)$ and letting $z\to x$, since $y\in\partial^c\varphi(x)$ is arbitrary we get
\[
\lip^+(\varphi)(x)\leq \inf_{y\in\partial^c\varphi(x)} \sfd(x,y).
\]
In particular, if $\tilde\ppi$ is an arbitrary plan in $\gopt(\mu_0,\mu_1)$, from the above inequality and the fact that $\gamma_1\in\partial^c\varphi(\gamma_0)$ for $\tilde\ppi$-a.e. $\gamma$ we deduce
\[
\lip^+(\varphi)(\gamma_0)\leq \sfd(\gamma_0,\gamma_1),\qquad\tilde\ppi-a.e.\ \gamma.
\]
Thus to conclude it is sufficient to show that
\[
\int\lip^+(\varphi)^2\,\d\mu_0\geq W_2^2(\mu_0,\mu_1).
\]
To this aim, let $\ppi\in\gopt(\mu_0,\mu_1)$ be as in the hypothesis, and use again the fact that $\varphi$ is a Kantorovich potential to get that for any $t\in(0,1]$ it holds
\begin{equation}
\label{eq:p1}
\varphi(\gamma_0)-\varphi(\gamma_t)\geq\frac{\sfd^2(\gamma_0,\gamma_1)}2-\frac{\sfd^2(\gamma_t,\gamma_1)}2=\sfd^2(\gamma_0,\gamma_1)(t-t^2/2),\qquad\ppi-a.e.\ \gamma.
\end{equation}
Dividing by $\sfd(\gamma_0,\gamma_t)=t\sfd(\gamma_0,\gamma_1)$, squaring and integrating w.r.t. $\ppi$ we get
\[
\limi_{t\downarrow 0}\int\left(\frac{\varphi(\gamma_0)-\varphi(\gamma_t)}{\sfd(\gamma_0,\gamma_t)}\right)^2\,\d\ppi(\gamma)\geq\int\sfd^2(\gamma_0,\gamma_1)\,\d\ppi(\gamma)=W_2^2(\mu_0,\mu_1).
\]
Now notice that since $\varphi$ is locally Lipschitz, $\lip^+(\varphi)$ is an upper gradient for $\varphi$, hence it holds
\[
\begin{split}
\int\left(\frac{\varphi(\gamma_0)-\varphi(\gamma_t)}{\sfd(\gamma_0,\gamma_t)}\right)^2\,\d\ppi(\gamma)&\leq\int\frac1{t^2}\left(\int_0^t\lip^+(\varphi)(\gamma_s)\,\d s\right)^2\d\ppi(\gamma)\\
&\leq\frac1t\iint_0^t\lip^+(\varphi)^2(\gamma_s)\,\d s\,\d\ppi(\gamma)=\frac1t\int_0^t\int \lip^+(\varphi)^2\,\d(\e_s)_\sharp\ppi\,\d s.
\end{split}
\]
From our assumptions and \eqref{eq:sempre} we get that  $\u_N\Big(\frac1t\int_0^t(\e_s)_\sharp\ppi\,\d s\Big)\to \u_N(\mu)$ as $t\downarrow0$. Also, notice that by hypothesis the Borel function $\lip^+(\varphi)^2$ is bounded on $B$. Thus we can apply  Lemma \ref{le:scemo} and get
\[
\lim_{t\downarrow0}\frac1t\int_0^t\int \lip^+(\varphi)^2\,\d(\e_s)_\sharp\ppi\d s=\int\lip^+(\varphi)^2\,\d\mu_0,
\]
and the proof is completed.
\end{proof}

The following proposition is the variant of Proposition \ref{prop:cambiata}: notice that we are assuming \eqref{eq:cheeger2} to hold, replacing Sobolev functions with Lipschitz ones and the assumption of bounded compression with $(\e_0)_\sharp\ppi\ll\mm$ and $\u_N((\e_t)_\sharp\ppi)\to\u_N((\e_0)_\sharp\ppi)$ as $t\downarrow 0$.
\begin{proposition}\label{prop:tecnvar}
Let $(X,\sfd,\mm)$ be as in \eqref{eq:mms}. Assume that  \eqref{eq:cheeger2} holds and let  $\ppi\in\prob{C([0,1],X)}$ be such that   $(\e_0)_\sharp\ppi\ll\mm$, and  $\u_N((\e_t)_\sharp\ppi)\to\u_N((\e_0)_\sharp\ppi)$ as $t\downarrow0$ for some $N\in(1,\infty)$. Assume furthermore that for some bounded and closed set $B\subset X$ with $\mm(B)<\infty$ it holds $\supp((\e_t)_\sharp\ppi)\subset B$ for any $t\in[0,1]$.

Then for every function $f:X\to\R$ Lipschitz on bounded sets and every $p\in(1,\infty)$ it holds
\begin{equation}
\label{eq:pervarhorver}
\lims_{t\downarrow0}\int\frac{f(\gamma_t)-f(\gamma_0)}t\,\d\ppi(\gamma)\leq \frac{\|\weakgrad f\|^p_{L^p(X,(\e_0)_\sharp\sppi)}}p+\frac{\|\ppi\|_q^q}q,
\end{equation}
where $q$ is the conjugate exponent of $p$.
\end{proposition}
\begin{proof} If $\|\ppi\|_q=\infty$ there is nothing to prove. Thus we can assume $\|\ppi\|_q<\infty$ which in particular implies that for some $T\in(0,1]$ the plan $({\rm restr}_0^T)_\sharp\ppi$ is concentrated on absolutely continuous curves and that $(\e_t)_\sharp\ppi$ weakly converges to $(\e_0)_\sharp\ppi$ in duality with $C_b(B)$. 

Since $f$ is locally Lipschitz, $\lip(f)$ is an upper gradient for $f$ and thus
\[
\frac{f(\gamma_t)-f(\gamma_0)}t\leq \frac1t\int_0^t \lip(f)(\gamma_s)|\dot\gamma_s|\,\d s\leq \frac{\int_0^t\lip( f)^p(\gamma_s)\,\d s}{tp}+\frac{\int_0^t|\dot\gamma_s|^q\,\d s}{tq},
\]
holds for any absolutely continuous curve $\gamma$ and $t\in(0,1]$. Integrate over $\ppi$ and let $t\downarrow 0$ to get
\[
\lims_{t\downarrow0}\int\frac{f(\gamma_t)-f(\gamma_0)}t\,\d\ppi(\gamma)\leq \frac1p\lims_{t\downarrow0}\frac1t\iint_0^t\lip(f)^p(\gamma_s)\,\d s\,\d\ppi(\gamma)  +\frac{\|\ppi\|_q^q}q.
\]
From our assumptions on $\ppi$ and \eqref{eq:sempre}  we get that $\u_N\left(\frac1t\int_0^t(\e_s)_\sharp\ppi\,\d s\right)\to\u_N((\e_0)_\sharp\ppi)$. Since $(\e_0)_\sharp\ppi\ll\mm$ and everything is taking place in the closed and bounded set $B$, we can apply Lemma \ref{le:scemo} to the bounded Borel function $\lip( f)^p$ to get that
\[
\lims_{t\downarrow0}\frac1t\iint_0^t\lip( f)^p(\gamma_s)\,\d s\,\d\ppi(\gamma) =\lims_{t\downarrow0}\int \lip( f)^p\,\d\left(\frac1t\int_0^t(\e_s)_\sharp\ppi\,\d s\right)=\int\lip(f)^p\,\d(\e_0)_\sharp\ppi.
\]
Finally, we use \eqref{eq:cheeger2} to replace $\lip(f)$ with $\weakgrad f$.
\end{proof}
This proposition suggest the following variant of Definition \ref{def:represent}:
\begin{definition}[Plans weakly representing gradients]
Let $(X,\sfd,\mm)$ be as in \eqref{eq:mms}, $p,q\in(1,\infty)$ conjugate exponents, $g:X\to\R$ a function Lipschitz on bounded sets and $\ppi\in\prob{C([0,1],X)}$. We say that $\ppi$  weakly $q$-represents $\nabla g$ provided:
\begin{itemize}
\item[i)] $\|\ppi\|_q<\infty$,
\item[ii)] $(\e_0)_\sharp\ppi\ll\mm$ and $\u_N((\e_t)_\sharp\ppi)\to\u_N((\e_0)_\sharp\ppi)$ as $t\downarrow0$ for some $N\in(1,\infty)$,
\item[iii)] for some bounded and closed set $B\subset X$ with $\mm(B)<\infty$ it holds $\supp((\e_t)_\sharp\ppi)\subset B$ for any $t\in[0,1]$,
\item[iv)] it holds
\begin{equation}
\label{eq:weakrepr}
\limi_{t\downarrow0}\int\frac{g(\gamma_t)-g(\gamma_0)}t\,\d\ppi(\gamma)\geq \frac{\|\weakgrad g\|^p_{L^p(X,(\e_0)_\sharp\sppi)}}p+\frac{\|\ppi\|_q^q}q.
\end{equation}
\end{itemize}
\end{definition}
Note: it is not the representation of $\nabla g$ to be weak in this definition, but the assumption that $\ppi$ has bounded compression that has been weakened into $(ii)$ above.

As a direct consequence of the metric Brenier theorem \ref{prop:varbrenier} we have the following simple and important result.
\begin{corollary}\label{cor:bm}
With the same assumptions and notations of Proposition \ref{prop:varbrenier} above, assume furthermore that \eqref{eq:cheeger2} holds. Then $\ppi$  weakly 2-represents $\nabla(-\varphi)$.
\end{corollary}
\begin{proof}
Dividing \eqref{eq:p1} by $t$, integrating w.r.t. $\ppi$ and letting $t\downarrow0$ we get
\[
\limi_{t\downarrow0}\int\frac{\varphi(\gamma_0)-\varphi(\gamma_t)}t\,\d\ppi(\gamma)\geq \int \sfd^2(\gamma_0,\gamma_1)\,\d\ppi(\gamma).
\]
On the other hand, since $\ppi$ is concentrated on geodesics it holds
\[
\frac1t\iint_0^t|\dot\gamma_s|^2\,\d s\,\d\ppi(\gamma)=\int \sfd^2(\gamma_0,\gamma_1)\,\d\ppi(\gamma),\qquad\forall t\in(0,1].
\]
Finally, \eqref{eq:bm}, \eqref{eq:cheeger2} and the fact that $(\e_0)_\sharp\ppi\ll\mm$  give $\int\weakgrad\varphi^2\,\d(\e_0)_\sharp\ppi=\int \sfd^2(\gamma_0,\gamma_1)\,\d\ppi(\gamma)$.
\end{proof}
We conclude with the variant of Theorem \ref{thm:horver} which we will use in the proof of the Laplacian comparison estimates.
\begin{proposition}[Variant of horizontal-vertical derivation]\label{prop:horvervar}
Let $(X,\sfd,\mm)$ be as in \eqref{eq:mms} and assume that  \eqref{eq:cheeger2} holds. Let $p,q\in(1,\infty)$ be conjugate exponents and $N\in(1,\infty)$. Let $f,g:X\to\R$ be two functions Lipschitz continuous on bounded subsets of $X$ and $\ppi\in\prob{C([0,1],X}$ a plan weakly $q$-representing $\nabla g$.

Then the same conclusions of Theorem \ref{thm:horver}  hold, i.e.
\begin{equation}
\label{eq:horvervar}
\begin{split}
\int D^+f(\nabla g)\weakgrad g^{p-2}\,\d(\e_0)_\sharp\ppi&\geq\lims_{t\downarrow 0}\int \frac{f(\gamma_t)-f(\gamma_0)}{t}\,\d\ppi(\gamma)\\
&\geq\limi_{t\downarrow 0}\int \frac{f(\gamma_t)-f(\gamma_0)}{t}\,\d\ppi(\gamma)\geq \int D^-f(\nabla g)\weakgrad g^{p-2}\,\d(\e_0)_\sharp\ppi.
\end{split}
\end{equation}
\end{proposition}
\begin{proof}
The proof is the same as the one of Theorem \ref{thm:horver}, with Proposition \ref{prop:tecnvar} replacing Proposition \ref{prop:cambiata}. Indeed notice that the assumptions of Proposition \ref{prop:tecnvar} are fulfilled, so that from \eqref{eq:pervarhorver} applied to the function $g+\eps f$, $\eps\in\R$, we get
\begin{equation}
\label{eq:stazione}
\lims_{t\downarrow 0}\int\frac{(g+\eps f)(\gamma_t)-(g+\eps f)(\gamma_0)}{t}\,\d\ppi(\gamma)\leq \frac{\|\weakgrad {(g+\eps f)}\|_{L^p(X,(\e_0)_\sharp\sppi)}^p}{p}+\frac{\|\ppi\|^q_q}{q},
\end{equation}
while from our assumptions on $g$ and $\ppi$ we know that
\begin{equation}
\label{eq:stazione2}
\limi_{t\downarrow 0}\int\frac{g(\gamma_t)-g(\gamma_0)}{t}\,\d\ppi(\gamma)\geq \frac{\|\weakgrad g\|_{L^p(X,(\e_0)_\sharp\sppi)}^p}{p}+\frac{\|\ppi\|^q_q}{q}.
\end{equation}
Subtract  \eqref{eq:stazione2} from  \eqref{eq:stazione} to get
\[
 \lims_{t\downarrow0}\eps\int\frac{f(\gamma_t)-f(\gamma_0)}{t}\,\d\ppi(\gamma)\leq \int \frac{\weakgrad{(g+\eps f)}^p -\weakgrad{g}^p}{p}\,\d (\e_0)_\sharp\ppi.
\]
The thesis follows.
\end{proof}

\subsection{Proof of Laplacian comparison}
We are now ready to prove the main comparison result of this paper. The proof will come by combining a bound from below and a bound from above of the derivative of the internal energy along a geodesic.

We start with the bound from below, which is the one technically more involved. We will use all the machinery developed in the previous section.
Let us introduce the pressure functionals $\rp_N:[0,\infty)\to\R$, by putting $\rp_N(z):=zu'_N(z)-u_N(z)$, i.e.
\[
\rp_N(z)=\tfrac1N z^{1-\frac1N}.
\]
\begin{proposition}[Bound from below on the derivative of the internal energy]\label{prop:boundbasso}
Let $(X,\sfd,\mm)$ be as in \eqref{eq:mms}  such that \eqref{eq:cheeger2} holds and $\mm$ is finite on bounded sets.

Let $\mu=\mu_0\in \prob X$ be a measure with bounded support such that  $\mu_0\ll\mm$, say $\mu_0=\rho\mm$. Assume that $\supp(\mu_0)=\overline\Omega$, with $\Omega$ open such that $\mm(\partial\Omega)=0$. Assume also that the restriction of  $\rho$ to $\overline\Omega$ is Lipschitz and bounded from below by a positive constant.  Also, let $\mu_1\in\prob X$ and $\ppi\in\gopt(\mu_0,\mu_1)$. Assume that  $\supp((\e_t)_\sharp\ppi)\subset \overline\Omega$ for every $t\in[0,1]$ and that  $\u_N((\e_t)_\sharp\ppi)\to \u_N(\mu_0)$, as $t\downarrow0$. 

Then it holds
\begin{equation}
\label{eq:bounddalbasso}
\limi_{t\downarrow0}\frac{\u_N((\e_t)_\sharp\ppi)-\u_N((\e_0)_\sharp\ppi)}t\geq-\int_\Omega D^+(\rp_N(\rho))(\nabla\varphi)\,\d\mm,
\end{equation}
where $\varphi$ is any Kantorovich potential from $\mu_0$ to $\mu_1$ which is Lipschitz on bounded sets and the integrand $D^+(\rp_N(\rho))(\nabla\varphi)$ is defined according to Definition \ref{def:dfgg} in the space $(\overline\Omega,\sfd,\mm\restr\Omega)$ (this remark is necessary because in general $\rho$, being discontinuous along $\partial\Omega$, does not belong to any $\s^p(X,\sfd,\mm)$).
\end{proposition}
\begin{proof} We work on the space $(\overline\Omega,\sfd,\mm\restr\Omega)$. Since $\rho,\rho^{-1}:\overline\Omega\to\R$  are bounded, the function $u_N'(\rho):\overline\Omega\to\R$ is  bounded.  Thus for every $\nu\in\prob{\overline\Omega}$ absolutely continuous w.r.t. $\mm\restr{\Omega}$, the convexity of $u_N$ gives $\u_N(\nu)-\u_N(\mu)\geq \int_\Omega u_N'(\rho)(\frac{\d\nu}{\d\mm}-\rho)\,\d\mm$. Then a simple approximation argument based on the continuity of $\rho$ gives
\[
\u_N(\nu)-\u_N(\mu)\geq \int_{\overline\Omega} u_N'(\rho)\,\d\nu-\int_{\overline\Omega} u_N'(\rho)\,\d\mu,\qquad\forall\nu\in\prob{\overline\Omega}.
\]
Plugging $\nu:=(\e_t)_\sharp\ppi$, dividing by $t$ and letting $t\downarrow0$ we get
\[
\limi_{t\downarrow 0}\frac{\u_N((\e_t)_\sharp\ppi)-\u_N((\e_0)_\sharp\ppi)}t\geq\limi_{t\downarrow0}\int\frac{u_N'(\rho)\circ\e_t-u_N'(\rho)\circ\e_0}t\,\d\ppi.
\]
Now observe that for any Lipschitz function $f:X\to\R$ and $x\in\Omega$, the values of $\lip(f)(x),\lip^\pm(f)(x)$ calculated in $(X,\sfd)$ and $(\overline\Omega,\sfd)$ are the same. Taking into account Proposition \ref{prop:srestr} and the fact that $\mm(\partial\Omega)=0$, we can conclude that \eqref{eq:cheeger2} holds also in the space $(\overline\Omega,\sfd,\mm\restr{\Omega})$. 

In particular,  Corollary \ref{cor:bm} is applicable and we have that $\ppi$ weakly 2-represents $\nabla(-\varphi)$. Then, since the maps $u_N'(\rho),\varphi:\overline\Omega\to\R$ are Lipschitz, we can apply Proposition \ref{prop:horvervar} on the space $(\overline\Omega,\sfd,\mm\restr{\Omega})$ with $p=q=2$ to get
\[
\limi_{t\downarrow0}\int\frac{u_N'(\rho)\circ\e_t-u_N'(\rho)\circ\e_0}t\,\d\ppi\geq \int_\Omega D^-(u_N'(\rho))(\nabla(-\varphi))\rho\,\d\mm.
\]
To conclude, notice that $\rp_N'(z)=zu_N''(z)$ and apply twice the chain rule \eqref{eq:chainf}:
\[
 \int_\Omega D^-(u_N'(\rho))(\nabla(-\varphi))\rho\,\d\mm=-\int_\Omega \rho u_N''(\rho)D^+\rho(\nabla\varphi)\,\d\mm=-\int_\Omega D^+(\rp_N(\rho))(\nabla\varphi)\,\d\mm.
\]
\end{proof}
For $K\in \R$, $N\in(1,\infty)$, we introduce the functions $\tilde\tau_{K,N}:[0,\infty)\to \R$ by
\[
\tilde\tau_{K,N}(\theta):=
\left\{
\begin{array}{ll}
\dfrac 1N\left(1+\theta\sqrt{K(N-1)}\,{\rm cotan}\left(\theta\sqrt{\frac{K}{N-1}}\right) \right),&\qquad\textrm{if }K>0,\\
\\
1,&\qquad\textrm{if }K=0,\\
\\
\dfrac 1N\left(1+\theta\sqrt{-K(N-1)}\,{\rm cotanh}\left(\theta\sqrt{\frac{-K}{N-1}}\right) \right),&\qquad\textrm{if }K<0,
\end{array}
\right.
\]
so that $\tilde\tau_{K,N}(\theta)=\lim_{t\downarrow0}\frac{\tau^{(1)}_{K,N}(\theta)-\tau^{(1-t)}_{K,N}(\theta)}{t}$ (provided $K\theta^2<(N-1)\pi^2$ if $K>0$). Notice that if $K\leq 0$ then $\tilde\tau_{K,N}(\cdot)$ is bounded on bounded subsets of $[0,\infty)$, but for $K>0$ it blows up as $\theta$ approaches $\pi\sqrt{(N-1)/K}$. This creates a potential complication about the integrability of $\tilde\tau_{K,N}(\sfd(\gamma_0,\gamma_1))$ w.r.t. optimal geodesic plans $\ppi$ which must be dealt with.  This is the scope of the following lemma.
\begin{lemma}\label{le:integrabile}
Let $(X,\sfd,\mm)$ be a $\CD(K,N)$ space, $K\in \R$, $N\in(1,\infty)$, and $\mu,\nu\in\prob X$ with bounded supports and  $\mu\leq C\mm$ for some $C>0$. Then for every  $\ppi\in\gopt(\mu,\nu)$  the function $\gamma\mapsto \tilde\tau_{K,N}(\sfd(\gamma_0,\gamma_1))$ is in $L^1(\ppi)$ and the convergence of $\frac{\tau^{(1)}_{K,N}(\sfd(\gamma_0,\gamma_1))-\tau^{(1-t)}_{K,N}(\sfd(\gamma_0,\gamma_1))}{t}$ to $\tilde\tau_{K,N}(\sfd(\gamma_0,\gamma_1))$ as $t\to 0$ is dominated in $L^1(\ppi)$.
\end{lemma}
\begin{proof}
If $K\leq 0$ there is nothing to prove, thus let $K>0$ and recall that in this case $(X,\sfd)$ is compact (Proposition \ref{prop:basecd}). It is immediate to verify that the domination of $\left|\frac{\tau^{(1)}_{K,N}(\sfd(\gamma_0,\gamma_1))-\tau^{(1-t)}_{K,N}(\sfd(\gamma_0,\gamma_1))}{t}\right|$ is a consequence of the fact that $\gamma\mapsto \tilde\tau_{K,N}(\sfd(\gamma_0,\gamma_1))$ is in $L^1(\ppi)$, thus we focus on this second problem.

If $\sup_{\gamma\in\supp(\sppi)}\sfd(\gamma_0,\gamma_1)<\pi\sqrt{(N-1)/K}$ there is nothing to prove. Otherwise, taking into account the compactness of $(X,\sfd)$, there exists $\overline\gamma\in\supp(\ppi)$ such that $\sfd(\gamma_0,\gamma_1)=\pi\sqrt{(N-1)/K}=:\sf D$. Put $x_E:=\gamma_0$, $x_O:=\gamma_1$. We claim\footnote{this claim has been proved by Ohta in \cite{Ohta07bis}, we are reporting here the proof for completeness} that for any $x\in X$ it holds
\begin{equation}
\label{eq:diam}
\sfd(x,x_E)+\sfd(x,x_O)={\sf D}.
\end{equation}
If not, for some $r\in(0,{\sf D})$ the two disjoint balls $B_{r}(x_E)$, $B_{{\sf D}-r}(x_O)$ would not cover some neighborhood of $x$, so that
\begin{equation}
\label{eq:nonva}
\mm(B_{r}(x_E))+\mm(B_{{\sf D}-r}(x_O))=\mm\big(B_{r}(x_E)\cup B_{{\sf D}-r}(x_O)\big)<\mm(X).
\end{equation}
However, taking into account   the Bishop-Gromov volume comparison estimates (inequality \eqref{eq:BGv}) we have
\[
\mm(B_{r}(x_E))+\mm(B_{{\sf D}-r}(x_O))\geq \left( \frac{\int_0^r\sin(t\pi/{\sf D})^{N-1}\,\d t}{\int_0^{\sf D}\sin(t\pi/{\sf D})^{N-1}\,\d t}+ \frac{\int_0^{{\sf D}-r}\sin(t\pi/{\sf D})^{N-1}\,\d t}{\int_0^{\sf D}\sin(t\pi/{\sf D})^{N-1}\,\d t}\right)\!\mm(X)= \mm(X),
\]
which contradicts \eqref{eq:nonva}. Hence our claim is true. The argument also yields
\begin{equation}
\label{eq:palle}
\mm(B_{r}(x_E))=\frac{\int_0^r\sin(t\pi/{\sf D})^{N-1}\,\d t}{\int_0^{\sf D}\sin(t\pi/{\sf D})^{N-1}\,\d t}\mm(X),\qquad\forall r\in[0,{\sf D}]
\end{equation}
Now we claim that for every $\gamma\in\supp(\ppi)$ it holds either $\gamma_0=x_E$ or $\gamma_1=x_O$. Indeed, by cyclical monotonicity and \eqref{eq:diam} we have
\begin{equation}
\label{eq:benebene}
\begin{split}
\sfd^2(\gamma_0,\gamma_1)+{\sf D}^2&\leq \sfd^2(\gamma_0,x_O)+\sfd^2(x_E,\gamma_1)\\
&= {\sf D}^2+\sfd^2(\gamma_0,x_E)-2 {\sf D}\sfd(\gamma_0,x_E) +\sfd^2(x_E,\gamma_1)\\
&\leq {\sf D}^2+\sfd^2(\gamma_0,x_E)-2 \sfd(x_E,\gamma_1)\sfd(\gamma_0,x_E) +\sfd^2(x_E,\gamma_1)\\
&={\sf D}^2 +\big( \sfd(x_E,\gamma_1)-\sfd(\gamma_0,x_E) \big)^2,
\end{split}
\end{equation}
which yields $\sfd(\gamma_0,\gamma_1)\leq \big|\sfd(x_E,\gamma_1)-\sfd(\gamma_0,x_E) \big|$. Hence  the triangle inequality forces the inequalities in \eqref{eq:benebene} to be equalities. The claim follows observing that the second inequality in \eqref{eq:benebene} is an equality if and only if either $\sfd(\gamma_0,x_E)=0$ or $\sfd(\gamma_1,x_E)={\sf D}$ and in this latter case, \eqref{eq:diam} yields $\sfd(\gamma_1,x_O)=0$.

By assumption, $(\e_0)_\sharp\ppi\ll\mm$, so that, since $\mm$ has no atoms, for $\ppi$-a.e. $\gamma$ it holds $\gamma_0\neq x_E$ and what we just proved yields $\gamma_1=x_O$ for $\ppi$-a.e. $\gamma$. Thus
\[
\tilde\tau_{K,N}(\sfd(\gamma_0,\gamma_1))=\tilde\tau_{K,N}(\sfd(\gamma_0,x_O)),\qquad\ppi-a.e.\ \gamma,
\]
and, taking the assumption $\mu\leq C\mm$ and the defining formula of $\tilde\tau_{K,N}(\cdot)$ in mind, our thesis reduces to prove that 
\[
\int\limits_{B_{{\sf D}/2}(x_E)} \frac{1}{\sin\big(\sfd(x,x_O)\sqrt{\frac K{N-1}}\big)}\,\d\mm(x)<\infty.
\]
Let $T:X\to[0,{\sf D}]$ be given by $T(x):=\sfd(x,x_E)$ and observe that \eqref{eq:palle} gives 
\[
\d T_\sharp\mm(r)=\left(\frac{\sin(r\pi/{\sf D})^{N-1}}{\int_0^{\sf D}\sin(t\pi/{\sf D})^{N-1}\,\d t}\mm(X)\right)\d\mathcal L^1\restr{[0,{\sf D}]}(r).
\]
Therefore
\[
\begin{split}
\int\limits_{B_{{\sf D}/2}(x_E)} \frac{1}{\sin\big(\sfd(x,x_O)\sqrt{\frac K{N-1}}\big)}\,\d\mm(x)&=\int\limits_{B_{{\sf D}/2}(x_E)} \frac{1}{\sin\big(\sfd(x,x_E)\frac\pi{\sf D}\big)}\,\d\mm(x)\\
&=\frac{\mm(X)}{\int_0^{\sf D}\sin(t\pi/{\sf D})^{N-1}\,\d t}\int_0^{{\sf D}/2}\sin\big(r\frac\pi{\sf D}\big)^{N-2}\,\d r.
\end{split}
\]
The conclusion follows from the fact that $N>1$.
\end{proof}
\begin{proposition}[Bound from above on the derivative of the internal energy]\label{prop:boundabove} Let $(X,\sfd,\mm)$ be as in \eqref{eq:mms},  $K\in \R$ and $N\in(1,\infty)$. Let $\ppi\in\prob{\geo(X)}$ be such that $(\e_i)_\sharp\ppi$ has bounded support, $i=0,1$, and $(\e_0)_\sharp\ppi\ll\mm$ with bounded density. Assume also that $\ppi$ satisfies
\begin{equation}
\label{eq:pezzocd}
\u_{N}((\e_t)_\sharp\ppi)\leq-\int\tau^{(1-t)}_{K,N}\big(\sfd(\gamma_0,\gamma_1)\big)\rho^{-\frac1{N}}(\gamma_0)\,\d\ppi(\gamma),\qquad\forall t\in[0,1],
\end{equation}
where $\rho$ is the density of $(\e_0)_\sharp\ppi$.

Then
\begin{equation}
\label{eq:dersup}
\lims_{t\downarrow 0}\frac{\u_N((\e_t)_\sharp\ppi)-\u_N((\e_0)_\sharp\ppi)}{t}\leq \int\tilde\tau_{K,N}\big(\sfd(\gamma_0,\gamma_1)\big)\rho^{-\frac1N}(\gamma_0)\,\d\ppi(\gamma).
\end{equation}
\end{proposition}
\begin{proof}
From $\u_N((\e_0)_\sharp\ppi)=-\int\rho^{-\frac1N}(\gamma_0)\,\d\ppi(\gamma)$, and \eqref{eq:pezzocd}  we get
\[
\frac{\u_N((\e_t)_\sharp\ppi)-\u_N((\e_0)_\sharp\ppi)}{t}\leq\int\rho^{-\frac1N}\bigg(\frac{1-\tau^{(1-t)}_{K,N}\big(\sfd(\gamma_0,\gamma_1)\big)}t\bigg)\,\d\ppi(\gamma),
\]
and the conclusion follows by letting $t\downarrow0$ and using the dominate convergence  theorem and Lemma \ref{le:integrabile} to pass the limit inside the integral.
\end{proof}
The combination of \eqref{eq:bounddalbasso} and \eqref{eq:dersup} will produce the Laplacian comparison estimate. Before concluding, however, we need a result which grants that Proposition \ref{prop:boundbasso} can be applied in a sufficiently large class of situations, the tricky part being to ensure that $\supp((\e_t)_\sharp\ppi)\subset\overline\Omega$ for any $t\in[0,1]$. This is the scope of the following lemma. 
\begin{lemma}\label{le:cheppalle}
Let $(X,\sfd,\mm)$ be a proper  geodesic metric space with $\mm$ locally finite. Let $\varphi:X\to\R$ a locally Lipschitz $c$-concave function and $B\subset X$ a compact set. Then there exists another locally Lipschitz $c$-concave function $\tilde\varphi:X\to\R$ and a bounded open set $\Omega \supset B$ such that the following are true. 
\begin{itemize}
\item[i)] $\tilde\varphi=\varphi$ on $B$.
\item[ii)] For any $x\in X$, the set $\partial^c\tilde\varphi(x)$ is non-empty.
\item[iii)] For any $x\in\Omega$, $y\in\partial^c\tilde\varphi(x)$ and $\gamma\in\geo(X)$ connecting $x$ to $y$ it holds $\gamma_t\in\Omega$   for any $t\in[0,1]$.
\item[iv)] $\mm(\partial\Omega)=0$.
\end{itemize}
\end{lemma}
\begin{proof} By Proposition \ref{prop:fg} we know that for any $x\in X$ the set $\partial^c\varphi(x)$ is non-empty and that the set $C\subset X$ defined by
\[
C:=\Big\{y\in X\ :\ y\in\partial^c\varphi(x),\ \textrm{ for some }x\in  B\Big\},
\]
is compact. Define $\tilde\varphi:X\to\R$ by
\[
\tilde\varphi(x):=\inf_{y\in C}\frac{\sfd^2(x,y)}2-\varphi^c(y).
\]
By definition, $\tilde\varphi$ is $c$-concave, locally Lipschitz and $\tilde\varphi=\varphi$ on $B$. By Proposition \ref{prop:fg} again we get that $(ii)$ is fulfilled as well. Let $a:=\sup_{x\in B}\tilde \varphi(x)$ and for $b>a$ define
\[
\Omega_b:=\{x\ :\ \tilde\varphi(x)<b\}.
\]
Clearly $\Omega_b$ is an open  set which contains $B$. We claim that it is also bounded. Indeed since $\varphi^c$ is $c$-concave and $C$ is bounded, it holds $\sup_{y\in C}\varphi^c(y)<\infty$, so that the boundedness of $\Omega_b$ follows from the inequality  $\tilde\varphi(x)\geq\frac{\sfd^2(x,C)}{2}-\sup_C\varphi^c$. Now pick an arbitrary $x\in \Omega_b$ and $y\in \partial^c\tilde\varphi(x)$. From
\[
\begin{split}
\tilde\varphi(x)&=\frac{\sfd^2(x,y)}2-\tilde\varphi^c(y),\\
\tilde\varphi(z)&\leq\frac{\sfd^2(z,y)}2-\tilde\varphi^c(y),\qquad\forall z\in X,
\end{split}
\]
we get that  for any $z\in X$ such that $\sfd(z,y)\leq \sfd(x,y)$ it holds $\tilde\varphi(z)\leq \tilde\varphi(x)$. In particular, this applies to $z:=\gamma_t$, with $\gamma\in\geo(X)$ connecting $x$ to $y$, thus $(iii)$ is proved as well.

It remains to show that $b$ can be chosen so that $\mm(\partial\Omega_b)=0$. To achieve this, notice that for $b\neq b'$ it holds $\partial\Omega_b\cap\partial\Omega_{b'}=\emptyset$ so that the $\sigma$-finiteness of $\mm$ yields that for any $b>a$ except at most a countable number it holds  $\mm(\partial\Omega_b)=0$.
\end{proof}
Notice that on $\CD(K,N)$ spaces, if $\varphi$ is a locally Lipschitz $c$-concave function, then Proposition \ref{prop:varbrenier}, \eqref{eq:cheeger2} and Lemma \ref{le:integrabile} give that $\tilde\tau_{K,N}(\weakgrad\varphi)\in L^1_{\rm loc}(X,\mm)$.

In the following theorem and the discussion thereafter, we say that $(X,\sfd,\mm)$ is infinitesimally strictly convex provided for some $q\in(1,\infty)$ it is $q$-infinitesimally strictly convex. In this case, since locally Lipschitz functions belong to $\s^p_{\rm loc}(X,\sfd,\mm)$ for any $p\in(1,\infty)$, for any $f,g$ locally Lipschitz it holds
\[
D^+f(\nabla g)=D^-f(\nabla g),\qquad\mm-a.e.,
\]
and we will denote this common quantity, as usual, by $Df(\nabla g)$. In particular, on these spaces for $g$ locally Lipschitz and in $D(\bd)$, the set $\bd g$ contains only one measure, which abusing a bit the notation we will indicate as $\bd g$.
\begin{theorem}[Comparison estimates]\label{thm:controllo}
Let $K\in\R$, $N\in(1,\infty)$ and  $(X,\sfd,\mm)$ be a $\CD(K,N)$ space. Assume that $(X,\sfd,\mm)$ is infinitesimally strictly convex.  Let $\varphi:X\to\R$ a locally Lipschitz $c$-concave function. Then 
\begin{equation}
\label{eq:comparison}
\varphi\in D(\bd)\qquad\qquad\textrm{and}\qquad\qquad\bd\varphi\leq N\,\tilde\tau_{K,N}(\weakgrad\varphi)\,\mm.
\end{equation}
\end{theorem}
\begin{proof}
Fix a compact set $B\subset X$ and use Lemma \ref{le:cheppalle} to find a $c$-concave function $\tilde\varphi$ and a bounded open set $\Omega$ satisfying $(i),(ii),(iii),(iv)$ of the statement.

Consider the multivalued map $\overline\Omega\ni x\mapsto T(x)\subset \prob{\overline\Omega}$ defined by: $\mu\in T(x)$ if and only if $\supp(\mu)\subset\partial^c\tilde\varphi(x)$. Notice that $(ii)$ of Lemma \ref{le:cheppalle} ensures that $T(x)\neq\emptyset$ for any $x\in\overline\Omega$ Recall that $\prob{\overline\Omega}$ endowed with the weak topology given by the duality with $C(\overline\Omega)$ is Polish  and observe that since $\partial^c\tilde\varphi$ is closed, we easily get that the graph of $T$ is a closed subset of $\overline\Omega\times \prob{\overline\Omega}$. Hence standard measurable selection arguments (see e.g. Theorem 6.9.3. in \cite{Bogachev07}) ensure that there exists a Borel map $x\mapsto \eta_x$ such that $\eta_x\in T(x)$ for any $x\in\overline\Omega$.

Fix a bounded function $\rho:X\to [0,\infty)$ such that $\int\rho\,\d\mm=1$, $\rho\equiv 0$ on $X\setminus\overline\Omega$ and $\rho\restr{\overline\Omega}$ is Lipschitz and bounded from below by a positive constant. Define $\mu,\nu\in\prob{\overline\Omega}$ by $\mu:=\rho\mm$ and $\nu:=\int\eta_x\,\d\mu(x)$. Since $(X,\sfd,\mm)$ is a $\CD(K,N)$ space we know that there exists $\ppi\in\gopt(\mu,\nu)$ such that \eqref{eq:cd} holds. In particular
\begin{equation}
\label{eq:pesto2}
\u_{N}((\e_t)_\sharp\ppi)\leq-\int\tau^{(1-t)}_{K,N}\big(\sfd(\gamma_0,\gamma_1)\big)\rho^{-\frac1{N}}(\gamma_0)\,\d\ppi(\gamma),\qquad\forall t\in[0,1].
\end{equation}
Proposition \ref{prop:boundabove} grants
\begin{equation}
\label{eq:alto}
\lims_{t\downarrow 0}\frac{\u_{N}((\e_t)_\sharp\ppi)-\u_N(\mu)}t\leq \int\rho^{-\frac1N}(\gamma_0)\tilde\tau_{K,N}\big(\sfd(\gamma_0,\gamma_1)\big)\,\d\ppi(\gamma).
\end{equation}
By construction, $\tilde\varphi$ is a Kantorovich potential from $\mu$ to $\nu$, and   $\ppi\in\gopt(\mu,\nu)$.  Therefore for any $\gamma\in\supp(\ppi)$ it holds $\gamma_1\in\partial^c\tilde\varphi(\gamma_0)$. Using property $(iii)$ of Lemma \ref{le:cheppalle} we get that $\supp((\e_t)_\sharp\ppi)\subset\overline\Omega$ for any $t\in[0,1]$.

Furthermore, using \eqref{eq:pesto2} and the semicontinuity of $\u_N$ on $\prob{\overline\Omega}$ we get that $\u_N((\e_t)_\sharp\ppi)\to\u_N(\mu)$ as $t\downarrow0$, and thus Corollary \ref{cor:bm} gives that $\ppi$ weakly 2-represents $\nabla(-\tilde\varphi)$. 

Also, from $(iv)$ of Lemma \ref{le:cheppalle} we get $\mm(\partial\Omega)=0$.

Therefore all the assumptions of  Proposition \ref{prop:boundbasso} are satisfied (put $\mu$ in place of $\mu_0$ and $\tilde\varphi$ in place of $\varphi$), and  we get
\begin{equation}
\label{eq:basso}
\limi_{t\downarrow0}\frac{\u_N((\e_t)_\sharp\ppi)-\u_N((\e_0)_\sharp\ppi)}t\geq-\int_\Omega D^+(\rp_N(\rho))(\nabla\tilde\varphi)\,\d\mm.
\end{equation}
From inequalities \eqref{eq:alto} and \eqref{eq:basso}, the metric Brenier theorem \ref{prop:varbrenier}, \eqref{eq:cheeger2} and taking into account that $\rp_N(z)=\frac1Nz^{1-\frac1N}$ we deduce
\begin{equation}
\label{eq:bona}
-\frac1N\int_\Omega D^+(\rho^{1-\frac1N})(\nabla\tilde\varphi)\,\d\mm\leq \int \rho^{1-\frac1N}\,\tilde\tau_{K,N}(\weakgrad{\tilde\varphi})\,\d\mm.
\end{equation}
Now let $f:X\to\R$ be a non-negative Lipschitz function such that $\supp(f)\subset B$ and, for $\eps>0$, define $\rho_\eps:X\to[0,\infty)$ by $\rho_\eps:=c_\eps(f+\eps)^{\frac N{N-1}}\nchi_{\overline\Omega}$, where $c_\eps$ is such that $\int \rho_\eps\,\d\mm=1$. Apply \eqref{eq:bona} to $\rho:=\rho_\eps$ to obtain
\[
-\frac1N\int_\Omega D^+(f+\eps)(\nabla\tilde\varphi)\,\d\mm\leq \int_\Omega (f+\eps)\,\tilde\tau_{K,N}(\weakgrad{\tilde\varphi})\,\d\mm.
\]
Since $(X,\sfd,\mm)$ is infinitesimally strictly convex, with a cut-off argument and using $(i)$ of Proposition \ref{prop:srestr}  we have that $(\overline\Omega,\sfd,\mm)$ is infinitesimally strictly convex as well. Let $\eps\downarrow0$  and recall that from $(i)$ of Lemma \ref{le:cheppalle} we have $\tilde\varphi\equiv \varphi$ on $B$ to get
\begin{equation}
\label{eq:fine}
-\int_\Omega Df(\nabla\varphi)\,\d\mm\leq N\int f\,\tilde\tau_{K,N}(\weakgrad{\varphi})\,\d\mm.
\end{equation}
By $(ii)$ of Proposition \ref{prop:srestr} we have $\int_\Omega Df(\nabla\varphi)\,\d\mm=\int_X Df(\nabla\varphi)\,\d\mm$, in the sense that the integrand $Df(\nabla\varphi)$ can be equivalently computed in the space $(\overline\Omega,\sfd,\mm\restr\Omega)$ or in $(X,\sfd,\mm)$ (actually, for $\varphi$ in general it doesn't hold $\supp(\varphi)\subset\Omega$, but thanks to the locality property \eqref{eq:localgrad2} and a cut-off argument, the claim is still true). 

In summary, we proved that for any compact set $B\subset X$ and any $f:X\to[0,\infty)$ Lipschitz, non-negative and with $\supp(f)\subset B$ it holds
\begin{equation}
\label{eq:ottima}
-\int Df(\nabla\varphi)\,\d\mm\leq N\int f\,\tilde\tau_{K,N}(\weakgrad{\varphi})\,\d\mm.
\end{equation}
Since $B$ is arbitrary, we just proved  \eqref{eq:ottima} for any non-negative $f\in\test X$. The conclusion comes from Proposition \ref{prop:comparison}.
\end{proof}

\begin{corollary}[Laplacian of the distance]
Let $K\in\R$, $N\in(1,\infty)$ and  $(X,\sfd,\mm)$ an infinitesimally strictly convex  $\CD(K,N)$ space. For $x_0\in X$ denote by $\sfd_{x_0}:X\to[0,\infty)$ the function $x\mapsto\sfd(x,x_0)$. Then 
\begin{equation}
\label{eq:distq}
\frac{\sfd_{x_0}^2}2\in D(\bd),\qquad\textrm{with }\qquad\bd\frac{\sfd^2_{x_0}}{2}\leq N\,\tilde\tau_{K,N}(\sfd_{x_0})\,\mm\qquad\forall x_0\in X,
\end{equation}
and 
\begin{equation}
\label{eq:dist}
\sfd_{x_0}\in D(\bd,X\setminus\{x_0\})\qquad\textrm{with }\qquad\bd\sfd_{x_0}\restr{X\setminus\{x_0\}}\leq \frac{N\,\tilde\tau_{K,N}(\sfd_{x_0})-1}{\sfd_{x_0}}\,\mm\qquad\forall x_0\in X.
\end{equation}
\end{corollary}
\begin{proof}  \eqref{eq:distq} is a particular case of \eqref{eq:comparison}. To get \eqref{eq:dist} notice that since $(X,\sfd)$ is geodesic, it holds $\lip^-(\sfd_{x_0})=1$ on $X\setminus\{x_0\}$, so that by Theorem \ref{thm:cheeger} we get $\weakgrad{\sfd_{x_0}}=1$ $\mm$-a.e. and from the chain rule \eqref{eq:chaineasy} that $\weakgrad{(\frac{\sfd_{x_0}^2}2)}=\sfd_{x_0}$ $\mm$-a.e.. The conclusion comes from the chain rule of Proposition \ref{prop:chainlapl} with $\varphi(z):=\sqrt{2z}$.
\end{proof}
Notice that in proving the Laplacian comparison estimates we showed that
\[
\eqref{eq:comparison}\qquad\Rightarrow\qquad\eqref{eq:distq}
\]
In a smooth situation, also the converse implication holds. It is not so clear if the same holds in the abstract framework, because we didn't prove a property like
\[
g_i\in D(\bd),\quad \bd g_i\leq \mu,\qquad\forall i\in I,\qquad\qquad\Rightarrow\qquad\qquad \inf_ig_i\in D(\bd),\quad\bd(\inf_ig_i)\leq\mu.
\]
\begin{remark}[Laplacian comparison from the $MCP(K,N)$ condition]{\rm
If one is only interested in \eqref{eq:distq} and not in \eqref{eq:comparison}, then the Measure Contraction Property $MCP(K,N)$ (in conjunction with infinitesimal strict convexity) is sufficient to conclude. The proof is exactly the same: just interpolate from an arbitrary measure to $\delta_{x_0}$ using the Markov kernel given by the definition of $MCP(K,N)$. 
}\fr\end{remark}
\begin{remark}[Laplacian comparison from the $\CD^*(K,N)$ condition]{\rm
In \cite{Bacher-Sturm10} Bacher and Sturm introduced a variant of the $\CD(K,N)$ condition, called reduced curvature dimension condition and denoted by $\CD^*(K,N)$, where the distortion coefficients $\tau_{K,N}^{(t)}(\theta)$ are replaced by the coefficients $\sigma^{(t)}_{K,N}(\theta)$ defined by
\[
\sigma^{(t)}_{K,N}(\theta):=\left\{
\begin{array}{ll}
+\infty,&\qquad\textrm{ if }K\theta^2\geq N\pi^2,\\
\frac{\sin(t\theta\sqrt{K/N})}{\sin(\theta\sqrt{K/N})}&\qquad\textrm{ if }0<K\theta^2 <N\pi^2,\\
t&\qquad\textrm{ if }K\theta^2=0,\\
\frac{\sinh(t\theta\sqrt{K/N})}{\sinh(\theta\sqrt{K/N})}&\qquad\textrm{ if }K\theta^2 <0.
\end{array}
\right.
\]
Notice that $\tau^{(t)}_{K,N}(\theta)=t^{\frac1N}(\sigma^{(t)}_{K,N-1}(\theta))^{1-\frac1N}$. The motivation for the introduction of this notion comes from the study of the local to global property. Indeed, while the question on whether the local validity of the $\CD(K,N)$ condition implies its global validity is still open, the same question for the $\CD^*(K,N)$ condition has affirmative answer, at least for non-branching spaces. Also, the $\CD^*(K,N)$ condition is stable w.r.t. mGH convergence and it is locally equivalent to the $\CD(K,N)$ condition (all this has been proved in \cite{Bacher-Sturm10}). The drawback of the $\CD^*(K,N)$ condition is  that the geometric and functional inequalities which are its consequence sometime occur with non-sharp constants (although some positive recent development in this direction can be found in \cite{Cavalletti-Sturm}).

Now, if we introduce the functions $\tilde\sigma_{K,N}:[0,\infty)\to \R$ by
\[
\tilde\sigma_{K,N}(\theta):=
\left\{
\begin{array}{ll}
\theta\sqrt{\frac KN}\,{\rm cotan}\Big(\theta\sqrt{\frac{K}{N}}\Big),&\qquad\textrm{if }K>0,\\
\\
1,&\qquad\textrm{if }K=0,\\
\\
\theta\sqrt{\frac KN}\,{\rm cotanh}\Big(\theta\sqrt{\frac{K}{N}}\Big),&\qquad\textrm{if }K<0,
\end{array}
\right.
\]
and we repeat the arguments of Theorem \ref{thm:controllo}, we get that on infinitesimally strictly convex $\CD^*(K,N)$ spaces, locally Lipschitz $c$-concave functions $\varphi$ are in $D(\bd)$ and it holds
\[
\bd\varphi\leq N\,\tilde\sigma_{K,N}(\weakgrad\varphi)\,\mm.
\]
}\fr\end{remark}

As a first step towards the proof of the Cheeger-Gromoll splitting theorem in this abstract framework, we show that the Laplacian comparison estimates for the distance function imply that the Busemann function associated to a 
geodesic ray on $\CD(0,N)$ spaces is subharmonic. 

Recall that a map $\gamma:[0,\infty)\to X$ is called geodesic ray provided for any $T>0$ its restriction to $[0,T]$ is a minimal geodesic. We will always assume that geodesic rays are parametrized by unit speed. Given a geodesic ray, the Busemann function $b$ associated to it is defined by
\[
b(x):=\lim_{t\to+\infty}b_t(x),\qquad\textrm{ where }\qquad b_t(x):=t-\sfd(x,\gamma_t).
\]
Notice that this is indeed a good definition, because the triangle inequality ensures that for any $x\in X$ the map $t\mapsto b_t(x)$ is non-decreasing (and the inequality $\sfd(x,\gamma_t)\geq\sfd(\gamma_0,\gamma_t)-\sfd(x,\gamma_0)$ yields that $b$ is finite everywhere).

Technically speaking, the proof of the subharmonicity of $b$ is not a direct consequence of the estimate \eqref{eq:dist} and of the stability result of Proposition \ref{prop:laplstab}, because we don't really know if $\bd b_{t_n}$ weakly converges to some limit measure for some $t_n\uparrow\infty$ (we miss a bound from below, which is needed to get compactness). Instead, we will prove an a priori regularity result for the Busemann function which ensures that it is in the domain of the Laplacian. This is the content of the following purely metric lemma. It is stated and proved in geodesic spaces, but a simple approximation argument shows that the same holds on length spaces. 
\begin{lemma}[Regularity of the Busemann function]\label{le:regb}
Let $(X,\sfd)$ be a geodesic space, $\gamma:[0,\infty)\to X$ a geodesic ray and $b$ the Busemann function associated to it. Then $-b$ is  $c$-concave.
\end{lemma}
\begin{proof} We know from \eqref{eq:cc} that $(-b)^{cc}\geq -b$, so to conclude we only need to prove the opposite inequality. For any $x\in X$ and $t\geq 0$  let $\gamma^{t,x}:[0,\sfd(x,\gamma_t)]\to X$ be a unit speed geodesic connecting $x$ to $\gamma_t$. Fix $x\in X$ and notice that
\[
(-b)^{cc}(x)=\inf_{y\in X}\sup_{\tilde x\in X}\frac{\sfd^2(x,y)}{2}-\frac{\sfd^2(\tilde x,y)}{2}-b(\tilde x)\leq \sup_{\tilde x\in X}\frac{1}2-\frac{\sfd^2(\tilde x,\gamma^{t,x}_1)}{2}-b_t(\tilde x),\qquad\forall t\geq 1.
\]
Then observe that for any $\tilde x\in X$ and $t\geq 1$ it holds
\[
\begin{split}
\frac{1}2-\frac{\sfd^2(\tilde x,\gamma^{t,x}_1)}{2}-b_t(\tilde x)&=\frac{1}2-\frac{\sfd^2(\tilde x,\gamma^{t,x}_1)}{2}+\sfd(\tilde x,\gamma_t)-t\\
&\leq \frac{1}2-\frac{\sfd^2(\tilde x,\gamma^{t,x}_1)}{2}+\sfd(\tilde x,\gamma^{t,x}_1)+\sfd(\gamma^{t,x}_1,\gamma_t)-t\\
&=-\frac{1}2-\frac{\sfd^2(\tilde x,\gamma^{t,x}_1)}{2}+\sfd(\tilde x,\gamma^{t,x}_1)+\sfd(x,\gamma_t)-t\\
&\leq- b_t(x).
\end{split}
\]
Thus we proved that $(-b)^{cc}(x)\leq -b_t(x)$ for any $t\geq 1$. The conclusion follows by letting $t\to\infty$.
\end{proof}
In the next proposition we are going to assume both that $(X,\sfd,\mm)$ is infinitesimally strictly convex and that $W^{1,2}(\Omega)$ is uniformly convex for any $\Omega\subset X$ open with $\mm(\partial\Omega)=0$: these are two unrelated notions, as the former concerns strict convexity of the norm in the `tangent bundle' while the latter is about uniform convexity of the norm in the `cotangent bundle' (see Appendix \ref{app:cottan} for further details about this wording). Notice that on infinitesimally Hilbertian spaces these assumptions are fulfilled.
\begin{proposition}[Busemann functions are subharmonic in $\CD(0,N)$ spaces]\label{prop:bus}
Let $(X,\sfd,\mm)$ be an infinitesimally strictly convex  $\CD(0,N)$ space such that $W^{1,2}(\overline\Omega,\sfd,\mm)$ is uniformly convex for any $\Omega\subset X$ open such that $\mm(\partial\Omega)=0$. Let $\gamma$ be a geodesic ray and $b$ the Busemann function associated to it. Then 
\[
b\in D(\bd)\qquad\textrm{ and }\qquad\bd b\geq 0.
\]
\end{proposition}
\begin{proof} Being the pointwise limit of 1-Lipschitz functions, $b$ is itself 1-Lipschitz. Hence from Theorem \ref{thm:controllo} and Lemma \ref{le:regb} we deduce $b\in D(\bd)$. 

We claim that $\lip^+(b)\equiv1$. Indeed, fix $x\in X$ and, as in the proof of Lemma \ref{le:regb}, let $\gamma^{t,x}:[0,\sfd(x,\gamma_t)]\to X$ be a unit speed geodesic connecting $x$ to $\gamma_t$. Since $X$ is proper, for every $r>0$ and sequence $t_n\uparrow\infty$, the sequence $n\mapsto\gamma^{x,t_n}_{r}$ has a subsequence, not relabeled, which converges to some $x_r\in X$ such that $\sfd(x,x_r)=r$. By construction, it holds $b_{t_n}(\gamma^{x,t_n}_r)=b_{t_n}(x)+r$, and from the fact that $b_t$ is 1-Lipschitz we have
\[
b_{t_n}(x_r)\geq b_{t_n}(\gamma^{x,t_n}_r)-\sfd(\gamma^{x,t_n}_r,x_r)=b_{t_n}(x)+r-\sfd(\gamma^{x,t_n}_r,x_r),
\]
so that passing to the limit we obtain $\frac{b(x_r)-b(x)}r\geq 1$. Letting $r\downarrow 0$ we get $\lip^+(b)(x)\geq 1$, and   our claim is proved. 

Let $f\in\test X$ be non-negative and $\Omega\subset X$ an open bounded set such that $\mm(\partial\Omega)=0$ and $\supp(f)\subset\Omega$. Consider the functions $b_t\restr{\overline\Omega},\ b\restr{\overline\Omega}$ as functions on the space $(\overline\Omega,\sfd,\mm\restr\Omega)$. 
Since $\overline\Omega$ is compact, $b_t\restr{\overline\Omega}$ uniformly converge to $b\restr{\overline\Omega}$ as $t\to\infty$. Also, we know that $\lip^+(b_t)\equiv 1$ on $X\setminus\{\gamma_t\}$, so that by \eqref{eq:cheeger2} we deduce $\weakgrad{b_t}=1$
 $\mm$-a.e.. Since we previously proved that $\lip^+(b)\equiv 1$, we also have $\weakgrad b=1$ $\mm$-a.e., and therefore, keeping Proposition \ref{prop:srestr} in mind, we get $b_t\restr{\overline\Omega}\to b\restr{\overline\Omega}$ in energy in $W^{1,2}(\overline\Omega,\sfd,\mm\restr{\overline\Omega})$ as $t\to\infty$. The uniform convexity of $W^{1,2}(\overline\Omega,\sfd,\mm)$ then easily  imply  that $b_t\restr{\overline\Omega}\to b\restr{\overline\Omega}$ in $W^{1,2}(\overline\Omega,\sfd,\mm)$ as $t\to\infty$. In particular, Corollary \ref{cor:dfgg} gives
\begin{equation}
\label{eq:fra}
\lim_{t\to+\infty}\int Df(\nabla b_t)\,\d\mm=\int Df(\nabla b)\,\d\mm.
\end{equation}
Since $\Omega$ is bounded, there is $\overline t\geq 0$ such that for $t\geq \overline t$ we have $\sfd(\gamma_t,\Omega)>1$. Hence, from the trivial inequality $\sfd(\gamma_t,x)\geq \sfd(\gamma_t,\Omega)$ valid for any $x\in\Omega$ and  \eqref{eq:dist}, for $t\geq\overline t$ we have 
\begin{equation}
\label{eq:perbus}
b_t\in D(\bd,\Omega)\qquad\textrm{ and }\qquad \bd b_t\restr\Omega\geq -\frac{N-1}{\sfd(\gamma_t,\Omega)}\mm.
\end{equation}
Therefore
\[
\begin{split}
\int f\,\d\bd b&=-\int Df(\nabla b)\,\d\mm=\lim_{t\to+\infty}-\int Df(\nabla b_t)\,\d\mm=\lim_{t\to+\infty}\int f\,\d\bd b_t\restr\Omega\\
&\geq \lim_{t\to+\infty}-\frac{N-1}{\sfd(\gamma_t,\Omega)}\int f\,\d\mm=0,
\end{split}
\]
so that from the arbitrariness of $f$ we get the thesis.
\end{proof}
Notice the slight abuse of notation in the name of Proposition \ref{prop:bus}: we just proved that $\bd b\geq 0$, but we didn't prove any sort of maximum principle, so that we don't really know whether a function with non-negative Laplacian stays below, when restricted to some ball $B$ the harmonic function which attains the same value at the boundary of $B$.

We also remark that it is the lack of the strong maximum principle that is preventing to prove that the Busemann function associated to a \emph{line} is harmonic on $\CD(0,N)$ spaces. This subject will need further investigations in the future\footnote{in the context of non-linear potential theory the strong maximum principle for local sub/super-minimizers of the energy has been proved in \cite{Bjorn-Bjorn11}. In the recent paper \cite{Gigli-Mondino12} it has been shown that on doubling spaces supporting a 2-Poincar\'e inequality, local sub/super-mininimizers of the energy can be characterized as those functions having non-negative/non-positive distributional Laplacian, in line with the smooth case, thus providing the converse implication of the one given in Proposition \ref{prop:divano}.}.

\bigskip

We conclude the chapter with some comments about the splitting theorem in this context. The key steps of its proof in smooth Riemannian manifolds are:
\begin{itemize}
\item[i)] To prove that the Busemann function is superharmonic
\item[ii)] To use the strong maximum principle to get that the Busemann function associated to a line is harmonic
\item[iii)] To use the equality case in the Bochner inequality to deduce that the Busemann function has 0 Hessian
\item[iv)] To recall that the gradient flow of a convex function is a contraction, so that the gradient flow of the Busemann function produces a family of isometries
\item[v)] To conclude that the manifold isometrically factorizes as $\{b=0\}\times\R$
\end{itemize}
Out of all these steps, the key one which fails in a Finsler setting is the fifth\footnote{this is not the only part of the proof of the splitting theorem which fails in a Finsler setting. Another problem is that in general the gradient flow of a convex function does \emph{not} contract the distance, as deeply and carefully explained in \cite{Sturm-Ohta10}. Yet, this is not the crucial point. For instance, in the case $(\R^2,\|\cdot\|,\mathcal L^2)$, where $\|\cdot\|$ is a smooth and uniformly convex norm, the gradient flow of the Busemann function $b$ associated to any geodesic line, provides a family of translations. Thus  in particular it provides a one parameter family of isometries. But due to the local structure of the arbitrary norm $\|\cdot\|$ we can't conclude that it holds $\R^2=\{b=0\}\times\R$ as metric spaces.}. The problem is that on the product space  $\{b=0\}\times\R$ one puts the product distance
\begin{equation}
\label{eq:proddist}
\tilde\sfd\big((x,t),(x',t')\big):=\sqrt{\sfd^2(x,x')+|t-t'|^2},
\end{equation}
which is a Riemannian product in nature.  For example, in the case $(\R^2,\|\cdot\|,\mathcal L^2)$ we certainly have that the measure factorizes, i.e. $\mathcal L^2=\mathcal L^1\times\mathcal L^1$, but \eqref{eq:proddist} in general fails (as far as I know, the validity of the splitting theorem in a Finsler context is an open problem\footnote{after the work on this paper was finished, I got aware of a very recent result by Ohta, who showed that a diffeomorphic and measure preserving splitting holds on smooth Finlser geometries, and that the gradient flow of the Busemann function provides a 1-parameter family of isometries on Berwald spaces, see \cite{Ohta12}.} if one allows for product distances different than the one in \eqref{eq:proddist}). 

Proposition \ref{prop:bus} addresses $(i)$ in the abstract framework, but given that Finsler geometries are included in the class of $\CD(K,N)$ spaces, we know that the splitting theorem fails in this setting. The hope then would be to find an appropriate additional condition which enforces a Riemannian-like behavior on small scales. The natural candidate seems to be infinitesimal Hilbertianity, discussed in Section \ref{se:lineare}: this was one of the research perspectives behind the definition of spaces with Riemannian Ricci curvature bounded from below given in \cite{Ambrosio-Gigli-Savare11bis}. But to understand whether the splitting theorem holds under this additional assumption still requires quite a lot of work\footnote{In the recent paper \cite{Gigli13} it has been indeed proved that the splitting holds on infinitesimally Hilbertian $\CD(0,N)$ spaces. The argument does not make explicit use of the Hessian, but is rather based on the fact that the Busemann function is a pointwise minimizer for the Bochner inequality.}.

\appendix
\section{On the duality between cotangent and tangent spaces}\label{app:cottan}
The topic we try to address here is up to what extent one can use the differential calculus developed in Chapter \ref{se:diff} to provide an abstract notion of cotangent and tangent spaces on arbitrary metric measure spaces. For simplicity, we are going to assume that the reference measure $\mm$ is a probability measure. 
\begin{definition}[Cotangent and tangent spaces]\label{def:cottan}
Let $(X,\sfd,\mm)$ be as in \eqref{eq:mms} with $\mm\in\prob X$, and $p,q\in(1,\infty)$ conjugate exponents. 

The cotangent space $\ctang p(X,\sfd,\mm)$ is defined as the abstract completion of $\s^p(X,\sfd,\mm)/\sim$, where we say that $f\sim g$ provided $\|f-g\|_{\s^p}=0$. Given $f\in\s^p(X,\sfd,\mm)$, we will denote by $Df$ the equivalence class of $f$ in $\ctang p(X,\sfd,\mm)$.

The tangent space $\tang q(X,\sfd,\mm)$ is the dual of $\ctang p(X,\sfd,\mm)$.
\end{definition}
\begin{remark}{\rm
If $(X,\sfd,\mm)$ is a Finsler or Riemannian manifold, then the notion of cotangent space we just provided does not really coincide with the family of $p$-integrable cotangent vector fields on the manifold, but only with those arising as differentials of functions, which is clearly a smaller class. Similarly for the tangent space. 

If one wants a notion which produces all cotangent vector fields, he can proceed with a patching argument, following the idea that every smooth cotangent vector field is, locally, close to the differential of a function.  More precisely, one starts considering the set $V$ defined as
\[
V:=\Big\{\{(A_i,f_i)\}_{i\in \N}\ :\ A_i\cap A_j=\emptyset \textrm{ for }i\neq j,\ X=\cup_{i\in\N}A_i,\ f_i\in \s^p(X,\sfd,\mm),\ \forall i\in\N\Big\}.
\]
Then a sum, a multiplication by a scalar and a seminorm on $V$ are defined as
\[
\begin{split}
\big\{(A_i,f_i)\big\}_{i\in\N}+\big\{(\tilde A_j,\tilde f_j)\big\}_{j\in\N}&:=\big\{(A_i\cap \tilde A_j,f_i+\tilde f_j)\big\}_{i,j\in\N},\\
\lambda\cdot \big\{(A_i,f_i)\big\}_{i\in\N}&:=\big\{(A_i,\lambda f_i)\big\}_{i\in\N},\\
\big\|\big\{(A_i,f_i)\big\}_{i\in\N}\big\|_{\s^p}^p&:=\sum_{i\in\N}\int_{A_i}\weakgrad {f_i}^p\,\d\mm.
\end{split}
\]
It is then easy to check that the quotient of $V$ w.r.t. the subspace of elements with 0 norm is a normed vector space. The abstract completion of this space is a Banach space, which in a smooth situation can be identified with the space of $p$-integrable cotangent vector fields. Its dual can therefore be identified with the space of $q$-integrable tangent vector fields. 

A similar construction can be made on the space $({\rm Grad}^q(X,\sfd,\mm),{\rm Dist}_q)$ introduced below.

The biggest drawback of the definition we proposed is that the cotangent and tangent spaces are just vector spaces, and not moduli over $L^\infty$. This in particular prevents discussing the validity of the Leibniz rule $D(fg)=fDg+gDf$, simply because we didn't define the objects $fDg$ and $gDf$. Some results in this direction are given in \cite{Weaver01} via completely different means. We expect the two presentation to be compatible, but it is not clear to us whether this is really the case or not, partly because the setting is not completely the same: Weaver works with a metric space, a notion of null sets and differential of Lipschitz functions while we work on metric measure spaces and differentials of Sobolev functions. 

Anyway, for the purpose of the discussion we want to make here, we will be satisfied with Definition \ref{def:cottan}.
}\fr\end{remark}
For $L\in\tang q(X,\sfd,\mm)$ and $g\in\s^p(X,\sfd,\mm)$, we will denote by $Dg(L)\in\R$ the dualism between the operator $L$ and the equivalence class $Dg$ of $g$ in $\ctang p(X,\sfd,\mm)$. Notice that it holds
\begin{equation}
\label{eq:dualL}
Dg(L)\leq \|Dg\|_{\ctang p}\|L\|_{\tang q}\leq \frac{\|Dg\|^p_{\ctang p}}{p}+\frac{\|L\|^q_{\tang q}}q,
\end{equation}
for any $g\in\s^p(X,\sfd,\mm),\ L\in \tang q(X,\sfd,\mm)$,
hence in analogy with Definition \ref{def:represent}, we may say that $L$ \emph{$q$-represents} $\nabla g$ provided
\begin{equation}
\label{eq:repL}
Dg(L)\geq  \frac{\|Dg\|^p_{\ctang p}}{p}+\frac{\|L\|^q_{\tang q}}q.
\end{equation}
Also, we say that $L\in\tang q(X,\sfd,\mm)$ is \emph{almost represented} by a $q$-test plan $\ppi\in\prob{C([0,1],X)}$ provided $(\e_0)_\sharp\ppi=\mm$, $\|\ppi\|_q<\infty$ and for any $f\in\s^p(X,\sfd,\mm)$ it holds
\begin{equation}
\label{eq:ar}
\limi_{t\downarrow 0}\int\frac{f(\gamma_t)-f(\gamma_0)}{t}\,\d\ppi(\gamma)\leq Df(L)\leq \lims_{t\downarrow 0}\int\frac{f(\gamma_t)-f(\gamma_0)}{t}\,\d\ppi(\gamma).
\end{equation}
Notice that as a consequence of Theorem \ref{thm:horver}, if $(X,\sfd,\mm)$ is $q$-infinitesimally strictly convex, the the $\limi$ and $\lims$ in \eqref{eq:ar} are actually limits and the equality holds.
\begin{proposition}\label{prop:Lp}
Let $(X,\sfd,\mm)$ be as in \eqref{eq:mms} with $\mm\in\prob X$,  $\ppi\in \prob{C([0,1],X)}$ be a $q$-test plan with $(\e_0)_\sharp\ppi=\mm$ and $\|\ppi\|_q<\infty$. Then there exists at least one $L\in\tang q(X,\sfd,\mm)$ which is almost represented by $\ppi$, and for any such $L$ it holds 
\begin{equation}
\label{eq:Lboundppi}
\|L\|_{\tang q }\leq\|\ppi\|_q.
\end{equation}
Also, if $\ppi\in \prob{C([0,1],X)}$ is a $q$-test plan which $q$-represents $\nabla g$ for some $g\in\s^p(X,\sfd,\mm)$, and $L$ is almost represented by $\ppi$, then $L$ $q$-represents $\nabla g$ as well and $\|L\|_{\tang q}=\|\ppi\|_q$.
\end{proposition}
\begin{proof}
Consider the map
\[
\s^p(X,\sfd,\mm)\ni f\qquad\mapsto\qquad T(f):=\lims_{t\downarrow 0}\int\frac{f(\gamma_t)-f(\gamma_0)}{t}\,\d\ppi(\gamma).
\]
It is clearly convex and positively 1-homogeneous. The trivial inequality $T(f)-T(g)\leq T(f-g)\leq \|f-g\|_{\s^p}\|\ppi\|_q$ yields that $T$ passes to the quotient and defines a convex, positively 1-homogeneous map, still denoted by $T$, on $\ctang p(X,\sfd,\mm)$, and this map is also Lipschitz, with Lipschitz constant $\|\ppi\|_q$. By the Hahn-Banach theorem, there exists a linear map $L:\ctang p(X,\sfd,\mm)\to\R$ such that
\[
Df(L)\leq T(f),\qquad\forall f\in\s^p(X,\sfd,\mm).
\]
Since $T$ is Lipschitz, $L$ is bounded, hence $L\in\tang q(X,\sfd,\mm)$. Exchanging $f$ with $-f$ we conclude that $L$ is almost represented by $\ppi$. The fact that any $L$ which is almost represented by $\ppi$ satisfies \eqref{eq:Lboundppi} is obvious.

The second part of the statement is a direct consequence of the definition and of the inequality \eqref{eq:Lboundppi}.
\end{proof}
With the very same arguments used in the proof of Theorem \ref{thm:horver}, we get that for $f,g\in\s^p(X,\sfd,\mm)$ and $L\in\tang q(X,\sfd,\mm)$  which $q$-represents $\nabla g$   it holds
\begin{equation}
\label{eq:carbonara}
\int D^+f(\nabla g)\weakgrad g^{p-2}\,\d\mm\geq Df(L)\geq \int D^-f(\nabla g)\weakgrad g^{p-2}\,\d\mm.
\end{equation}
The compactness properties of $\tang q(X,\sfd,\mm)$ allow to prove that equalities can actually hold.
\begin{theorem}[Horizontal and vertical derivatives - II]\label{thm:horver2}
Let  $(X,\sfd,\mm)$ be as in \eqref{eq:mms} with $\mm\in\prob X$, $p\in(1,\infty)$ and  $f,g\in\s^p(X,\sfd,\mm)$. Then there exist $L^+,L^-\in \tang q(X,\sfd,\mm)$, depending on both $f$ and $g$, $q$-representing $\nabla g$ such that
\begin{align}
\label{eq:horver2}
\int D^+f(\nabla g)\weakgrad g^{p-2}\,\d\mm&= Df(L^+),\\
\label{eq:horver3}
\int D^-f(\nabla g)\weakgrad g^{p-2}\,\d\mm&= Df(L^-).
\end{align}
\end{theorem}
\begin{proof}
We prove \eqref{eq:horver2} only, as the proof of \eqref{eq:horver3} follows along similar lines. Taking \eqref{eq:carbonara} into account, we need only to find some $L^+$ for which $\leq$ holds. For $\eps>0$, use Theorem \ref{thm:extest} to find a $q$-test plan  $\ppi^\eps$ which $q$-represents $\nabla(g+\eps f)$ and such that 
$(\e_0)_\sharp\ppi^\eps=\mm$, and then Proposition \ref{prop:Lp} to find $L^\eps\in\tang q(X,\sfd,\mm)$  almost represented by $\ppi^\eps$. By \eqref{eq:dualL} and \eqref{eq:repL} we have
\begin{equation}
\label{eq:ovvio}
\begin{split}
D(g+\eps f)(L^\eps)&\geq \frac{\|D(g+\eps f)\|^p_{\ctang p}}{p}+\frac{\|L^\eps\|^q_{\tang q}}q,\\
Dg(L^\eps)&\leq\frac{\|Dg\|^p_{\ctang p}}{p}+\frac{\|L^\eps\|^q_{\tang q}}q. 
\end{split}
\end{equation}
Subtracting and dividing by $\eps>0$ we get
\[
Df(L^\eps)\geq\int \frac{\weakgrad{(g+\eps f)}^p-\weakgrad g^p}{\eps p}\,\d\mm.
\]
As $\eps\downarrow 0$, the right hand side converges to $\int D^+f(\nabla g)\weakgrad g^{p-2}\,\d\mm$. Concerning the left one, notice that the bound 
\[
\|L^\eps\|_{\tang q}\leq\|\ppi^\eps\|_q\leq \|D(g+\eps f)\|_{\ctang p}^{\frac pq}\leq \Big(\|Dg\|_{\ctang p}+\eps\|Df\|_{\ctang p}\Big)^{\frac pq}.
\]
ensures that the family $\{L^\eps\}_{\eps\in(0,1)}$ is weakly$^*$ relatively compact in $\tang q(X,\sfd,\mm)$. Hence for some sequence $\eps_n\downarrow 0$ we have $L^{\eps_n}\weakto^* L^+\in \tang q(X,\sfd,\mm)$, and therefore
\[
Df(L^+)\geq \int D^+f(\nabla g)\weakgrad g^{p-2}\,\d\mm.
\]
It remains to prove that $L^+$ $q$-represents $\nabla g$. But this is obvious, as it is sufficient to pass to the limit in the first inequality in \eqref{eq:ovvio} taking into account that $\|L^+\|_{\tang q}\leq\limi_{n\to\infty}\|L^{\eps_n}\|_{\tang q}$.
\end{proof}
Now assume that $(X,\sfd,\mm)$ is $q$-infinitesimally strictly convex. Then for $g\in\s^p(X,\sfd,\mm)$ and $\ppi$ $q$-test plan which $q$-represents $\nabla g$ and $(\e_0)_\sharp\ppi=\mm$ we have
\[
\lim_{t\to 0}\int\frac{f(\gamma_t)-f(\gamma_0)}t\,\d\ppi(\gamma)=\int Df(\nabla g)\weakgrad g^{p-2}\,\d\mm,\qquad\forall f\in\s^p(X,\sfd,\mm).
\]
In other words, the map
\begin{equation}
\label{eq:perdual}
\s^p(X,\sfd,\mm)\ni f\qquad\mapsto\qquad\lim_{t\to 0}\int\frac{f(\gamma_t)-f(\gamma_0)}t\,\d\ppi(\gamma),
\end{equation}
is a linear bounded functional on $\ctang p(X,\sfd,\mm)$, i.e. an element of $\tang q(X,\sfd,\mm)$. We define the set ${\rm Grad}^q(X,\sfd,\mm)$ as:
\[
\begin{split}
{\rm Grad}^q(X,\sfd,\mm):=\Big\{\ppi\in\prob{C([0,1],X)}\ :\ &\ppi\textrm{ is a }q\textrm{ test plan with }(\e_0)_\sharp\ppi=\mm\\
&\textrm{which } q\textrm{ represents }\nabla g\textrm{ for some }g\in\s^p\Big\}.
\end{split}
\]
Then \eqref{eq:perdual} provides a map $\iota:{\rm Grad}^q(X,\sfd,\mm)\to \tang q(X,\sfd,\mm)$. Notice that as a consequence of $Df(\iota(\ppi))\leq \|f\|_{\s^p}\|\ppi\|_q$ and of $Dg(\iota(\ppi))\geq \|g\|_{\s^p}\|\ppi\|_q$ if $\ppi$ $q$-represents $\nabla g$, we get that $\|\ppi\|_q=\|\iota(\ppi)\|_{\tang q}$ for any $\ppi\in{\rm Grad}^q(X,\sfd,\mm)$. It is for this reason that we called $\|\ppi\|_q$ the $q$-``norm'' of $\ppi$.

\begin{remark}{\rm
We remark that in passing from $\ppi$ to $\iota(\ppi)$ there is a loss of information because plans can be localized (i.e. one can consider the action of a plan on any Borel subset $A\subset X$ with positive mass by looking at the plan $\mm(A)^{-1}\ppi\restr{\e_0^{-1}(A)}$), while the same operation a priori cannot be performed on elements of $\tang q(X,\sfd,\mm)$.

Yet, it is possible to check that at least under the assumptions of Theorem \ref{thm:figata} below, an adequate localization of elements on $\tang q(X,\sfd,\mm)$ is possible. We won't discuss this topic here. 
}\fr\end{remark}

Interestingly enough, if $\ctang p(X,\sfd,\mm)$ is reflexive, the range of $\iota$ is dense in $\tang q(X,\sfd,\mm)$ in the sense expressed by the following theorem. Notice that we are going to assume that $\ctang p(X,\sfd,\mm)$ is smooth: by this we will mean that its norm is differentiable everywhere outside 0. This assumption certainly implies that $(X,\sfd,\mm)$ is $q$-infinitesimally strictly convex, as the latter just means that $\s^p/\sim$ is smooth. It is unclear to us whether the converse implication also holds, as we don't know whether any loss of smoothness can happen in the completion process. We also remark that this result goes in the same direction of Theorem 4.3 in \cite{CheegerKleiner09} obtained by Cheeger and Kleiner, where the span of the set of derivations along curves is shown to be dense in the tangent space under doubling and Poincar\'e assumptions.
\begin{theorem}[Duality between cotangent and tangent spaces]\label{thm:figata}
Let $(X,\sfd,\mm)$ be as in \eqref{eq:mms} with $\mm\in\prob X$ and $p,q\in(1,\infty)$ two conjugate exponents. Assume that $\ctang p(X,\sfd,\mm)$ is  smooth and reflexive. Then for every $L\in\tang q(X,\sfd,\mm)$ there exists a sequence $(\ppi_n)\subset{\rm Grad}^q(X,\sfd,\mm)$ such that
\[
\iota(\ppi^n)\weakto L,\qquad\textrm{ and } \qquad \|\ppi^n\|_q\to \|L\|_{\tang q},
\]
as $n\to\infty$. 
In particular, if $\tang q(X,\sfd,\mm)$ is uniformly convex, then $\iota({\rm Grad}^q(X,\sfd,\mm))$ is dense in $\tang q(X,\sfd,\mm)$ w.r.t. the strong topology.
\end{theorem}
\begin{proof}
Consider the maps $\mathcal N:\tang q(X,\sfd,\mm)\to\R$ and $\mathcal N^*:\ctang p(X,\sfd,\mm)\to\R$ given by
\[
\begin{split}
\mathcal N(L)&:=\frac{\|L\|^q_{\tang q}}q,\qquad\forall L\in\tang q(X,\sfd,\mm),\\
\mathcal N^*(\omega)&:=\frac{\|\omega\|^p_{\ctang p}}p,\qquad\forall\omega\in\ctang p(X,\sfd,\mm).
\end{split}
\]
Clearly, $\mathcal N$ is the Legendre transform of $\mathcal N^*$, and since $\ctang p(X,\sfd,\mm)$ is assumed to be reflexive, also the viceversa holds. Since both functions are locally Lipschitz, their subdifferential is always non-empty, and the assumption that $\ctang p(X,\sfd,\mm)$ is strictly convex means that $\partial^-\mathcal N^*(\omega)$ is single valued for any $\omega\in\ctang p(X,\sfd,\mm)$.

Fix $L\in\tang q(X,\sfd,\mm)$ and choose $\omega\in\partial^-\mathcal N(L)$, so that $L\in\partial^-\mathcal N^*(\omega)$. By definition of $\ctang p(X,\sfd,\mm)$, we can find a sequence  $(f_n)\subset\s^p(X,\sfd,\mm)$  such that $Df_n\to\omega$ in $\ctang p(X,\sfd,\mm)$. For every $n\in\N$, use Theorem \ref{thm:extest} to find  $\ppi_n\in {\rm Grad}^q(X,\sfd,\mm)$ which $q$-represents $\nabla f_n$ so that it holds
\begin{equation}
\label{eq:buffo}
Df_n(\iota(\ppi_n))\geq \frac{\|Df_n\|_{\ctang p}^p}{p}+\frac{\|\iota(\ppi_n)\|^q_{\tang q}}{q}.
\end{equation}
Since $\sup_n\|\iota(\ppi_n)\|_{\tang q}<\infty$, up to pass to a subsequence, not relabeled, we may assume that $\iota(\ppi_n)$ weakly converges to some $\tilde L$ for which it holds $\|\tilde L\|_{\tang q}\leq\limi_{n\to\infty}\|\iota(\ppi_n)\|_{\tang q}$. Passing to the limit in \eqref{eq:buffo} using the strong convergence of $Df_n$ to $\omega$ in $\ctang p(X,\sfd,\mm)$ we obtain
\begin{equation}
\label{eq:buffo2}
\omega(\tilde L)\geq \frac{\|\omega\|_{\ctang p}^p}{p}+\frac{\|\tilde L\|^q_{\tang q}}{q},
\end{equation}
i.e. $\tilde L\in\partial^-\mathcal N^*(\omega)$. Since we know by assumption that $\partial^-\mathcal N^*(\omega)$ is a singleton, it must hold $\tilde L=L$. To conclude we need to show that $\|\iota(\ppi_n)\|_{\tang q}\to \|L\|_{\tang q}$ as $n\to\infty$. This is a consequence of the fact that the inequality in \eqref{eq:buffo2} cannot be strict, thus  it can't be any loss of norm in passing from \eqref{eq:buffo} to \eqref{eq:buffo2}.

For the last part of the statement, just notice that the strict convexity of $\tang q(X,\sfd,\mm)$ implies the smoothness of $\ctang p(X,\sfd,\mm)$, and that its uniform convexity yields that weak convergence plus convergence of norms is equivalent to strong convergence.
\end{proof}
As the proof shows, if $(\s^p/\sim,\|\cdot\|_{\s^p})$ is complete, then \emph{every} element of $\tang q(X,\sfd,\mm)$ can be recovered as derivation along some $\ppi\in{\rm Grad}^q(X,\sfd,\mm)$. 

In any case, the theorem has the following geometric consequence. Consider $\ppi_1,\ppi_2\in{\rm Grad}^q(X,\sfd,\mm)$, put $L:=\iota(\ppi_1)+\iota(\ppi_2)$ and find a sequence $(\tilde\ppi_n)\subset{\rm Grad}^q(X,\sfd,\mm)$ converging to $L$ in the sense of the above theorem. Then the (derivative at time 0 of the) curves in the support of $\tilde\ppi_n$ can be seen as approximate sum of the  (derivative at time 0 of the)  curves in the support of $\ppi_1,\ppi_2$. At least in the sense of their action on differentials. In case of infinitesimally Hilbertian spaces, there is no need of passing through an approximation: if $\ppi_i$ 2-represents $\nabla g_i$, $g_i\in\s^2(X,\sfd,\mm)$, $i=1,2$, then any plan 2-representing $\nabla(g_1+g_2)$ can be seen as the `sum' of $\ppi_1$ and $\ppi_2$ in the sense just described. A deeper geometric meaning to this fact  could be given  in case the open problem mentioned at the end of this appendix has an affirmative answer, see the discussion below.

\bigskip

The set ${\rm Grad}^q(X,\sfd,\mm)$ can be endowed with the following natural semi-distance:
\begin{equation}
\label{eq:distgrad}
{\rm Dist}_q(\ppi^1,\ppi^2):= \lims_{t\downarrow 0}\frac1{t}\sqrt[q]{\int W_q^q\Big((\e_t)_\sharp\ppi^1_x,(\e_t)_\sharp\ppi^2_x\Big)\,\d\mm(x)},
\end{equation}
where $\{\ppi^i_x\}$ is the disintegration of $\ppi^i$ w.r.t. the map $\e_0$, $i=1,2$. Notice that $\|\ppi\|_q={\rm Dist}_q(\ppi,{\mathbf 0})$, where the zero plan ${\mathbf 0}$ is the only plan in ${\rm Grad}^q(X,\sfd,\mm)$ concentrated on constant curves.
\begin{remark}{\rm
To understand why we are saying that this is the natural distance, start noticing that if we were on a Finsler manifold and two tangent vector fields $v^1,v^2$ in $L^q$ were given, then certainly the distance between them would have been $\sqrt[q]{\int \|v^1-v^2\|^q\,\d\mm}$. Now, if our vector fields were represented by smooth families of diffeomorphisms $\Phi^i_t$, $i=1,2$, in the sense that $\frac\d{\d t}\Phi^i_t\restr{t=0}=v^i$, $i=1,2$, then their distance read at the level of the $\Phi^i_t$'s can be written as
\begin{equation}
\label{eq:diffeo}
\lim_{t\downarrow0}\frac1t\sqrt[q]{\int \sfd^q(\Phi^1_t,\Phi^2_t)\,\d\mm}.
\end{equation}
In an abstract metric measure space we don't have any a priori notion of tangent vector field. But as already discussed various times along the paper, a general idea behind our discussion is to try to define tangent vector fields as  `derivative at time 0 of appropriate curves'. In practice, it seems more convenient to work with weighted family of curves, i.e. with probability measures in $C([0,1],X)$, rather than with curves themselves, which lead us to the concept of `plan representing a gradient'. If they are given $\ppi^1,\ppi^2\in{\rm Grad}^q(X,\sfd,\mm)$ with the property that $\ppi^i_x$ is concentrated on a single curve, say $\gamma^{i}_x$, for $\mm$-a.e. $x$ and $i=1,2$, then by analogy with \eqref{eq:diffeo} their distance in the tangent space should be defined as
\[
\lims_{t\downarrow0}\frac1t\sqrt[q]{\int \sfd^q\big((\gamma^{1}_x)_t,(\gamma^{2}_x)_t\big)\,\d\mm},
\]
the $\lim$ being substituted with a $\lims$ because of a lack of an a priori notion of smoothness. However, for a general plan in ${\rm Grad}^q(X,\sfd,\mm)$ we don't know whether its disintegration is made of deltas or not, so that we don't really have a single point $(\gamma_x)_t$ which is the position at time $t$ of the particle which originally was in $x$. What we have is the probability distribution $(\e_t)_\sharp\ppi_x$. Thus we have to substitute the expression $\sfd^q((\gamma^{1}_x)_t,(\gamma^{2}_x)_t)$ with its analogous  $W_q^q\Big((\e_t)_\sharp\ppi^1_x,(\e_t)_\sharp\ppi^1_x\Big)$. This operation leads to formula \eqref{eq:distgrad}.
}\fr\end{remark}
If our interpretation of plans in ${\rm Grad}^q(X,\sfd,\mm)$ as tangent vector fields is correct, we expect the distance between $\iota(\ppi^1)$ and $\iota(\ppi^2)$ on the tangent space to be bounded by ${\rm Dist}^q(\ppi^1,\ppi^2)$, because ${\rm Dist}^q$ is what controls the infinitesimal relative behavior of $\ppi^1$ and $\ppi^2$. In the following proposition we show that this is actually the case if $\ctang p(X,\sfd,\mm)$ is uniformly convex. Recall that given $f:X\to \R$ the function $\overline{|Df|}:X\to [0,\infty]$ is defined as 0 if $x$ is isolated, otherwise as
\[
\overline{|Df|}(x):=\inf_{r>0}\sup_{y_1\neq y_2\in B_r(x)}\frac{|f(y_1)-f(y_2)|}{\sfd(y_1,y_2)}.
\]
\begin{proposition}\label{prop:treno}
Let $(X,\sfd,\mm)$ be as in \eqref{eq:mms}  with $\mm\in\prob X$ and $p,q\in(1,\infty)$ two conjugate exponents. Assume that $\ctang p(X,\sfd,\mm)$ is uniformly convex. Let $\ppi^1,\ppi^2\in{\rm Grad}^q(X,\sfd,\mm)$. Then for every $f\in\s^p(X,\sfd,\mm)$ it holds
\begin{equation}
\label{eq:noncambia}
\lims_{t\downarrow0}\left|\int\frac{f(\gamma_t)-f(\gamma_0)}{t}\,\d\ppi^1(\gamma)-\int\frac{f(\gamma_t)-f(\gamma_0)}{t}\,\d\ppi^2(\gamma)\right|\leq  \|Df\|_{\ctang p}{\rm Dist}_q(\ppi^1,\ppi^2).
\end{equation}
In particular, if ${\rm Dist}_q(\ppi^1,\ppi^2)=0$ and $\ppi^1$ $q$-represents $\nabla g$ for some $g\in\s^p(X,\sfd,\mm)$, then  $\ppi^2$ $q$-represents $\nabla g$ as well. 
\end{proposition}
\begin{proof}
Let $f:X\to\R$ be a Lipschitz function and for every $x\in X$ define $G_x:[0,\infty)\to\R$ as 
\[
G_x(r):=\left\{
\begin{array}{ll}
\sup\limits_{y_1\neq y_2\in B_r(x)}\dfrac{|f(y_1)-f(y_2)|}{\sfd(y_1,y_2)},&\qquad\textrm{ if }r>0,\ \textrm{ and }B_r(x)\neq\{x\},\\
&\\
\overline{|Df|}(x),&\qquad\textrm{ if }r=0,\ \textrm{ or }B_r(x)=\{x\}.
\end{array}
\right.
\]
By definition, $G_x(\cdot)$ is continuous at 0, and it holds $|f(y_1)-f(y_2)|\leq \sfd(y_1,y_2) G_x(\max\{\sfd(y_1,x),\sfd(y_2,x)\})$ for any $x,y_1,y_2\in X$.

Let $\{\ppi^i_x\}$ be the disintegration of $\ppi^i$ w.r.t. $\e_0$, $i=1,2$ and, for any $t>0$,  let $\ggamma_{t,x}$ be an optimal plan from $(\e_0)_\sharp\ppi^1_x$ to $(\e_0)_\sharp\ppi^2_x$. We have
\[
\begin{split}
&\left|\int f(\gamma_t)-f(\gamma_0)\,\d\ppi^1(\gamma)-\int f(\gamma_t)-f(\gamma_0)\,\d\ppi^2(\gamma)\right|\\
&=\left|\int\left( \int f(\gamma_t)\,\d\ppi^1_x(\gamma)-\int f(\gamma_t)\,\d\ppi^2_x(\gamma)\right)\,\d\mm(x)\right| \\
&=\left|\iint f(y_1)-f(y_2)\,\d\ggamma_{t,x}(y_1,y_2)\,\d\mm(x)\right|\\
&\leq \iint G_x\big(\max\{\sfd(y_1,x),\sfd(y_2,x)\}\big)\sfd(y_1,y_2)\,\d\ggamma_{t,x}(y_1,y_2)\,\d\mm(x)\\
&\leq \sqrt[p]{\iint\Big( G_x\big(\max\{\sfd(y_1,x),\sfd(y_2,x)\}\big)\Big)^p\,\d\ggamma_{t,x}(y_1,y_2)\,\d\mm(x) }\sqrt[q]{\int W_q^q\big((\e_t)_\sharp\ppi^1_x,(\e_t)_\sharp\ppi^2_x\big)\,\d\mm(x)}.
\end{split}
\]
Dividing by $t$, letting $t\downarrow 0$ and taking into account that $\ggamma_{t,x}$ converges to $\delta_x\times\delta_x$ in duality with $C_b(X^2)$ we get
\begin{equation}
\label{eq:conG}
\lims_{t\downarrow0}\left|\int\frac{f(\gamma_t)-f(\gamma_0)}{t}\,\d\ppi^1(\gamma)-\int\frac{f(\gamma_t)-f(\gamma_0)}{t}\,\d\ppi^2(\gamma)\right|\leq  \big\|\overline{|Df|}\big\|_{L^p}\,{\rm Dist}_q(\ppi^1,\ppi^2).
\end{equation}
Pick now $f\in \s^p(X,\sfd,\mm)\cap L^\infty(X,\mm)\subset W^{1,p}(X,\sfd,\mm)$ and use Theorem \ref{thm:lipdense} and Corollary \ref{cor:lipdense} to find a sequence $(f_n)$ of Lipschitz functions such that $\overline{|Df_n|}\to \weakgrad f$ in $L^p(X,\mm)$ and  $\|f-f_n\|_{\s^p}\to 0$ as $n\to\infty$. Then observe that for a general $\ppi\in{\rm Grad}^q(X,\sfd,\mm)$ it holds
\begin{equation}
\label{eq:stanco}
\begin{split}
\lims_{t\downarrow0}&\left|\int\frac{f(\gamma_t)-f(\gamma_0)}{t}\,\d\ppi(\gamma)-\int\frac{f_n(\gamma_t)-f_n(\gamma_0)}{t}\,\d\ppi(\gamma)\right|\\
&=\lims_{t\downarrow0}\left|\int\frac{(f-f_n)(\gamma_t)-(f-f_n)(\gamma_0)}t\,\d\ppi(\gamma)\right|\leq \|f-f_n\|_{\s^p}\|\ppi\|_q,
\end{split}
\end{equation}
so that taking \eqref{eq:conG} into account,  \eqref{eq:noncambia} follows.

For the case of general $f\in\s^p(X,\sfd,\mm)$, just truncate $f$ defining $f_N:=\min\{\max\{f,-N\},N\}\in\s^p(X,\sfd,\mm)\cap L^\infty(X,\mm)$, observe that $\|f-f_N\|_{\s^p}\to 0$ as $N\to\infty$ and conclude as in \eqref{eq:stanco} 
\end{proof}
Notice that for $q$-infinitesimally strictly convex spaces, \eqref{eq:noncambia} reads as
\begin{equation}
\label{eq:elalatra}
\|\iota(\ppi^1)-\iota(\ppi^2)\|_{\tang q}\leq {\rm Dist}_q(\ppi^1,\ppi^2),\qquad\forall \ppi^1,\ppi^2\in{\rm Grad}^q,
\end{equation}
and that if both $\tang q(X,\sfd,\mm)$ and $\ctang p(X,\sfd,\mm)$ are uniformly convex, then Theorem \ref{thm:figata} and Proposition \ref{prop:treno} ensure that the map $\iota$ can be continuously extended to the abstract completion $\overline{{\rm Grad}^q}(X,\sfd,\mm)$ of $({\rm Grad}^q(X,\sfd,\mm),{\rm Dist}_q)$, and this extension provides a 1-Lipschitz map from $\overline{{\rm Grad}^q}$ to $\tang q$. 

It is not clear if this map is a surjection. This is related to the following  open problem: does  equality hold in \eqref{eq:elalatra}?. The answer is expected (at least to me) to be affirmative, because `gradient vector fields should be completely identified by their actions on differential vector fields'. The question is open also on infinitesimally Hilbertian spaces with $p=q=2$. 

Notice that a positive answer would greatly increase the knowledge of the local geometry of the space. For instance, it would tell that for any $g\in\s^p(X,\sfd,\mm)$ and $\mm$-a.e. $x$ there is only one curve, up to equivalence, which realizes the gradient of $g$ at $x$ in the sense that $\gamma_0=x$ and 
\[
\limi_{t\to 0}\frac{g(\gamma_t)-g(\gamma_0)}{t}\geq \frac{\weakgrad g^p(x)}p+\lims_{t\downarrow 0}\frac1{qt}\int_0^t|\dot\gamma_s|^q\,\d s,
\]
where we are saying that two curves $\gamma^1,\gamma^2$ with $\gamma^1_0=\gamma^2_0=x$ are equivalent provided they are tangent at 0, i.e.
\[
\lims_{t\downarrow0}\frac{\sfd(\gamma^1_t,\gamma^2_t)}{t}=0.
\]

\section{Remarks about the definition of the Sobolev classes}\label{app:sob}
The approach we proposed in Section \ref{se:sobcla} to introduce the Sobolev classes is slightly different from the one used in \cite{Ambrosio-Gigli-Savare11}, \cite{Ambrosio-Gigli-Savare-pq}. Here we collect the simple arguments which show that indeed the two approaches are the same.

As in the rest of the paper, $(X,\sfd)$ is a complete and separable metric space, and $\mm$ is a Radon  non-negative measure on $X$. In  \cite{Ambrosio-Gigli-Savare-pq} (generalizing the approach proposed in \cite{Ambrosio-Gigli-Savare11} for the case $p=q=2$) it has been first given the definition of $q$-test plan, as we did in Definition \ref{def:testplan}, then it is introduced the concept of  $p$-negligible set of curves  as:
\begin{definition}[$p$-negligible set of curves]
Let $\Gamma\subset C([0,1],X)$ be a Borel set. We say that $\Gamma$ is $p$-negligible provided 
\[
\ppi(\Gamma)=0,\qquad{\rm{ for\ any }}\ q{\rm{-test\ plan }}\ \ppi,
\]
where $p,q\in(1,\infty)$ are conjugate exponents. A property which holds for any curve, except for those in a $p$-negligible set is said to hold for $p$-almost every curve. 
\end{definition}
Notice that $C([0,1],X)\setminus AC^q([0,1],X)$ is $p$-negligible.

Then the definition of functions Sobolev along $p$-a.e. curve and of $p$-weak upper gradients were proposed as follows:
\begin{definition}[Functions which are Sobolev along $p$-a.e. curve]\label{def:sobae}
Let $f:X\to\R$ be a Borel function. We say that $f$ is Sobolev along $p$-a.e. curve provided for $p$-a.e. $\gamma$ the function $f\circ\gamma$ coincides a.e. in $[0,1]$ and in $\{0,1\}$ with an absolutely continuous map $f_\gamma:[0,1]\to\R$.
\end{definition}
\begin{definition}[$p$-weak upper gradients]\label{def:soborigi}
Let $f$ be a function Sobolev along $p$-a.e. curve. A Borel function $G:X\to[0,\infty]$ is a $p$-weak upper gradient provided
\begin{equation}
\label{eq:sobae}
\big|f(\gamma_1)-f(\gamma_0)\big|\leq \int_0^1 G(\gamma_s)|\dot\gamma_s|\,\d s,\qquad {\rm{for }}\ p{\rm{-a.e.}}\ \gamma.
\end{equation}
\end{definition}
\begin{theorem}
Let $(X,\sfd,\mm)$ be as in \eqref{eq:mms},  $p,q\in[1,\infty]$ two conjugate exponents and $f:X\to\R$, $G:X\to[0,\infty]$ Borel functions, with $G\in L^p(X,\mm)$. Then the following are equivalent.
\begin{itemize}
\item[i)] $f\in \s^p(X,\sfd,\mm)$ and $G$ is a $p$-weak upper gradient of $f$ in the sense of Definition \ref{def:sobcl}.
\item[ii)] $f$ is Sobolev along $p$-a.e. curve and $G$ is a $p$-weak upper gradient of $f$ in the sense of Definitions \ref{def:sobae}, \ref{def:soborigi}.
\end{itemize}
\end{theorem}
\begin{proof}\\*
\noindent $\mathbf{(ii)\Rightarrow(i)}$   Fix a $q$-test plan $\ppi$ and notice that by definition of $p$-negligible set of curves it holds
\[
\big|f(\gamma_1)-f(\gamma_0)\big|\leq \int_0^1 G(\gamma_s)|\dot\gamma_s|\,\d s,\qquad {\rm{for }}\ \ppi{\rm{-a.e.}}\ \gamma.
\]
Integrating w.r.t. $\ppi$ we get
\[
\int\big|f(\gamma_1)-f(\gamma_0)\big|\,\d\ppi(\gamma)\leq \iint_0^1 G(\gamma_s)|\dot\gamma_s|\,\d s\,\d\ppi(\gamma),
\]
which is the thesis.\\

\noindent $\mathbf{(i)\Rightarrow(ii)}$ Fix a $q$-test plan $\ppi$. Arguing as for   \eqref{eq:defsobpunt} we have  that
\begin{equation}
\label{eq:commonroom}
\forall t<s\in[0,1]\quad\textrm{it holds}\qquad |f(\gamma_s)-f(\gamma_t)|\leq \int_t^sG(\gamma_r)|\dot\gamma_r|\,\d r,\qquad\ppi-a.e.\ \gamma,
\end{equation}
hence from Fubini's theorem we get
\begin{equation}
\label{eq:pezzo}
\textrm{ for }\ppi-a.e.\ \gamma \quad\textrm{it holds}\qquad |f(\gamma_s)-f(\gamma_t)|\leq \int_t^sG (\gamma_r)|\dot\gamma_r|\,\d r,\qquad \mathcal L^2-a.e.\ t<s\in[0,1].
\end{equation}
Now notice that the assumptions $G\in L^p(X,\mm)$, $\int E_{q,1}\,\d\ppi<\infty$ and $(\e_t)_\sharp\ppi\leq C\mm$ for any $t\in[0,1]$ and some $C\in\R$ easily yield that the function $\gamma\mapsto \int_0^1G(\gamma_r)|\dot\gamma_r|\,\d r$ is in $L^1(\ppi)$. In particular, for $\ppi$-a.e. $\gamma$ the map $t\mapsto G(\gamma_t)|\dot\gamma_t|$ is in $L^1(0,1)$. Therefore from \eqref{eq:pezzo} and Lemma \ref{le:luigi} below we deduce that for $\ppi$-a.e. $\gamma$ the function $t\mapsto f(\gamma_t)$ coincides a.e. in $[0,1]$ with an absolutely continuous map $f_{\gamma,\sppi}$, whose derivative is bounded - in modulus - by $G(\gamma_t)|\dot\gamma_t|$.

By construction we have
\begin{equation}
\label{eq:perzero}
f_{\gamma,\sppi}(t)=f(\gamma_t),\qquad\ppi\times\mathcal L^1\restr{[0,1]}-a.e.\ (\gamma,t),
\end{equation}
and we need to prove that for $\ppi$-a.e. $\gamma$  it holds  $f_{\gamma,\sppi}(0)=f(\gamma_0)$  and  $f_{\gamma,\sppi}(1)=f(\gamma_1)$. Using \eqref{eq:commonroom} with $t=0$ and \eqref{eq:perzero}  we can find a sequence $s_n\downarrow 0$ such that $f_{\gamma,\sppi}(s_n)=f(\gamma_{s_n})$ for $\ppi$-a.e. $\gamma$ and every $ n\in\N$ and
\[
|f_{\gamma,\sppi}(s_n)-f(\gamma_0)|\leq \int_0^{s_n}G(\gamma_r)|\dot\gamma_r|\,\d r,\qquad\ppi-a.e.\ \gamma,\ \forall n\in\N.
\]
Passing to the limit as $n\to\infty$ we get that for $\ppi$-a.e. $\gamma$ it holds $f_{\gamma,\sppi}(0)=\lim_nf_{\gamma,\sppi}(s_n)=f(\gamma_0)$. A similar argument works for $t=1$.

It remains to prove that the functions $f_{\gamma,\sppi}$ do not depend on $\ppi$, but only on $\gamma$. But this is obvious, because the map $t\mapsto f(\gamma_t)$ has at most one continuous representative.
\end{proof}
The following lemma was proved in \cite{Ambrosio-Gigli-Savare-pq}, the proof is reported here for completeness.
\begin{lemma}\label{le:luigi}
Let $f:(0,1)\to\R$, $q\in (1,\infty)$, $g\in L^q(0,1)$ nonnegative be satisfying
$$
|f(s)-f(t)|\leq\left|\int_s^t g(r)\,\d r\right|\qquad\text{for $\mathcal L^{2}$-a.e. $(s,t)\in (0,1)^2$}.
$$
Then $f\in W^{1,q}(0,1)$ and $|f'|\leq g$ a.e. in $(0,1)$.
\end{lemma}
\begin{proof} 
Let $N\subset (0,1)^2$ be the $\mathcal L^{2}$-negligible subset where the above 
inequality fails. By Fubini's theorem, also the set $\{(t,h)\in (0,1)^2:\ (t,t+h)\in N\cap (0,1)^2\}$
is $\mathcal L^{2}$-negligible. In particular, by Fubini's theorem, for a.e. $h$ we have
$(t,t+h)\notin N$ for a.e. $t\in (0,1)$. Let $h_i\downarrow 0$ with this property and use
the identities
$$
\int_0^1f(t)\frac{\phi(t+h)-\phi(t)}{h}\,\d t=-\int_0^1\frac{f(t-h)-f(t)}{-h}\phi(t)\,\d t
$$
with $\phi\in C^1_c(0,1)$ and $h=h_i$ sufficiently small to get
$$
\biggl|\int_0^1f(t)\phi'(t)\,\d t\biggr|\leq\int_0^1g(t)|\phi(t)|\,\d t.
$$
By $L^q$ duality, this gives the $W^{1,q}(0,1)$ regularity and, at the same time, the inequality
$|f'|\leq g$ a.e. in $(0,1)$. 
\end{proof}

\def\cprime{$'$}

\end{document}